\documentclass[leqno,onefignum]{siamltex1213}%final

\usepackage{amssymb,amsmath} %,amsthm}

\newcounter{tcnt}[section]

\newtheorem{assumption}[tcnt]{Assumption}
\newtheorem{remark}[tcnt]{Remark}

\title{Optimal Control of a Semidiscrete Cahn-Hilliard-Navier-Stokes System with Non-Matched Fluid Densities
\thanks{This research was supported by the German Research Foundation DFG 
through the SPP 1506 and the
SFB-TRR 154 and by the Research Center {\sc MATHEON} through project 
C-SE5 and D-OT1 funded by the Einstein Center for Mathematics Berlin.}
}

\author{Michael Hinterm\"uller, Tobias Keil, Donat Wegner
\thanks{Institute for Mathematics, Humboldt-Universit\"at zu Berlin, Unter den Linden 6, 10099 Berlin, Germany.}
}
\begin{document}
\maketitle
\slugger{mms}{xxxx}{xx}{x}{x--x}%slugger should be set to mms, siap, sicomp, sicon, sidma, sima, simax, sinum, siopt, sisc, or sirev

\begin{abstract}
This paper is concerned with the distributed optimal control of a
time-discrete Cahn--Hilliard/Navier--Stokes system with variable
densities.
It focuses on the double-obstacle potential which yields an optimal
control problem for a family of coupled systems in each time instance of a
variational inequality of fourth order and the Navier--Stokes equation.
By proposing a suitable time-discretization, energy estimates are proved
and the existence of solutions to the primal system and of optimal
controls is established for the original problem as well as for a family
of regularized problems. The latter correspond to Moreau--Yosida type
approximations of the double-obstacle potential. The consistency of these
approximations is shown and first order optimality conditions for the
regularized problems are derived. Through a limit process, a stationarity
system for the original problem is established which is related to a
function space version of C-stationarity.
\end{abstract}


%49K20   	Problems involving partial differential equations
%90C46   	Optimality conditions, duality
%35J87   	Nonlinear elliptic unilateral problems and nonlinear elliptic variational inequalities
%35Q35   	PDEs in connection with fluid mechanics
%35Q30   	Navier-Stokes equations
%%35Q93   	PDEs in connection with control and optimization
%76D55   	Flow control and optimization
%76D05   	Navier-Stokes equations
%76D03   	Existence, uniqueness, and regularity theory
%76T99   	None of the above, but in this section
%76T10   	Liquid-gas two-phase flows, bubbly flows

\pagestyle{myheadings}
\thispagestyle{plain}
\markboth{Optimal control of a semidiscrete Cahn-Hilliard-Navier-Stokes system}{M. Hinterm\"uller, T. Keil, D. Wegner}
% \begin{document} % % % %
% % \begin{titlepage}
% \maketitle 
% \end{titlepage}
% Optimal control of the semi-discrete CH-NS system
\section{Introduction}
%\input{introembed2}
%\noindent
% The flow of two or more immiscible fluids has been traditionally described by sharp interface models.
% In the recent past, a new approach received an increasing amount of attention.
% Phase field models were successfully established to deal in a natural way with the theoretical difficulties such as topological changes,
% e.g. droplet break-ups or the coalescence of interfaces,
% and to facilitate the implementation of numerical simulations.
% Herein an order parameter depicts the concentration of the fluids,
% attaining determined values at the pure phases and bridging the intermediate interval in a thin interfacial layer.
In this paper we are concerned with the optimal control of two (or more) 
immiscible fluids
with non-matched densities. For the mathematical formulation of the 
fluid phases, we use
phase field models which have recently been used successfully in 
applications involving, e.g.,
phase separation phenomena (see, e.g., \cite{Aland2013, Eckert2013, Kim2004a}).
% \cite{Eckert2013}, or the simulation of bubble dynamics, as in Taylor flows \cite{Aland2013}, or pinch-offs of liquid-liquid jets \cite{Kim2004},
% to the formation of polymeric membranes \cite{Zhou2006} or proteins
% crystallization, see e.g. \cite{Kim2004a}
%phase separation phenomena.
Some of the strengths of phase field 
approaches are due to their
ability to overcome both, analytical difficulties of topological 
changes, such as, e.g., droplet
break-ups or the coalescence of interfaces, and numerical challenges in 
capturing the interface
dynamics.
In this context, a so-called order parameter depicts the concentration 
of the
fluids, attaining extreme values at the pure phases and intermediate 
values within a thin
(diffuse) interface layer, and it is associated with 
decreasing/minimizing a suitably
chosen energy.
% In this context,  a so-called order parameter depicts the 
% concentration of the fluids,
% attaining extreme values at the pure phases and intermediate values 
% within a thin (diffuse)
% interface layer, seeks to minimize a suitably chosen energy.

A renowned diffuse interface model is the Cahn-Hilliard system which was first introduced by Cahn and Hilliard in \cite{Cahn1958}.
In the presence of hydrodynamic effects, the system has to be enhanced by an equation which captures the behavior of the fluid.
In \cite{Hohenberg1977}, Hohenberg and Halperin published a first basic model for immiscible, viscous two-phase flows.
Their so-called 'model H' combines the Cahn-Hilliard system with the Navier-Stokes equation.
It is however restricted to the case where the two fluids possess nearly identical densities, i.e., matched densities.
% 2. model formulation
Recently, Abels, Garcke and Gr\"un \cite{Abels2012} obtained the following diffuse interface model for two-phase flows with non-matched densities:
% \begin{subequations}\label{adjoint_system}
% \begin{align}
% \overline{\mu}_i = (\overline{\mu}_i - \tilde{c}g_i(\overline{\bm\varphi}))_+,\,\,i
% = 1,\dots,4, \tilde{c} &> 0,\label{adjoint_system_1}\\
% A(\overline{\bm \varphi})^*\overline{p} &=
% (J^{\epsilon}_{\gamma})'_{u}(\overline{u},\overline{\bm\varphi}),\label{adjoint_system_2}\\
% \langle (J^{\epsilon}_\gamma)'_{\bm\varphi}(\overline{u},\overline{\bm\varphi}) -
% (E'_{\bm \varphi}(\overline{u},\overline{\bm\varphi}))^*\overline{p} + g'_{\bm
% \varphi}(\overline{\bm \varphi})^*\overline{\mu},\bm \varphi - \overline{\bm
% \varphi} \rangle &
%  %g'_{\bm \varphi}(\overline{\bm \varphi})^*\overline{\mu}
%  %(\sum_{i}\overline{\mu}_i(g'_{\bm \varphi}(\overline{\bm \varphi})(\bm \varphi -
% \overline{\bm \varphi}))_i
%  \ge 0,\,\,\forall \bm \varphi \in \mathcal{G} \cap U_{c},\label{adjoint_system_3}
% \end{align}
% \end{subequations}
\begin{subequations}\label{CHNS}
\begin{align}
%\begin{aligned}
\partial_t\varphi
+v\nabla\varphi
-\textnormal{div}(m(\varphi)\nabla\mu)
&=0,\label{CHNS3}\\ %\text{ a.e. }\Omega\label{A2}\\
-\Delta\varphi
+\partial\Psi_0(\varphi) %\text{ a.e. }\Omega,\label{A3}
-\mu
-\kappa\varphi
&=0,\label{CHNS4}\\
\partial_t(\rho(\varphi) v)+\textnormal{div}( v\otimes\rho(\varphi) v)
-\textnormal{div}(2\eta(\varphi)\epsilon(v))
+\nabla p&\nonumber\\
+\textnormal{div}(v\otimes J)
-\mu\nabla\varphi
&=0,\label{CHNS1}\\
\textnormal{div}v
&=0,\label{CHNS2}\\
v_{|\partial\Omega}
&=0,\label{CHNS5}\\
\partial_n\varphi_{|\partial\Omega}
=\partial_n\mu_{|\partial\Omega}
&=0,\label{CHNS6}\\
(v,\varphi)_{|t=0}
&=(v_a,\varphi_a),\label{CHNS7}
%\end{aligned}\tag{CHNS}
\end{align}
\end{subequations}
which is supposed to hold in the space-time cylinder $\Omega\times(0,\infty)$, where $\partial\Omega$ denotes the boundary of $\Omega$.
This system is thermodynamically consistent in the sense that it allows for the derivation of local entropy or free energy inequalities.

%we consider the optimal control problem with respect to a semi-discrete version in time of the system (\ref{CHNS}) .\\
In the above model, $v$ represents the velocity of the fluid and $p$ describes the fluid pressure. %$\rho=\rho(\varphi)=\frac{\rho_1+\rho_2}{2}+\frac{\rho_2-\rho_1}{2}\varphi;\ \varphi\in [-1;1]$
% and $\mu$ the related chemical potential.
% The viscosity and mobility coefficients are denoted by $\eta$ and $m$, respectively,
The symmetric gradient of $v$ is defined by $\epsilon(v):=\frac{1}{2}(\nabla v+\nabla v^\top)$.
The density $\rho$ of the mixture of the fluids depends on the order parameter $\varphi$ which reflects the mass concentration of the fluid phases.
More precisely,
\begin{align}
\rho(\varphi)=\frac{\rho_1+\rho_2}{2}+\frac{\rho_2-\rho_1}{2}\varphi, \label{rhophi}
\end{align}
where $\varphi$ ranges in the interval $[-1,1]$,
and $0<\rho_1\leq\rho_2$ are the given densities of the two fluids under consideration.
This is one of the main distinctions compared to the model with matched densities, where $\rho$ is a fixed constant.
The quantity $\mu$ denotes the chemical potential in the Cahn-Hilliard system and helps to split the fourth-order in space differential operator into two second-order operators.
Another important difference between (\ref{CHNS})  and model 'H' is the presence of a relative flux
$J:=-\frac{\rho_2-\rho_1}{2}m(\varphi)\nabla\mu$
which corresponds to the diffusion of the two phases and additionally complicates the analytical situation.
%and demands a different approach than for the matched density case.
The viscosity and mobility coefficients of the system, $\eta$ and $m$, depend on the actual concentration of the two fluids at each point in time and space.
The initial states are given by $v_a$ and $\varphi_a$, and $\kappa>0$ is a positive constant.
% For this reason, the new system demands for a different approach than for the matched density case
% when it comes to the derivation of existence results, etc..\\
%which reflects the concentration of each fluid.
% The chemical potential is denoted by $\mu$ and $\eta$, $m$ are the viscosity and
% %$m$ the
% mobility coefficients of the system.
% The symmetric gradient of $v$ is defined by $\epsilon(v):=\frac{1}{2}(\nabla v+\nabla v^\top)$.
% and %$J$ is given by
%where $0<\rho_1\leq\rho_2$ are the given densities of the two fluids under consideration. % in the pure state.  
% 3. potential choices
Furthermore, $\Psi_0$ represents the convex part of the homogeneous free energy density contained in the Ginzburg-Landau energy model
which is associated with the Cahn-Hilliard part of (\ref{CHNS}). % cf. \cite{ginzburg}.\\
Usually, the homogeneous free energy density serves the purpose of restricting the order parameter $\varphi$ to the physically meaningful range $[-1,1]$ %on the one hand
and to capture the spinodal decomposition of the phases.
For this reason, it is typically non-convex and maintains two local minima near or at $-1$ and $1$.

Depending on the underlying applications, different choices have been investigated in the literature.
In their original paper \cite{Cahn1958}, Cahn and Hilliard considered the logarithmic form
$\Psi(\varphi)=(1+\varphi)\ln (1+\varphi)+(1-\varphi)\ln (1-\varphi) -\frac{\kappa}{2}\varphi^2 $
which also plays an important role in the Flory-Huggins solution theory of the thermodynamics of polymer solutions.
%Typically, the various versions of Ψ aim at confining the values
%of y to [−1, 1] or (−1, 1). In this context,
Another possible choice is the smooth double-well
potential $\Psi(\varphi)=\frac{\kappa}{2}(1-\varphi^2)^2$, see e.g. \cite{Copetti1990,Elliott1986}.
It permits pure phases but fails to restrict the order parameter to $[-1,1]$. Therefore, it is perhaps a less relevant choice in material science.
% Another possibility which lacks the differentiability properties of the previous potentials
% is the double-obstacle potential.
%which enjoys high smoothness properties.
%Another choice is of logarithmic form and goes back to Cahn
%and Hilliard’s original work [21]; see also [2]. Logarithmic forms of the free energy
%density are also important in the Flory–Huggins solution theory of the thermodynamics of polymer solutions.
%While the double-well type free energy allows violations
%of y(x) ∈ [−1, 1], the logarithmic potential does not.
%Both choices, however, share
%certain differentiability properties such that ∂Ψ becomes single-valued and the second
%equation in (1.1) becomes an equality with the derivative Ψ on the right-hand side.
%On the other hand,
In \cite{Oono1988}, Oono and Puri found that in the case of deep quenches of,
e.g. binary alloys, the double-obstacle potential,
is better suited than the other free energy models mentioned above.
A similar observation appears to be true in the case of polymeric membrane formation
under rapid wall hardening.
The double-obstacle potential $\Psi(\varphi)=I_{[-1,1]}(\varphi) -\frac{\kappa}{2}\varphi^2 $, with $I_{[-1,1]}$ denoting the indicator function of the interval $[-1,1]$ in $\mathbb{R}$,
combines the advantage of the existence of pure phases and the exclusiveness of the interval $[-1,1]$
at the cost of losing differentiability (when compared, e.g., to the double-well potential).
% due to the nondifferentiability the fifth relation
As a consequence, %the fourth equation in the system (\ref{CHNS})
(\ref{CHNS4})
becomes a variational inequality
which complicates the analytical and numerical treatment of the overall model.

%For the resulting Cahn–Hilliard system, a comprehensive mathematical analysis can be found in \cite{bloweyelliot}.
%Concerning numerical solvers we refer the reader to [8, 9, 16, 33, 34, 40] and the references therein.
In this paper we study the optimal control of a time discrete coupled Cahn-Hilliard-Navier-Stokes (CHNS) system
with the double-obstacle potential.
% 4. optimal control problem formulation
For this purpose, we introduce a distributed control $u$ which enters the Navier-Stokes equation (\ref{CHNS1}) % first equation of the system (\ref{CHNS})
on the right-hand side
and aims to minimize an objective functional $\mathcal{J}$ subject to the control-version CHNS($u$) of the Cahn-Hilliard-Navier-Stokes system:
\begin{align*}
%\begin{aligned}
&\text{minimize } \mathcal{J}(\varphi,\mu,v,u) \text{ over } (\varphi,\mu,v,u)\\
&\text{subject to (s.t.) }
%H_m^2(\Omega)^{M} \times H_m^2(\Omega)^{M}\times H^{1}_{0,\sigma}(\Omega;\mathbb{R}^N)^{M-1}\times L^2(\Omega;\mathbb{R}^N)^{M-1}\\
u\in U_{ad},\ (\varphi,\mu,v,u)\text{ satisfies CHNS($u$)},
%\end{aligned}
\end{align*}
where $U_{ad}$ is a given set of admissible controls.

% 5. applications
Regarding physical applications, we point out that the CHNS system is used to model a variety of situations.
%The resulting coupled system is used to model
These range from the aforementioned solidification process of liquid
metal alloys, cf. \cite{Eckert2013}, or the simulation of bubble dynamics, as in Taylor flows \cite{Aland2013}, or pinch-offs of liquid-liquid jets \cite{Kim2004},
to the formation of polymeric membranes \cite{Zhou2006} or proteins
crystallization, see e.g. \cite{Kim2004a} and references within.
Furthermore, the model can be easily adapted to include the effects of surfactants such as colloid
particles at fluid-fluid interfaces in gels and emulsions used in food, pharmaceutical, cosmetic, or petroleum industries \cite{Aland2012,Praetorius2013}.
% Even the simulation of
% cooling systems in nuclear power plants or applications in computer graphics
% are conducted by using these models [3].
In many of these situations an optimal control context is desirable
where the system is influenced in such a way that a prescribed system behavior needs to be guaranteed.

%control goal is achieved.\\ %, see e.g. \cite{optconappl}.\\
% With respect to applications the study of the above optimization problem is relevant for instance in the formation of polymeric membranes in the context of an
% immersion precipitation process. In this context, a polymer solution is immersed into
% a coagulation bath which contains a nonsolvent. Due to the concentration difference
% between the polymeric solution and the nonsolvent, the polymeric solution decomposes into two phases, a polymer-rich and a polymer-poor one, respectively. It is
% well-known [35] that the performance of the resulting polymer membrane depends
% significantly on its morphology (i.e., the porosity structure), which is the result of the
% phase separation process.
% 6. literature(opt cont.)
In the literature, the classical case of two-phase flows of liquids with matched densities is well investigated, see e.g. \cite{Hohenberg1977}.
When it comes to the modeling of fluids with different densities, then the literature presents %there are
various approaches,
ranging from quasi-incompressible models with non-divergence free velocity fields, see e.g. \cite{Lowengrub1998}, to possibly thermodynamically inconsistent models with solenoidal fluid velocities, cf. \cite{Ding2007}.
We refer to \cite{Boyer2002,Boyer2004,Gal2010} for additional analytical and numerical results for some of these models.
% Additional numerical conclusions with respect to the double-well potential can be also found in \cite{numericdwp}.
In \cite{Abels2013}, Abels, Depner and Garcke derived an existence result for the given system (\ref{CHNS})  with a logarithmic potential,
and in the recent preprint \cite{Garcke2014} system (\ref{CHNS}) with smooth potentials (thus excluding the double-obstacle homogeneous free energy density) is considered
in a fully discrete and an alternative semi-discrete in time setting including numerical simulations. 

The optimal control problem associated to the Cahn-Hilliard-Navier-Stokes system with matched densities and a non-smooth homogeneous free energy density 
(double-obstacle potential) has been previously studied by the first and 
last author of this work in \cite{Hintermueller2014}. We also mention the recent preprint \cite{Frigeri2014a} which treats the control of a nonlocal Cahn-Hilliard-Navier-Stokes 
system in two dimensions.
%The optimal control problem associated to the Cahn-Hilliard-Navier-Stokes system with matched densities has been previously studied by two of the authors in \cite{Hintermueller2014}
%and there exists a recent preprint where a nonlocal Cahn-Hilliard-Navier-Stokes system is controlled in two dimensions, \cite{Frigeri2014a}. 
Apart from these contributions the literature on the optimal control of the coupled CHNS-system with non-matched densities is - to the best of our knowledge - essentially void.
Nevertheless, we mention that there are numerous publications concerning the optimal control of the phase separation process itself, i.e. the distinct Cahn-Hilliard system,
see e.g. \cite{Colli2014,Colli2014a,Hintermueller2011,Hintermueller2012,Wang2000,Yong1991}.

We point out that the presence of a non-smooth homogeneous free energy 
density associated with the underlying Ginzburg-Landau energy in the 
Cahn-Hilliard system gives rise to an optimal control problem for the 
Navier-Stokes system coupled to the Cahn-Hilliard variational 
inequality. In particular, due to the presence of the variational 
inequality constraint, classical constraint qualifications (see, e.g., 
\cite{Zowe1979}) fail which prevents the application of 
Karush-Kuhn-Tucker (KKT) theory in Banach space for the first-order 
characterization of an optimal solution by (Lagrange) multipliers. In 
fact, it is known \cite{Hintermuller2009,Hintermueller2014} that the resulting problem falls into the realm of 
mathematical programs with equilibrium constraints (MPECs) in function space. 
Even in finite dimensions, this problem class is well-known for its 
constraint degeneracy \cite{Luo1996,Outrata1998}.
As a result, stationarity conditions for this problem class 
are no longer unique (in contrast to KKT conditions); compare 
\cite{Hintermuller2009,Hintermuller2014} in function space and, e.g., \cite{Scheel2000} in 
finite dimensions. They rather depend on the underlying problem 
structure and/or on the chosen analytical approach. In this work, we 
utilize a Yosida regularization technique with a subsequent passage to 
the limit with the Yosida parameter in order to derive conditions of 
C-stationarity type. This technique is reminiscent of the one pioneered 
by Barbu in \cite{Barbu1984}, but for different problem classes.

% 7. what do we do
% 8. paper content
The remainder of the paper is organized as follows.
In section 2 we introduce the semi-discrete Cahn-Hilliard-Navier-Stokes system and assign it to the corresponding optimal control problem. 
In section 3 we show the existence of feasible points to the original optimal control problem, as well as to regularized problems.
Section 4 is concerned with the existence of globally optimal solutions, and section 5 deals with the consistency of the chosen regularization technique.
%The main result of this paper is contained i
In section 6 we derive first-order optimality conditions for the regularized problems using a classical result from non-linear optimization theory.
Then a limiting process leads to a stationarity system for the original problem.
The latter is the content of section 7.
%Section 7 finishes the paper by applying the previous results to the case of the double-obstacle potential.

%\input{intro}
%\input{introold}
\section{The semi-discrete CHNS-system and the optimal control problem}
% In order to allow a better mathematical treatment,
% we embed the above CHNS-system in suitable function spaces and consider an implicit discretization in time.\\  %paragraph
As a first step towards the numerical treatment of the underlying Cahn-Hilliard-Navier-Stokes system, we study a semi-discrete (in time) variant.
For our subsequent analysis we start by fixing the associated function spaces and by invoking our working assumptions.

For this purpose, let $\Omega\subset\mathbb{R}^N,N=2,3$, be a bounded domain with smooth boundary $\partial\Omega\in C^2(\Omega)$.
In particular, $\Omega$ satisfies the cone condition, cf. \cite[Chapter IV, 4.3]{Adams2003}.

For $k\in\mathbb{N}$ and $1\leq p\leq \infty$ we introduce the following Sobolev spaces:
\begin{align*}
H^{k}_{0,\sigma}(\Omega;\mathbb{R}^N)&=\left\{f\in H^{k}(\Omega;\mathbb{R}^N)\cap H^{1}_0(\Omega;\mathbb{R}^N):\textnormal{div} f =0,\text{ a.e. on }\Omega\right\},\\ %,\ k\in\mathbb{N},\\
%\end{align*}
%and the Sobolev spaces with zero mean value %
%\begin{align*}
\overline{W}^{k,p}(\Omega) &= %P_m(W^{k,p}(\Omega))=
\left\{f\in W^{k,p}(\Omega):\int_\Omega f dx=0\right\},\\ %,\ k\in\mathbb{N},\ 1\leq p\leq \infty,\\ %
\overline{W}^{k,p}_{\partial_n}(\Omega) &= %P_m(W^{k,p}(\Omega))=
\left\{f\in \overline{W}^{k,p}(\Omega):\partial_n f_{|\partial\Omega}=0\text{ on }{\partial\Omega}\right\}, %,\ k\in\mathbb{N},\ 1\leq p\leq \infty.
\end{align*} % % %
where 'a.e.' stands for 'almost everywhere'.
Here, ${W}^{k,p}(\Omega)$ and ${W}_0^{k,p}(\Omega)$ denote the usual Sobolev space, see \cite{Adams2003}.
For $p=2$, we also write $H^k(\Omega)$ respectively $H^k_0(\Omega)$ instead.
Unless otherwise noted, $(\cdot,\cdot)$ represents the $L^2$-inner product,
$\left\|\cdot\right\|$ the induced norm,
and $\left\langle\cdot,\cdot\right\rangle:=\left\langle\cdot,\cdot\right\rangle_{\overline{H}^{-1},\overline{H}^1}$ the duality pairing between $\overline{H}^1(\Omega)$ and $\overline{H}^{-1}(\Omega)$. % % % %
For a Banach space $W$, we denote by $W^*$ its topological dual, and
$\mathcal{L}(W,W^*)$ defines the space of all linear and continuous operators from $W$ to $W^*$.
In our notation for norms, we do not distinguish between scalar- or vector-valued functions.
The inner product of vectors is denoted by '$\cdot$', the vector product is represented be '$\otimes$' and the tensor product for matrices is written as ':'.
\begin{remark}\label{shifty}
Before we present the semi-discrete system and assuming integrability in time, from (\ref{CHNS3}) we get
%it is important to note that
%the equation for $\partial_t\varphi$ ensures that
\begin{align*}
\int_\Omega\partial_t\varphi dx
&=-\int_\Omega v\nabla\varphi dx
+\int_\Omega\textnormal{div}(m(\varphi)\nabla\mu) dx
% \\
% &=\int_\Omega \textnormal{div}(v)\varphi dx
% -\int_{\partial\Omega} v\varphi\vec{n} dx
% +\int_{\partial\Omega}m(\varphi)\nabla\mu \vec{n}dx\\
% &
=0,
\end{align*}
%due to the conditions on $\mu$ and $v$.
Hence utilizing (\ref{CHNS7}) the integral mean of $\varphi$ satisfies
$$\frac{1}{|\Omega|}\int_\Omega\varphi dx\equiv\frac{1}{|\Omega|}\int_\Omega\varphi_a dx=:\overline{\varphi_a},$$ %\in(-1,1)$$
i.e., it is constant in time.
By assuming $\overline{\varphi_a}\in(-1,1) $, we exclude the uninteresting case $\left|\overline{\varphi_a}\right|=1$. %where $\varphi$ is constant a.e. on $\Omega$.)
This can be achieved by considering the shifted system (\ref{CHNS}), where $\varphi$ is replaced by its projection onto $\overline{L}^2(\Omega)$.
Consequently, we need to work with shifted variables such as, e.g. $m(y+\overline{\varphi_a})$, which we again denote by $m(y)$ in a slight misuse of notation.
%For this reason, we consider a shifted system where we replace $\varphi$ by its projection on $\overline{L}^2(\Omega)$
%and reformulate the respective functions via a shift of variables, e.g. $m_{new}(y):=m_{old}(y+\overline{\varphi_a})$.
\end{remark}

Motivated by physics, %background,
we assume throughout that the mobility and viscosity coefficients are strictly positive as specified in Assumption \ref{assum1} below.
%Furthermore, it is necessary to extend the 
Furthermore, we extend the connection (\ref{rhophi}) between $\varphi$ and $\rho$ to all of $\mathbb{R}$,
as our studies include certain double-well type potentials which allow for values of $\varphi$ outside the physically relevant interval $[-1,1]$.
\begin{assumption}\label{assum1}
\begin{enumerate}
\item The coefficient functions $m,\eta\in C^2(\mathbb{R})$ in (\ref{CHNS1}) and (\ref{CHNS3}) as well as their derivatives up to second order are bounded,
i.e. there exist constants $0<b_1\leq b_2$ such that for every $x\in\mathbb{R}$, it holds that $b_1\leq \min\{m(x),\eta(x)\}$ and
\begin{align*}
\max\{m(x),\eta(x),|m'(x)|,|\eta'(x)|,|m''(x)|,|\eta''(x)|\}\leq b_2.
\end{align*}
\item The initial state satisfies $(v_a,\varphi_a)\in H^{2}_{0,\sigma}(\Omega;\mathbb{R}^N)\times \left(\overline{H}^2_{\partial_n}(\Omega)\cap \mathbb{K}\right)$
where
$$\mathbb{K}:=\left\{v\in \overline{H}^1(\Omega):\psi_1\leq v\leq\psi_2\textnormal{ a.e. in }\Omega\right\},$$
with $-1-\overline{\varphi_a}=:\psi_1<0<\psi_2:=1-\overline{\varphi_a}$.
\item The density $\rho$ depends on the order parameter $\varphi$ via
$$\rho(\varphi)=\max\left\{\frac{\rho_1+\rho_2}{2}+\frac{\rho_2-\rho_1}{2}(\varphi+\overline{\varphi_a}),0\right\}\geq 0.$$ %, where $0<\rho_1\leq\rho_2$. %
%\item Suppose a given time step size $h>0$ and a number of time steps $M$.
\end{enumerate}
\end{assumption}
We note that by Remark \ref{shifty} the pure phases are attained at $x$ when $\varphi(x)=\psi_1$ or $\varphi(x)=\psi_2$,
and the $max$-operator in Assumption \ref{assum1}.3 ensures that the density remains always non-negative.
The latter is necessary to derive appropriate energy estimates.

With these assumptions we now state the semi-discrete Cahn-Hilliard-Navier-Stokes system. % via induction over the time steps.
For the sake of generality, we additionally introduce a distributed force on the right-hand side of the Navier-Stokes equation, which will later serve the purpose of a distributed control.
Below and throughout the paper, $\tau>0$ denotes the time step-size and $M\in\mathbb{N}$ the total number of time instances in the semi-discrete setting.
\begin{definition}[Semi-discrete CHNS-system]\label{defsemidis}
Let %$(v_a,\varphi_{a})\in H^{1}_{0,\sigma}(\Omega;\mathbb{R}^N)\times \overline{H}^2_{\partial_n}(\Omega)$ be given and let
$\Psi_0:\overline{H}^1(\Omega)\rightarrow \mathbb{R}$ be a convex functional with subdifferential $\partial\Psi_0$.
% For given time step size $h>0$ and number of time steps $M$,
Fixing $(\varphi _{-1}, v _{0})=(\varphi _a , v _a) $ we say that a triple
\begin{align*}
(\varphi,\mu,v)=((\varphi_i)_{i=0}^{M-1},(\mu_i)_{i=0}^{M-1},(v_i)_{i=1}^{M-1})
\end{align*}
in $\overline{H}^2_{\partial_n}(\Omega)^{M} \times \overline{H}^2_{\partial_n}(\Omega)^{M}\times H^{1}_{0,\sigma}(\Omega;\mathbb{R}^N)^{M-1}$
solves the semi-discrete CHNS system
with respect to a given control
$u=(u_i)_{i=1}^{M-1}\in L^2(\Omega;\mathbb{R}^N)^{M-1}$, denoted as $(\varphi,\mu,v)\in S_\Psi(u)$,
if %$v_0=v_a$, $\varphi_{-1}=\varphi_a$, and
it holds for all $\phi\in \overline{H}^1(\Omega)$ and $\psi\in H^{1}_{0,\sigma}(\Omega;\mathbb{R}^N)$ that
\begin{align} %
&\left\langle\frac{\varphi_{i+1} -\varphi_{i} }{\tau},\phi\right\rangle
+\left\langle v_{i+1}\nabla\varphi_{i},\phi\right\rangle
-\left\langle\textnormal{div}(m(\varphi_{i})\nabla\mu_{i+1}),\phi\right\rangle=0,\label{firsttim1}\\ %\ \forall \phi\in \overline{H}^1(\Omega),\label{firsttim1}\\
&\hspace*{0.6cm} \left\langle -\Delta\varphi_{i+1},\phi\right\rangle
+\left\langle \partial\Psi_0(\varphi_{i+1}),\phi\right\rangle
-\left\langle \mu_{i+1},\phi\right\rangle
-\left\langle \kappa\varphi_{i},\phi\right\rangle=0,\label{firsttim2}\\ %\ \forall \phi\in \overline{H}^1(\Omega),\label{firsttim2}\\
&\left\langle\frac{\rho(\varphi_{i}) v_{i+1}-\rho(\varphi_{i-1}) v_i}{\tau},\psi\right\rangle_{H^{-1}_{0,\sigma},H^1_{0,\sigma}}
+\left\langle\textnormal{div}(v_{i+1}\otimes \rho(\varphi_{i-1})v_i),\psi\right\rangle_{H^{-1}_{0,\sigma},H^1_{0,\sigma}}\nonumber\\
&-\left\langle\textnormal{div}(v_{i+1}\otimes \frac{\rho_2-\rho_1}{2}m(\varphi_{i-1})\nabla\mu_i),\psi\right\rangle_{H^{-1}_{0,\sigma},H^1_{0,\sigma}}
+(2\eta(\varphi_{i})\epsilon(v_{i+1}),\epsilon (\psi))\nonumber\\ %
&\hspace*{4cm}-\left\langle\mu_{i+1}\nabla\varphi_{i},\psi\right\rangle_{H^{-1}_{0,\sigma},H^1_{0,\sigma}}
=\left\langle u_{i+1},\psi\right\rangle_{H^{-1}_{0,\sigma},H^1_{0,\sigma}}.\label{firsttim3} %\ \forall \psi\in H^{1}_{0,\sigma}(\Omega;\mathbb{R}^N).\label{firsttim3} %
\end{align}
The first two equations are supposed to hold for every $0\leq i+1 \leq M-1$ and 
the last equation holds for every $1\leq i+1 \leq M-1$. %
\end{definition}
% \begin{remark}
% To simplify our discussion, here, we do not include
% the initial values $\varphi _a $ and $ v _a $ as components of $(\varphi ,\mu , v )$. % in the space $ X $.
% Therefore, we use the notational convention that in the sequel all occurrences of $\varphi _{-1} $ and $ v _{0}$
% refer to $\varphi _a $ and $ v _a $, respectively.  % and the unique $\mu _0$ solving~((2)), respectively.
% \end{remark}\mbox{}\\ 
\begin{remark}
In general, the subdifferential of a convex function $\Psi_0$ can be a set-valued mapping, see, e.g., \cite{Ekeland1999}.
In this case, by equation (\ref{firsttim2}) there exists $\beta\in\partial\Psi_0(\varphi_{i+1})$ such that
%For the sake of readability we use the same notation but the reader should always be aware that the true meaning of equation (\ref{firsttim2}) is
%In this case, equation (\ref{firsttim2}) and every upcoming equation containing $\partial\Psi_0$ have to be understood in the following sense 
\begin{align*}
%\exists \beta \in\partial\Psi_0(\varphi_{i+1})
%\forall \phi\in \overline{H}^1(\Omega):
\left\langle -\Delta\varphi_{i+1},\phi\right\rangle
+\left\langle \beta,\phi\right\rangle
-\left\langle \mu_{i+1},\phi\right\rangle
-\left\langle \kappa\varphi_{i},\phi\right\rangle=0,\ \forall \phi\in \overline{H}^1(\Omega).
\end{align*}
%Every other upcoming equation containing $\partial\Psi_0$ has to be interpreted in a similar way.
\end{remark}

We note that in the above system the boundary conditions specified in (\ref{CHNS}) are included in the respective function spaces.

% Furthermore, the first two equations (\ref{firsttim1}) and (\ref{firsttim2}) are independent of the velocity of the old time step
% and we therefore characterized $\varphi_0$ and $\mu_0$ for the first time step by considering only the decoupled Cahn Hilliard system.
% For the later time steps the given discretization maintains the coupling of the Cahn-Hilliard system and the Navier-Stokes equation.
It is interesting to note that our semi-discretization of (\ref{CHNS}) in time involves three time instances $(i-1,i,i+1)$.
Equations (\ref{firsttim1}) and (\ref{firsttim2}), however, do not involve the velocity at the ''old'' time instance $i-1$.
As a consequence, $(\varphi_0,\mu_0)$ are characterized by the  (decoupled) Cahn-Hilliard system only.
At the final time instance, however, the coupling of the Cahn-Hilliard and the Navier-Stokes system is maintained;
otherwise, we have little hope to derive some energy estimates for the system.

Finally, we present the optimal control problem for the semi-discrete CHNS system.
%where we aim to minimize a given functional under the condition that the investigated points satisfy the above system.
%We also include the case of a limited choice for the control variable in our investigations.
For its formulation, 
%\begin{assumption}\label{objadmissset}
%\begin{enumerate}
let $U_{ad}\subset L^2(\Omega;\mathbb{R}^N)^{M-1}$ and $\mathcal{J}:\mathcal{X}\rightarrow\mathbb{R}$ be a Fr\'echet differentiable function,
with 
$$\mathcal{X}:= \overline{H}^1(\Omega)^{M} \times \overline{H}^1(\Omega)^{M}\times H^{1}_{0,\sigma}(\Omega;\mathbb{R}^N)^{M-1}\times L^2(\Omega;\mathbb{R}^N)^{M-1}.$$
Further requirements on $U_{ad}$ and $\mathcal{J}$ are made explicit in connection with the existence result, Theorem 4.1, below.
%and the stationarity theorem, Theorem \ref{T:Mult}, below.
%\end{enumerate}
%\rightarrow 
%to $\mathbb{R}$.\\
%\end{assumption}\mbox{}\\ 
\begin{definition}\label{optprob}
% Let a suitable objective functional
% $\mathcal{J}:L^2(\Omega;\mathbb{R}^N)^{M-1} \times \overline{H}^2(\Omega)^{M} \times \overline{H}^2(\Omega)^{M}\times H^{1}_{0,\sigma}(\Omega;\mathbb{R}^N)^{M-1}\rightarrow \mathbb{R}$
% and $U_{ad}\subsetL^2(\Omega;\mathbb{R}^N)^{M-1}$ be given. % % %
The optimal control problem is given by
\begin{align*}
\begin{aligned}
&\textnormal{min } \mathcal{J}(\varphi,\mu,v,u)\textnormal{ over } (\varphi,\mu,v,u)\in\mathcal{X}\\
%\overline{H}^2(\Omega)^{M} \times \overline{H}^2(\Omega)^{M}\times H^{1}_{0,\sigma}(\Omega;\mathbb{R}^N)^{M-1}\times L^2(\Omega;\mathbb{R}^N)^{M-1}\\
&\textnormal{s.t. }u\in U_{ad},\ (\varphi,\mu,v)\in S_\Psi(u).
\end{aligned} \tag{$P_\Psi$}
\end{align*}
\end{definition}

% Our main concern in this paper is to study the optimal control of this system.
% One example for the objective functional, which is of great importance in the context of applications for the subsequent theory 
% A very important realization of $\mathcal{J}$ in the context of applications are so called tracking-type functionals.
% For now we consider an arbitrary objective functional $\mathcal{J}$.
In many applications, $\mathcal{J}$ is given by a tracking-type functional and $U_{ad}$ by unilateral or bilateral box constraints.
% In the course of this paper, 
% we impose additional assumptions on $\mathcal{J}$ and $U_{ad}$ in order to establish the existence of solutions to ($P_\Psi$) and to derive stationarity conditions.
%Nevertheless, we point out that the developed theory applies to this kind of functionals.

\section{Existence of feasible points}

%As stated in the introduction,
%the main goal of this work is to investigate the optimal control problem if the free energy density corresponds to the double-obstacle potential.
In this section, we prove the existence of feasible points for the optimization problem ($P_\Psi$).
As stated earlier, for deriving stationarity conditions we will later on approximate the double-obstacle potential by a sequence of smooth potentials of double-well type.
Therefore, %we show the existence of solutions to the semi-discrete CHNS system for
we consider here the following two types of free energy densities.
% characterize global solutions via necessary first-order optimality conditions.
% Our first goal in this section is to derive the existence of feasible points for the problem ($P_\Psi$).
% For this purpose, it is necessary to impose some additional assumptions on the data.
% \begin{assumption}\label{assPsi}
% The functional $\Psi_0:\overline{H}^1(\Omega)\rightarrow\mathbb{R}$ is given by $\Psi_0(\varphi):=\int_\Omega \psi_0(\varphi(x)) dx$,
% where $\psi_0$ satisfies $\psi_0(x)=0$ for every $\psi_1:=-1-\overline{\varphi_a}\leq x\leq\psi_2:=1-\overline{\varphi_a}$ % %
% and represents either the convex part of\\
% (a) a double-well type potential,\\ % % %
% (b) or the double-obstacle potential, i.e.
% $\psi_0(x):=\left\{\begin{array}[c]{ll}
% \infty & \text{if } x< \psi_1\\
% 0& \text{if } \psi_1\leq x\leq\psi_2 \\
% \infty & \text{if } x>\psi_2\\
% \end{array} \right.$\\
% with $\psi_1,\psi_2\in\mathbb{R}$ such that $\psi_1<0<\psi_2$.
% Furthermore, $\Psi$ is defined by $\Psi(x):=\int_\Omega\psi_0(x)-\frac{\kappa}{2}x^2dx=\Psi_0(x)-\int_\Omega\frac{\kappa}{2}x^2dx$. %
% \end{assumption}\mbox{}\\ 
\begin{assumption}\label{assPsi}
The functional $\Psi_0:\overline{H}^1(\Omega)\rightarrow\mathbb{R}$ is convex, proper and lower-semicontinuous.
It has one of the two subsequent properties:
\begin{enumerate}
\item Either it is given by $\Psi_0(\varphi):=\int_\Omega \psi_0(\varphi(x)) dx$
where $\psi_0:\mathbb{R}\rightarrow\overline{\mathbb{R}}:=\mathbb{R}\cup\{+\infty\}$ represents the double-obstacle potential,
$$\psi_0(z):=\left\{\begin{array}[c]{ll}
+\infty & \text{if } z< \psi_1,\\
0& \text{if } \psi_1\leq z\leq\psi_2, \\
+\infty & \text{if } z>\psi_2.\\
\end{array} \right.$$
\item Or it originates from a double-well type potential and satisfies:
\begin{enumerate}
\item $\Psi_0$ is Fr\'echet differentiable with $\left\{\Psi_0'(\varphi)\right\}=\partial\Psi_0(\varphi)\subset L^2(\Omega)$
for every $\varphi\in \overline{H}^1(\Omega)$;
\item There exists $B_u\in\mathbb{R}$ such that $\Psi_0(\varphi)\leq B_u$ for every
$\varphi\in\mathbb{K}.$
\end{enumerate}
\end{enumerate}

Additionally, we assume that the functional $\Psi(\varphi):=\Psi_0(\varphi)-\int_\Omega\frac{\kappa}{2}\varphi(x)^2dx,\kappa>0$, is bounded from below by a constant $B_l\in\mathbb{R}$. %
\end{assumption}

We start by studying the semi-discrete CHNS system for a single time step.
%by induction over the time steps.
For this purpose, assume that
the pair $(\tilde{\varphi},\tilde{v})\in\overline{H}^1(\Omega)\times H^{1}_{0,\sigma}(\Omega;\mathbb{R}^N)$ is
%the velocity $\tilde{v}$, the order parameter $\tilde{\varphi}$ and the chemical potential $\tilde{\mu}$ are
given. % for a specific time step.
%($\varphi_{-1}$ further represents the order parameter at the previous time step).
We then show the existence of a point $(\varphi,\mu,v)$ which solves a slightly modified system. % given $(\tilde{\varphi},\tilde{v})$. %(\ref{A2})-(\ref{A1}) for the next time step.
Theorem \ref{feaspoint} collects the results for all time steps via an induction argument.
Finally, Theorem \ref{Linftycon} shows that the modified system equals the original CHNS system %(\ref{firsttim1})-(\ref{firsttim3})
under suitable assumptions.

%Then we show that there exists a point $(\varphi,\mu,v)$ satisfying the equations (\ref{firsttim1})-(\ref{firsttim3}) for the next time step.\\ %
The starting point for our considerations is an energy estimate for the generalized system.
This estimate will be useful to establish the boundedness of the feasible set.
We note that in what follows, $C$, $C_1$ and $C_2$ denote generic constants which may take different values at different occasions.
\begin{lemma}[Energy estimate for a single time step]\label{energyest}
%\mbox{}\\
Let $\tilde{\varphi}\in \overline{H}^1(\Omega)$, %
$\tilde{v}\in H^{1}_{0,\sigma}(\Omega;\mathbb{R}^N)$, $\Theta_v\in (H^{1}_{0,\sigma}(\Omega;\mathbb{R}^N))^*$, $\Theta_\mu,\Theta_\varphi\in \overline{H}^{-1}(\Omega)$,
$\nu\in H^1(\Omega;\mathbb{R}^N)$, $f_0,f_{-1}\in L^2(\Omega)$, $f_0,f_{-1}\geq 0$ % and a skalar $h>0$
be given such that
\begin{align}
\frac{f_0 -f_{-1}}{\tau}+\textnormal{div}\nu=0\textnormal{ a.e. on }\Omega.\label{prereq1}
\end{align}
In case of the double-obstacle potential suppose additionally that
$\tilde{\varphi}\in\mathbb{K}$.

Then, if $(\varphi,\mu,v)\in \overline{H}^1(\Omega) \times \overline{H}^1(\Omega)\times H^{1}_{0,\sigma}(\Omega;\mathbb{R}^N)$
solves the system
\begin{align}
\left\langle\frac{\varphi -\tilde{\varphi} }{\tau},\phi\right\rangle
+\left\langle v\nabla\tilde{\varphi},\phi\right\rangle
-\left\langle\textnormal{div}(m(\tilde{\varphi})\nabla\mu),\phi\right\rangle
=\left\langle \Theta_\mu,\phi\right\rangle,\ \forall \phi\in \overline{H}^1(\Omega),\label{A2}\\
-\left\langle \mu,\phi\right\rangle
-\left\langle \kappa\tilde{\varphi},\phi\right\rangle %
+\left\langle -\Delta\varphi,\phi\right\rangle
+\left\langle \partial\Psi_0(\varphi),\phi\right\rangle
=\left\langle \Theta_\varphi,\phi\right\rangle,\ \forall \phi\in \overline{H}^1(\Omega),\label{A3}\\
\left\langle\frac{f_0 v-f_{-1} \tilde{v}}{\tau},\psi\right\rangle_{H^{-1}_{0,\sigma},H^1_{0,\sigma}}
+\left\langle\textnormal{div}(v\otimes \nu),\psi\right\rangle_{H^{-1}_{0,\sigma},H^1_{0,\sigma}}
+\left(2\eta(\tilde{\varphi})\epsilon(v),\epsilon (\psi)\right)\nonumber\\ %
-\left\langle\mu\nabla\tilde{\varphi},\psi\right\rangle_{H^{-1}_{0,\sigma},H^1_{0,\sigma}}
=\left\langle \Theta_v,\psi\right\rangle_{H^{-1}_{0,\sigma},H^1_{0,\sigma}},\ \forall \psi\in H^{1}_{0,\sigma}(\Omega;\mathbb{R}^N),\label{A1}
\end{align} %
the following energy estimate holds true:
\begin{align}
&\int_\Omega\frac{f_0 \left| v\right|^2}{2}dx
+\int_\Omega\frac{\left| \nabla\varphi\right|^2}{2}dx
+\Psi(\varphi)
+\int_\Omega f_{-1} \frac{\left| v -\tilde{v} \right|^2}{2}dx
+\int_\Omega \frac{\left| \nabla\varphi -\nabla\tilde{\varphi} \right|^2}{2}dx\nonumber\\
&\hspace{1cm}+\tau\int_\Omega2\eta(\tilde{\varphi})\left|\epsilon(v)\right|^2dx
+\tau\int_\Omega m(\tilde{\varphi})\left|\nabla\mu\right|^2dx
+\int_\Omega \kappa\frac{(\varphi-\tilde{\varphi})^2}{2}\nonumber\\
&\hspace{2cm}\leq \int_\Omega\frac{f_{-1} \left| \tilde{v}\right|^2}{2}dx
+\int_\Omega\frac{\left| \nabla\tilde{\varphi}\right|^2}{2}dx
+\Psi(\tilde{\varphi})
+g(\varphi,\mu,v),\label{EE}
\end{align}
where $g$ is defined as
\begin{align}
g(\varphi,\mu,v):=\left\langle \Theta_\mu,\mu\right\rangle
+\left\langle \Theta_\varphi,\frac{\varphi-\tilde{\varphi}}{\tau}\right\rangle
+\left\langle \Theta_v,v\right\rangle_{H^{-1}_{0,\sigma},H^1_{0,\sigma}}.
\end{align}
\end{lemma}
%
%\begin{proof}
\begin{proof}
First, we observe that
\begin{align}
\left( \textnormal{div}(v\otimes \nu),v\right)
&=\left( (\textnormal{div}\nu)v+(\nu\cdot\nabla)v,v\right)\nonumber\\
&=\int_\Omega((\textnormal{div}\nu)\frac{v}{2}+(\nu\cdot\nabla)v)vdx
+\int_\Omega(\textnormal{div}\nu)\frac{v}{2}vdx\nonumber\\
&=\int_\Omega\textnormal{div}\left(\nu\frac{\left|v\right|^2}{2}\right)
+(\textnormal{div}\nu)\frac{\left|v\right|^2}{2}dx
=\int_\Omega(\textnormal{div}\nu)\frac{\left|v\right|^2}{2}dx.\label{helpeq}
\end{align}
Next, one verifies % (cf. \cite[Lemma 4.3]{Abels2013})
\begin{align}
\left(f_0 v-f_{-1} \tilde{v},v\right)
&=\int_\Omega\frac{f_0 \left| v\right|^2}{2}dx
-\int_\Omega\frac{f_{-1} \left| \tilde{v}\right|^2}{2}dx\nonumber\\
&\hspace{1cm}+\int_\Omega\frac{(f_0-f_{-1}) \left| v\right|^2}{2}dx
+\int_\Omega\frac{f_{-1} \left| v-\tilde{v}\right|^2}{2}dx.\label{hoho}
\end{align}
Testing (\ref{A2}),(\ref{A3}) and (\ref{A1}) with $\mu$, $\frac{\varphi -\tilde{\varphi} }{\tau}$ and $v$,  respectively, summing up and integrating by parts, we obtain
\begin{align} %
0&=\int_\Omega\frac{f_0 \left| v\right|^2-f_{-1} \left| \tilde{v}\right|^2}{2\tau}dx
+\int_\Omega f_{-1}\frac{\left| v- \tilde{v}\right|^2}{2\tau}dx
+\int_\Omega\frac{(f_0-f_{-1}) \left| v\right|^2}{2\tau}dx\nonumber\\
&\hspace{1cm}+\int_\Omega(\textnormal{div}\nu)\frac{\left|v\right|^2}{2}dx
+\int_\Omega2\eta(\tilde{\varphi})\left|\epsilon(v)\right|^2dx
+\int_\Omega m(\tilde{\varphi})\left|\nabla\mu\right|^2dx\nonumber\\ %
&\hspace{2cm}+\frac{1}{\tau}\left\langle\partial\Psi_0(\varphi),\varphi -\tilde{\varphi}\right\rangle_{\overline{H}^{-1},\overline{H}^1}
-\kappa\int_\Omega\tilde{\varphi}\frac{\varphi -\tilde{\varphi} }{\tau}dx\nonumber\\
&\hspace{3cm}+\frac{1}{\tau}\int_\Omega\nabla\varphi(\nabla\varphi -\nabla\tilde{\varphi})dx
-g(\varphi,\mu,v),\label{h4*}
\end{align}
where we also use the previous equations (\ref{helpeq}) and (\ref{hoho}).
From the definition of the subdifferential we infer
\begin{align}
\left\langle\partial\Psi_0(\varphi),\varphi -\tilde{\varphi}\right\rangle %_{\overline{H}^{-1},\overline{H}^1}
\geq %\Psi_0(\varphi)-\Psi_0(\tilde{\varphi})\nonumber\\
\Psi(\varphi)-\Psi(\tilde{\varphi})+\frac{\kappa}{2}\int_\Omega\varphi^2-\tilde{\varphi}^2dx.\label{h1}
\end{align}
Inserting (\ref{prereq1}),(\ref{h1}) into (\ref{h4*}) and using  $2a(a-b)=a^2-b^2+(a-b)^2$ once for $(a,b)=(\nabla\varphi,\nabla\tilde{\varphi})$
and then for $(a,b)=(\tilde{\varphi},\varphi)$ we obtain the assertion.
%Additionally, the following equations hold pointwise
% \begin{align}
% \nabla\varphi(\nabla\varphi -\nabla\tilde{\varphi})
% =\frac{\left| \nabla\varphi\right|^2}{2}-\frac{\left| \nabla\tilde{\varphi} \right|^2}{2}+\frac{\left| \nabla\varphi -\nabla\tilde{\varphi} \right|^2}{2}\label{h2}
% \end{align}
% and
% \begin{align}
% \tilde{\varphi}(\varphi -\tilde{\varphi})
% =\frac{\varphi^2}{2}-\frac{\tilde{\varphi}^2}{2}-\frac{\left( \varphi -\tilde{\varphi} \right)^2}{2}.\label{h3}
% \end{align}
% Inserting (\ref{prereq1}),(\ref{h1}),(\ref{h2}) and (\ref{h3}) into equation (\ref{h4*}) proves the assertion.\end{proof}\\ 
\end{proof}
\begin{remark}\label{settin}
Note that the system (\ref{A2})-(\ref{A1}) corresponds to the system (\ref{firsttim1})-(\ref{firsttim3}) for one time step only
when choosing %the following setting
\begin{align*}
%\begin{aligned}
&\tilde{v}=v_i,\
\tilde{\varphi}=\varphi_i,\
f_0
=\rho(\varphi_i),\
f_{-1}
=\rho(\varphi_{i-1}),\\
&\nu
=\rho(\varphi_{i-1})v_i-\frac{\rho_2-\rho_1}{2}m(\varphi_{i-1})\nabla\mu_i,\\
&\Theta_v
=u,\
\Theta_\varphi
=\Theta_\mu
=0.
%\end{aligned}\tag{V}
\end{align*}
\end{remark}
%that the energy estimate (\ref{EE}) is true.\end{proof}\\ 
 % % %
Now we prove the existence of solutions to the system (\ref{A2})-(\ref{A1}).
The proof mainly relies on the application of Schaefer's fixed point theorem, also called the Leray-Schauder principle,
and combines arguments from \cite[Lemma 4.3]{Abels2013} and monotone operator theory.
\begin{theorem}[Existence of solutions to the CHNS system for a single time step]\label{Exist!} %
Let the assumptions of Lemma \ref{energyest} be satisfied.
% :=\left\{v\in \overline{H}^1_{\partial_n}(\Omega):\psi_1\leq v\leq\psi_2\text{ a.e. on }\Omega\right\}$.
Then the system (\ref{A2})-(\ref{A1}) has a solution
$(\varphi,\mu,v)\in \overline{H}^1(\Omega) \times \overline{H}^1(\Omega)\times H^{1}_{0,\sigma}(\Omega;\mathbb{R}^N)$.
\end{theorem}
\begin{proof} We start by defining % introduce the following spaces
\begin{align}
X:=\overline{H}^1(\Omega) \times \overline{H}^1(\Omega)\times H^{1}_{0,\sigma}(\Omega;\mathbb{R}^N),\\
Y:=\overline{H}^{-1}(\Omega) \times \overline{H}^{-1}(\Omega)\times H^{1}_{0,\sigma}(\Omega;\mathbb{R}^N)^*,
\end{align}
and the operators
$\mathcal{G}_1:\overline{H}^1(\Omega)\rightarrow \overline{H}^{-1}(\Omega)$,
$\mathcal{G}_2:\overline{H}^1(\Omega)\rightrightarrows \overline{H}^{-1}(\Omega)$,
$\mathcal{G}_3:H^{1}_{0,\sigma}(\Omega;\mathbb{R}^N)\rightarrow H^{1}_{0,\sigma}(\Omega;\mathbb{R}^N)^*$,
$\mathcal{G}:X\rightrightarrows Y$
and $\mathcal{F}:X\rightarrow Y$ (here and below '$\rightrightarrows$' indicates a set-valued mapping) via
\begin{align*}
\mathcal{G}_1(\mu):=-\textnormal{div}(m(\tilde{\varphi})\nabla\mu)-\Theta_\mu,&
\ \mathcal{G}_2(\varphi):=-\Delta\varphi+\partial\Psi_0(\varphi)-\Theta_\varphi,\\
\mathcal{G}_3(v):=-\textnormal{div}(2\eta(\tilde{\varphi})\epsilon(v))- \Theta_v,&\\
\mathcal{G}(\varphi,\mu,v):=\left(\mathcal{G}_1(\mu),\mathcal{G}_2(\varphi),\mathcal{G}_3(v)\right)^\top,&
\ \mathcal{F}(\varphi,\mu,v):=(\mathcal{F}_1,\mathcal{F}_2,\mathcal{F}_3)^\top,
\end{align*}
with
\begin{align*}
\mathcal{F}_1(\varphi,\mu,v)
&:=-\frac{\varphi -\tilde{\varphi} }{\tau}
-v\nabla\tilde{\varphi},
\ \mathcal{F}_2(\varphi,\mu,v)
:=\mu
+\kappa\tilde{\varphi},\\
\mathcal{F}_3(\varphi,\mu,v)
&:=-\frac{f_0 v-f_{-1} \tilde{v}}{\tau}
-\textnormal{div}(v\otimes \nu)
+\mu\nabla\tilde{\varphi}.
\end{align*} % % %
Using this notation, the system (\ref{A2})-(\ref{A1}) can be stated as %follows
\begin{align}
0\in\mathcal{G}(\varphi,\mu,v)-\mathcal{F}(\varphi,\mu,v)\subset Y.\label{fixeq} %
\end{align}
By standard arguments, the mappings $\mathcal{G}_1$ and $\mathcal{G}_3$ are invertible
and the respective inverse mapping is continuous.
Since the Laplace operator is invertible from $\overline{H}^1(\Omega)$ to $\overline{H}^{-1}(\Omega)$
and the subdifferential $\partial\Psi_0$ is maximal monotone (cf. \cite[Theorem A]{Rockafellar1970}), $\mathcal{G}_2$ is invertible, as well.
Concerning the continuity of $\mathcal{G}_2^{-1}$, let $\xi_1,\xi_2\in \overline{H}^{-1}(\Omega)$ and $\varphi_1,\varphi_2\in \overline{H}^1(\Omega)$ satisfy $\varphi_j=\mathcal{G}_2^{-1}(\xi_j)$ for $j=1,2$.
Using Poincar\'e's inequality and the monotonicity of $\partial\Psi_0$, we immediately obtain
\begin{align*}
\left\|\varphi_2-\varphi_1\right\|^2_{H^1}
%\leq& C\left\langle -\Delta(\varphi_2-\varphi_1),\varphi_2-\varphi_1\right\rangle\\ %_{H^{-1},H^1}\\
&\leq C(\left\langle -\Delta(\varphi_2-\varphi_1),\varphi_2-\varphi_1\right\rangle %_{H^{-1},H^1}\\
+\left\langle \partial\Psi_0(\varphi_2)-\partial\Psi_0(\varphi_1),\varphi_2-\varphi_1\right\rangle)\\ %_{H^{-1},H^1})\\
&= C\left\langle \xi_2-\xi_1,\varphi_2-\varphi_1\right\rangle %_{H^{-1},H^1}\\
\leq C\left\|\xi_2-\xi_1\right\|_{H^{-1}}\left\|\varphi_2-\varphi_1\right\|_{H^1},
\end{align*}
showing the continuity of $\mathcal{G}_2^{-1}$.

Due to the compact embedding of the space %$\overline{Y}$ %into the space  defined by
%\begin{align*}
$\overline{Y}:=L^{\frac{3}{2}}(\Omega) \times L^\frac{3}{2}(\Omega)\times L^{\frac{3}{2}}(\Omega;\mathbb{R}^N)$,
%\end{align*}
into $Y$,
the inverse of $\mathcal{G}$ is a compact operator from $\overline{Y}$ to ${X}$.
Further, $\mathcal{F}:{X}\rightarrow\overline{Y}$ is continuous.
%Next, we check that $\mathcal{F}$ is continuous regarded as a mapping from ${X}$ to $\overline{Y}$.
Hence, the operator $\mathcal{F}\circ\mathcal{G}^{-1}:\overline{Y}\rightarrow\overline{Y}$ is compact.

In what follows, we show the existence of a solution $\delta^*$ to the fixed point equation
%Instead of (\ref{fixeq}), we consider now for $\delta\in{Y}$
\begin{align}
\delta^*-\mathcal{F}\circ\mathcal{G}^{-1}(\delta^*)=0\in\overline{Y}.\label{fixx}
\end{align}
%for $\delta=\mathcal{G}(\varphi,\mu,v)$.
Then it immediately follows that $\mathcal{G}^{-1}(\delta^*)$ solves the system (\ref{A2})-(\ref{A1}).
%Our goal now is to show the existence of such a fixed point $\delta$ in (\ref{fixx}) by applying Schaefer's theorem with respect to the operator $\mathcal{F}\circ\mathcal{G}^{-1}$.
In order to apply Schaefer's theorem with respect to the operator $\mathcal{F}\circ\mathcal{G}^{-1}$ we verify the condition
that the set $D:= \bigcup_{0\leq\lambda\leq1}\left\{\delta\in\overline{Y}|\delta=\lambda\mathcal{F}\circ\mathcal{G}^{-1}(\delta)\right\}$ is bounded.
For this purpose,
assume that $\delta\in\overline{Y}$ and $\lambda\in[0,1]$ satisfy
\begin{align}
\delta=\lambda\mathcal{F}\circ\mathcal{G}^{-1}(\delta), \label{leray}
\end{align}
and define $(\varphi,\mu,v):=\mathcal{G}^{-1}(\delta)\in{X}$.
Thus, (\ref{leray}) can be rewritten as
\begin{align}
\mathcal{G}(\varphi,\mu,v)-\lambda\mathcal{F}(\varphi,\mu,v)=0
\end{align}
which is equivalent to the following system of equations
\begin{align*}
\left\langle\lambda\frac{\varphi -\tilde{\varphi} }{\tau},\phi\right\rangle
+\left\langle\lambda v\nabla\tilde{\varphi},\phi\right\rangle 
=\left\langle\textnormal{div}(m(\tilde{\varphi})\nabla\mu),\phi\right\rangle 
+\left\langle \Theta_\mu,\phi\right\rangle ,\ \forall \phi\in \overline{H}^1(\Omega),\\ %\label{boundmu2}
\left\langle\lambda\mu,\phi\right\rangle 
+\left\langle\lambda\kappa\tilde{\varphi},\phi\right\rangle %
=\left\langle-\Delta\varphi,\phi\right\rangle  %\nonumber\\
+\left\langle\partial\Psi_0(\varphi)),\phi\right\rangle 
-\left\langle \Theta_\varphi,\phi\right\rangle ,\ \forall \phi\in \overline{H}^1(\Omega),\\ %\label{boundmu}
\lambda\left\langle\frac{f_0 v-f_{-1} \tilde{v}}{\tau},\psi\right\rangle_{H^{-1}_{0,\sigma},H^1_{0,\sigma}}
+\lambda\left\langle\textnormal{div}( v\otimes \nu),\psi\right\rangle_{H^{-1}_{0,\sigma},H^1_{0,\sigma}}
+\left(2\eta(\tilde{\varphi})\epsilon(v),\epsilon (\psi)\right)\nonumber\\ %
=\lambda\left\langle\mu\nabla\tilde{\varphi},\psi\right\rangle_{H^{-1}_{0,\sigma},H^1_{0,\sigma}}
+\left\langle \Theta_v,\psi\right\rangle_{H^{-1}_{0,\sigma},H^1_{0,\sigma}} ,\ \forall \psi\in H^{1}_{0,\sigma}(\Omega;\mathbb{R}^N).
\end{align*} %
Analogously to the proof of Lemma \ref{energyest}, we test this system by $\mu$, $\frac{\varphi -\tilde{\varphi} }{\tau}$ and $v$, respectively,
sum up the resulting equations and integrate by parts to derive
\begin{align}
0&=\lambda \int_\Omega\frac{f_0 \left| v\right|^2-f_{-1} \left| \tilde{v}\right|^2}{2\tau}dx
+\lambda \int_\Omega f_{-1}\frac{\left| v- \tilde{v}\right|^2}{2\tau}dx
+\int_\Omega2\eta(\tilde{\varphi})\left|\epsilon(v)\right|^2dx\nonumber\\
&\hspace{1cm}+\int_\Omega m(\tilde{\varphi})\left|\nabla\mu\right|^2dx %
+\frac{1}{\tau}\int_\Omega\partial\Psi_0(\varphi)(\varphi -\tilde{\varphi})dx
-\lambda\kappa\int_\Omega\tilde{\varphi}\frac{\varphi -\tilde{\varphi} }{\tau}dx\nonumber\\
&\hspace{2cm}+\frac{1}{\tau}\int_\Omega\nabla\varphi(\nabla\varphi -\nabla\tilde{\varphi})dx
-g(\varphi,\mu,v),\label{hff1}
\end{align} % % %
which leads to
\begin{align}
\int_\Omega2\eta(\tilde{\varphi})\left|\epsilon(v)\right|^2dx
+\int_\Omega m(\tilde{\varphi})\left|\nabla\mu\right|^2dx %
+\frac{1}{\tau}\Psi(\varphi)
+\frac{1}{\tau}\int_\Omega\left|\nabla\varphi\right|^2dx
-g(\varphi,\mu,v)\nonumber\\
\leq\lambda \int_\Omega\frac{f_{-1} \left| \tilde{v}\right|^2}{2\tau}dx %
+\frac{1}{\tau}\int_\Omega\left|\nabla\tilde{\varphi}\right|^2dx
+\frac{1}{\tau}\Psi(\tilde{\varphi}).\label{hff2} %
\end{align}
Note that for obtaining (\ref{hff1}) we also make use of (\ref{prereq1}).
%With the help of Assumption \ref{assPsi}
The right-hand side of (\ref{hff2}) can be bounded by a constant ${C}:={C(N,\Omega,\tau,f_{-1},\tilde{v},\tilde{\varphi})}>0$ which  is independent of $\lambda$.
Since $\Psi$ is bounded from below,
this leads to
\begin{align}
\int_\Omega2\eta(\tilde{\varphi})\left|\epsilon(v)\right|^2dx
+\int_\Omega m(\tilde{\varphi})\left|\nabla\mu\right|^2dx
+\frac{1}{\tau}\int_\Omega\left|\nabla\varphi\right|^2dx
\leq C+g(\varphi,\mu,v).
\end{align}
Due to Korn's inequality, Poincar\'e's inequality and from the boundedness of $\eta(\cdot)$ and $m(\cdot)$, we infer
\begin{align}
\left\|v\right\|^2_{H^1}
+\left\|\mu\right\|^2_{H^1}
+\left\|\varphi\right\|^2_{H^1}
&\leq C+g(\varphi,\mu,v)\nonumber\\
&\leq C_1+C_2(\left\|v\right\|_{H^1}
+\left\|\mu\right\|_{H^1}
+\left\|\varphi\right\|_{H^1}),
\end{align}
where $C_2>0$ depends only on $\Theta_\mu$, $\Theta_\varphi$ and $\Theta_v$. % we additionally used the boundedness of $g$.
The last inequality yields the boundedness of $(\varphi,\mu,v)$ in $X$. % % % % %
 % % % %
Next, we derive bounds for $\mathcal{F}$.
In fact, we have
\begin{align*}
\left\|\mathcal{F}_1(\varphi,\mu,v)\right\|_{L^{3/2}}
&\leq C( \left\|\varphi\right\|+\left\|\tilde{\varphi}\right\|+\left\|v\right\|_{H^1}\left\|\tilde{\varphi}\right\|_{H^1}),\\
\left\|\mathcal{F}_2(\varphi,\mu,v)\right\|_{L^{{3}/{2}}}
&\leq C(\left\|\mu\right\|+\left\|\tilde{\varphi}\right\|),\\
\left\|\mathcal{F}_3(\varphi,\mu,v)\right\|_{L^{{3}/{2}}}
&\leq C( \left\|v\right\|_{H^1}+\left\|v\right\|_{H^1}\left\|\nu\right\|_{H^1}+\left\|\mu\right\|\left\|\tilde{\varphi}\right\|_{H^1}+\left\|\tilde{v}\right\|_{H^1}).
\end{align*}
Since $\tilde{\varphi}$, $\tilde{v}$ and $\nu$ are fixed, $D$ %$\bigcup_{0\leq\lambda\leq1}\left\{\delta\in\overline{Y}|\delta=\lambda\mathcal{F}\circ\mathcal{G}^{-1}(\delta)\right\}$ %$\delta=\lambda\mathcal{F}(\varphi,\mu,v)$
is bounded in $\overline{Y}$. % and %
Hence Schaefer's theorem is applicable
implying that equation (\ref{fixx}) admits a fixed point $\delta^*\in\overline{Y}$. Then $\mathcal{G}^{-1}(\delta^*)$ solves the system (\ref{A2})-(\ref{A1}).
\end{proof}  % % %

In our setting, the right-hand sides of the system (\ref{A2})-(\ref{A1}) are square integrable functions.
This enables the derivation of higher regularity properties for the solutions obtained in Theorem \ref{Exist!}.
% If the $\Theta_\varphi,\Theta_\mu$ are more regular, then the solutions from Theorem \ref{Exist!} possess additional regularity properties. %
\begin{lemma}[Regularity of solutions]\label{regsol}
Let the assumptions of Lemma \ref{energyest} be satisfied, and suppose additionally that %$\Theta_v\in L^2(\Omega;\mathbb{R}^N)$ and 
% $\tilde\varphi\in H^2(\Omega)$ and
$\Theta_\mu,\Theta_\varphi\in L^{2}(\Omega)$, % and $\Theta_\varphi\in L^2(\Omega)$ for $2\leq \leq 6$.\\
as well as $f_0,f_{-1}\in L^3(\Omega)$ and $\tilde{\varphi}\in H^2(\Omega)$.

Then it holds that $\varphi,\mu\in\overline{H}^2_{\partial_n}(\Omega)$ and $v\in H^2(\Omega;\mathbb{R}^N)$,
provided that $(\varphi,\mu,v)\in \overline{H}^1(\Omega) \times \overline{H}^1(\Omega)\times H^{1}_{0,\sigma}(\Omega;\mathbb{R}^N)$ satisfies the system (\ref{A2})-(\ref{A1}).
% If a point $(\varphi,\mu,v)\in \overline{H}^1(\Omega) \times \overline{H}^1(\Omega)\times H^{1}_{0,\sigma}(\Omega;\mathbb{R}^N)$ satisfies the system (\ref{A2})-(\ref{A1}),
% then
% $\varphi$ and $\mu$ are contained in $\overline{H}^2_{\partial_n}(\Omega)$
% and $v$ is contained in $H^2(\Omega;\mathbb{R}^N)$. % and $v$ in $H^2(\Omega;\mathbb{R}^N)$. % %
Moreover, there exists a constant $C=C(N,\Omega,b_1,b_2,\tau,\kappa)>0$ such that
\begin{align}
&\left\|\varphi\right\|_{H^2}
+\left\|\mu\right\|_{H^2}
+\left\|v\right\|_{H^2} \nonumber\\
&\hspace{1cm}\leq C (\left\|\varphi\right\|+\left\|\mu\right\|+\left\|\tilde{\varphi}\right\|+\left\|\Theta_\varphi\right\|
+\left\|\Theta_\mu\right\|
+\left\|v\right\|_{H^1}\left\|\tilde\varphi\right\|_{H^2}
+\left\|\Psi_0'(\varphi)\right\|).
\end{align}
In case of the double-obstacle potential,
it also holds that $\varphi\in \mathbb{K}$ and the term $\left\|\Psi_0'(\varphi)\right\|$ in the above inequality is dropped. % %
\end{lemma}
\begin{proof} Equation (\ref{A3}) is equivalent to
\begin{align}
\Delta\varphi+g_1\in \partial\Psi_0(\varphi)\text{ in }\overline{H}^{-1}(\Omega)\label{hh22} % %
\end{align}
with $g_1:=\mu+\kappa\tilde{\varphi}+\Theta_\varphi$. %\in L^2(\Omega)$.
By Sobolev`s embedding theorem %$H^1(\Omega)$ is embedded in $L^2(\Omega)$. Hence
$g_1$ is in $L^2(\Omega)$.
In case of the double-well type potential, Assumption \ref{assPsi}.2 (a) then implies $g_2:=-g_1+\Psi_0'(\varphi)\in L^2(\Omega)$ and
% $\Delta\varphi\in L^2(\Omega)$ by Assumption \ref{assPsi}.2 (a).
%\begin{align}
$\Delta\varphi=g_2$. %:=-g_1+\Psi_0'(\varphi),
%\end{align}
%where $g_2\in L^2(\Omega)$.
Applying \cite[Theorem 2.3.6]{Maugeri2000} and \cite[Remark 2.3.7]{Maugeri2000}, we deduce that $\varphi\in \overline{H}^2(\Omega)$ is the unique solution of the Neumann problem
\begin{align*}
\Delta \varphi=g_2 \text{ in }\Omega,\ \partial_n \varphi_{|\partial\Omega}=0 \text{ on }\partial\Omega.
\end{align*}
Furthermore, \cite[Theorem 2.3.1]{Maugeri2000} yields the existence of a constant $C:=C(N,\Omega)$ such that
\begin{align}
\left\|\varphi\right\|_{H^2}
\leq C (\left\|\varphi\right\|+\left\|g_2\right\|)
\leq C (\left\|\varphi\right\|+\left\|\mu\right\|+\kappa\left\|\tilde{\varphi}\right\|+\left\|\Theta_\varphi\right\|
+\left\|\Psi_0'(\varphi)\right\|).
\end{align}
%Consequently, $\varphi\in \overline{H}^2_{\partial_n}(\Omega)$ (cf. \cite[Theorem 3.2.1.3]{gris}).
In case of the double-obstacle potential, (\ref{hh22}) is equivalent to the variational inequality problem: %
\begin{align}
\text{Find }\varphi\in\mathbb{K}:\ \left\langle -\Delta\varphi-g_1,\phi-\varphi\right\rangle\geq 0,\ \forall \phi\in \mathbb{K}.\label{VIDO}
\end{align}
Then the assertion follows from the subsequent lemma.
% Due to the subsequent lemma,the solution to the variational inequality is contained in $\overline{H}^2_{\partial_n}(\Omega)$,
\begin{lemma}\label{kind}
If $\varphi\in\mathbb{K}$ solves the variational inequality problem (\ref{VIDO}) with $g_1\in L^2(\Omega)$, then $\varphi\in\overline{H}^2_{\partial_n}(\Omega)$
and there exists a constant $C=C(N,\Omega)>0$ such that
%\begin{align}
$\left\|\varphi\right\|_{H^2}\leq C \left\|g_1\right\|.$
%\end{align}
\end{lemma}

For the sake of completeness we provide a proof in the appendix.
It closely follows the lines of argumentation of \cite[Chapter IV]{Kinderlehrer2000}.

Regarding $\mu$, we argue similarly.
Indeed, first note that by Sobolev`s embedding theorem and H\"older's inequality
$g_3:=\frac{\varphi -\tilde{\varphi} }{\tau}+v\nabla\tilde{\varphi}-\Theta_\mu-\mu$ is an element of $L^2(\Omega)$.
Furthermore, the coefficient function $m(\tilde\varphi)$ is contained in $H^2(\Omega)$ and $W^{1,6}(\Omega)$, respectively (cf. \cite[II, Lemma A.3]{Kinderlehrer2000}).
Equation (\ref{A2}) is equivalent to
\begin{align}
m(\tilde{\varphi})\Delta\mu
+\nabla (m(\tilde{\varphi}))\nabla \mu
-\mu
%=\textnormal{div}(m(\tilde{\varphi})\nabla\mu)
% = \Theta_\mu
% -\frac{\varphi -\tilde{\varphi} }{\tau}
% - v\nabla\tilde{\varphi}
% -\mu
=g_3
\textnormal{ in }\overline{H}^{-1}(\Omega).
\end{align}
Again by \cite[Theorem 2.3.5]{Maugeri2000} and \cite[Theorem 2.3.1]{Maugeri2000} $\mu\in\overline{H}^2_{\partial_n}(\Omega)$ and it holds that
\begin{align}
\left\|\mu\right\|_{H^{2}}
\leq C (\left\|\mu\right\|+\left\|g_3\right\|)
\leq C (\left\|\mu\right\|+\left\|\varphi\right\|+\left\|\tilde{\varphi}\right\|+\left\|v\right\|_{H^1}\left\|\tilde\varphi\right\|_{H^2}+\left\|\Theta_\mu\right\|),
\end{align}
where $C>0$ depends only on $N,\Omega,b_1,b_2,\tau$.

% % Furthermore, the coefficient function $m(\tilde\varphi)$ is contained in 
% % The additional regularity of $\mu$ comes directly from equation (\ref{A2}) and the fact that
% % $\frac{\varphi -\tilde{\varphi} }{\tau}+v\nabla\tilde{\varphi}-\Theta_\mu$ is an element of $L^2(\Omega)$.\end{proof}\\ 
% % \begin{prop}
% % Let $(\varphi ,\mu , v , u )$ solve the semi-discrete CH-NS system and suppose that $ v _{0}\in H^{2}_{0,\sigma}(\Omega;\mathbb{R}^N)$.
% % Then $ v \in H^2(\Omega ;\mathbb{R}^N )^{M-1}$
% % and $ || v ||_{H^2} \leq C( h ,\varphi _a , v _a , u )$.
% % \end{prop}
% % 
% % We prove the assertion by induction over $i$.
% % The base case $i=0$ follows from the assumption $ \tilde{v}\in H^2(\Omega ;\mathbb{R}^N )$.
% % Next, we take the induction step from $i$ to $i+1$.
Finally, we show the desired regularity of $v$. % via a bootstrap argument.
%We define $\eta _0:=\eta (\tilde{\varphi} )$ and $\eta' _0:=\eta' (\tilde\varphi )$.
Since $\mathop{\rm div}(\epsilon( v  ))=\frac 12(\Delta v  +\nabla (\mathop{\rm div} v  ))$, it holds that
\[
\mathop{\rm div}(2 \eta (\tilde{\varphi} ) \epsilon( v  ) ) = 2 \eta' (\tilde{\varphi} ) \epsilon( v  ) \nabla \tilde\varphi + \eta (\tilde{\varphi} )\Delta v,  
\]
and therefore by equation (\ref{A1}) that
\begin{align}
\Delta v  
=\eta (\tilde{\varphi} )^{-1}\Big[
\mathop{\rm div}( v  \otimes \nu )- 2 \eta' (\tilde{\varphi} ) \epsilon( v  ) \nabla \tilde\varphi 
+ \frac1 \tau ( f_0 v  - f_{-1} \tilde{v} )- \mu  \nabla \tilde\varphi - \Theta_v  \label{P:Regul-1}
\Big]
\end{align}
%with $\nu := \rho (\varphi_{i-1}) \tilde{v} + -\frac{\rho_2-\rho_1}{2}m(\varphi_{i-1}) \nabla \mu _0 $.
Moreover,
$\mathop{\rm div}( v  \otimes \nu ) = (D v  )\nu + v  \mathop{\rm div}\nu .$
%Since $\varphi \in \overline{H}^2_{\partial_n }(\Omega ) $, $\mu \in \overline{H}^2_{\partial_n }(\Omega )$ by~Theorem~\ref{feaspoint},
By the assumptions all summands in the second line of~(\ref{P:Regul-1}) belong to $L^2(\Omega ;\mathbb{R}^N )$ and
$\nu \in H^1(\Omega ;\mathbb{R}^N )$.
Hence, we have
\begin{align}
\Delta v  
=\eta (\tilde{\varphi} )^{-1}\left[ (D v  )\nu + v  \mathop{\rm div}\nu - 2 \eta' (\tilde{\varphi} ) \epsilon( v  ) \nabla \tilde\varphi + f \right],\label{P:Regul-2}
\end{align}
with $ ||f|| \leq C(z)$ for a constant $C(z)>0$
depending only on
$$z=(N,\Omega ,\eta , \tau , ||\varphi ||_{H^2} , ||\tilde\varphi ||_{H^2}, || \mu  ||_{H^2}, || \Theta_v  ||).$$
In order to show $ {v} \in H^2(\Omega;\mathbb{R}^N)$, we apply a bootstrap argument and well-known regularity results
for the stationary Stokes' equation, cf.~\cite{Temam1977}.%[1ex]

%  	with $ ||\widetilde v ||_{H^1}, ||f||_{L^2} \le C(z)$ for a constant $C$ that
% 		only depend on $z=(\Omega ,\eta ,\rho , \tau , ||\varphi _{i} ||_{H^2}, || \mu _{i+1} ||_{H^2}, || u _{i+1} ||_{ U })$.
% 	In order to show $ v _{i} \in H^2$, we apply a bootstrap argument and well-known regularity results
% 		for the stationary Stokes' equation, cf.~\cite{Temam1977}.\\[1ex]
1. Since $ v  \in H^1_{0,\sigma}(\Omega ;\mathbb{R}^N )$, we have that 
$(D v  )\nu $, $\epsilon( v  )\nabla \tilde\varphi $ and $ v  \mathop{\rm div}\nu $ belong to $ L^{3/2}(\Omega ;\mathbb{R}^N )$.
Therefore, \cite[Prop.~2.3, p.~35]{Temam1977} and~(\ref{P:Regul-2}) show that
$ v  \in W^{2,3/2}(\Omega ;\mathbb{R}^N )$
and that
$ || v  ||_{W^{2,3/2}} \leq C(z)$ for a constant $C$ depending only on $z$.

%[1ex]
 %
2. Next, $ v  \in W^{2,3/2}(\Omega ;\mathbb{R}^N )$ and the continuous embedding of $W^{1,3/2}(\Omega)$ into $L^3(\Omega)$ (which we denote by $W^{1,3/2}(\Omega)\hookrightarrow L^3(\Omega)$) imply that
$(D v  )\nu $ and $\epsilon( v  )\nabla \tilde\varphi $ belong to $L^2(\Omega ;\mathbb{R}^N )$.
Moreover, $W^{2,3/2}(\Omega)\hookrightarrow L^p(\Omega)$ for every $p<\infty $.
Hence
$ v  \mathop{\rm div}\nu \in L^{2-\varepsilon }(\Omega ;\mathbb{R}^N )$ for every $\varepsilon >0$.
Applying \cite[Prop.~2.3, p.~35]{Temam1977} again yields $ v  \in W^{2,2-\varepsilon }(\Omega ;\mathbb{R}^N )$ for all $\varepsilon >0$
and $ || v  ||_{W^{2,2-\varepsilon }} \leq C(\varepsilon ,z)$.

%[1ex]
3. Finally, having $ v  \in W^{2,2-\varepsilon }(\Omega ;\mathbb{R}^N )$ and since $W^{2,2-\varepsilon }(\Omega)\hookrightarrow L^\infty(\Omega) $ for $\varepsilon $ sufficiently small,
it follows that also $ v  \mathop{\rm div}\nu $ belongs to $L^2(\Omega ;\mathbb{R}^N )$.
Thus, we arrive at $ v  \in H^2(\Omega ;\mathbb{R}^N )$ and $ || v  ||_{H^2} \leq C(z)$.

%By~Lemma~\ref{energy2}, the components of $z$ can be bounded by $C( h ,\varphi _a , v _a , u )$.
This completes the proof.
\end{proof}

Our aim is to prove the existence of a solution to the semi-discrete CHNS system with the help of Theorem \ref{Exist!}.
This result, however, is not directly applicable with the setting of Remark \ref{settin}, as $f_0,f_{-1}$ and $\nu$ need not satisfy equation (\ref{prereq1}).
% In other words, if $\frac{\rho_1+\rho_2}{2}+\frac{\rho_2-\rho_1}{2}(\varphi+\overline{\varphi_a})\leq 0$ - which can be the case for a double-well type potential -
% the following equation is not satisfied
% if $\varphi$ is far away from the physically relevant interval $\left[\psi_1,\psi_2\right]$ which can be the case
This is due to the nonsmoothness of the density function
and the fact that $\varphi$ may attain values in $\mathbb{R}$ (rather than $[\psi_1,\psi_2]$) for a double-well-type potential.
% Now we want to prove the existence of a point $(\varphi,\mu,v)\in S_\Psi(u)$ for every $u\in U_{ad}$.
% First, we show that there exist solutions to the system (\ref{firsttim1})-(\ref{firsttim3}),
% where equation (\ref{firsttim3}) is replaced by
% \begin{align}
% \left\langle\frac{\rho(\varphi_{i}) v_{i+1}-\rho(\varphi_{i-1}) v_i}{\tau},\psi\right\rangle_{H^{-1}_{0,\sigma},H^1_{0,\sigma}}
% +(2\eta(\varphi_{i})\epsilon(v_{i+1}),\epsilon (\psi))\nonumber\\
% +\left\langle\textnormal{div}(v_{i+1}\otimes \overline{\nu}(v_i,\varphi_i,\varphi_{i-1},\mu_i)),\psi\right\rangle_{H^{-1}_{0,\sigma},H^1_{0,\sigma}} %
% -\left\langle\mu_{i+1}\nabla\varphi_{i},\psi\right\rangle_{H^{-1}_{0,\sigma},H^1_{0,\sigma}}\nonumber\\
% =\left\langle u_{i+1},\psi\right\rangle_{H^{-1}_{0,\sigma},H^1_{0,\sigma}},\ \forall \psi\in H^{1}_{0,\sigma}(\Omega;\mathbb{R}^N),\label{NSwithop} %
% \end{align}
We overcome this difficulty by applying Theorem \ref{Exist!} with the setting
% \begin{align*}
% \begin{aligned}
% f_0
% &=\rho(\tilde{\varphi})\\
% f_{-1}
% &=\rho(\varphi_{-1})\\
% \nu
% &=\overline{\nu}(\tilde{v},\tilde{\varphi},\varphi_{-1},\tilde{\mu})\\
% \Theta_v
% &=u\\
% \Theta_\varphi
% =\Theta_\mu
% &=0,
% \end{aligned}
% \end{align*}
\begin{align}
\begin{aligned}
&\tilde{v}
:=v_i,\ 
\tilde{\varphi}
:=\varphi_i,\ 
f_0
:=\rho(\varphi_i),\ 
f_{-1}
:=\rho(\varphi_{i-1}),\\
&\nu
:=\overline{\nu}(v_i,\varphi_i,\varphi_{i-1},\mu_i),\
\Theta_v
:=u_{i+1},\ 
\Theta_\varphi
:=\Theta_\mu
:=0,
\end{aligned}\label{V}
\end{align}
% and $\overline{\nu}:H^{1}_{0,\sigma}(\Omega;\mathbb{R}^N)\times \overline{H}^1(\Omega)^3\rightarrow H^1(\Omega;\mathbb{R}^N)$
% is an arbitrary operator satisfying
% \begin{align}
% \frac{\rho(\varphi_1) -\rho(\varphi_{2})}{\tau}+\textnormal{div}\overline{\nu}(v,\varphi_1,\varphi_{2},\mu)=0\text{ a.e. on}\Omega.\label{prereq12}
% \end{align}
% Then we apply the result to a specific operator
where $\overline{\nu}:H^{1}_{0,\sigma}(\Omega;\mathbb{R}^N)\times \overline{H}^2(\Omega)^3\rightarrow H^1(\Omega;\mathbb{R}^N)$ is given by
\begin{align}
\overline{\nu}(v,\varphi,\tilde{\varphi},\mu):=\left\{\begin{array}[c]{ll}
\rho(\tilde{\varphi})v-\frac{\rho_2-\rho_1}{2}m(\tilde{\varphi})\nabla\mu & \text{if } \rho(\varphi),\rho(\tilde{\varphi})> 0\text{ a.e. in }\Omega,\\
G(\frac{\rho(\varphi) -\rho(\tilde{\varphi})}{\tau})& \text{else.}
\end{array} \right.\label{vudef}
\end{align}
Here $G:L^2(\Omega)\rightarrow H^1(\Omega;\mathbb{R}^N),\delta\longmapsto \zeta$, is a solution operator to
\begin{align}
-\textnormal{div}\zeta=\delta\text{ a.e. on }\Omega.
\end{align} % %
A specific realization of $G$ is obtained by first solving $-\Delta\xi=\delta$ in $L^2(\Omega)$ with $\xi=0$ on $\partial\Omega$, yielding $\xi\in H^2(\Omega)\cap H^1_0(\Omega)$,
and then setting $\zeta:=\nabla\xi\in H^1(\Omega,\mathbb{R}^N)$.
%A possible choice is to define $\zeta$ as the gradient of the unique solution to the Poisson equation with Dirichlet boundary conditions.\\
We next prove that there always exists a solution to the system (\ref{firsttim1}),(\ref{firsttim2}),(\ref{NSwithop})
where the semi-discrete Navier-Stokes equation (\ref{firsttim3}) is replaced by
\begin{align}
\left\langle\frac{\rho(\varphi_{i}) v_{i+1}-\rho(\varphi_{i-1}) v_i}{\tau},\psi\right\rangle_{H^{-1}_{0,\sigma},H^1_{0,\sigma}}
+(2\eta(\varphi_{i})\epsilon(v_{i+1}),\epsilon (\psi))\nonumber\\
+\left\langle\textnormal{div}(v_{i+1}\otimes \overline{\nu}(v_i,\varphi_i,\varphi_{i-1},\mu_i)),\psi\right\rangle_{H^{-1}_{0,\sigma},H^1_{0,\sigma}} %
-\left\langle\mu_{i+1}\nabla\varphi_{i},\psi\right\rangle_{H^{-1}_{0,\sigma},H^1_{0,\sigma}}\nonumber\\
=\left\langle u_{i+1},\psi\right\rangle_{H^{-1}_{0,\sigma},H^1_{0,\sigma}},\ \forall \psi\in H^{1}_{0,\sigma}(\Omega;\mathbb{R}^N).\label{NSwithop} %
\end{align}
% Afterwards, Theorem \ref{Linftycon} will verify that equation (\ref{NSwithop}) coincides with equation (\ref{firsttim3}) in the cases under consideration.
We point out that (\ref{NSwithop}) coincides with (\ref{firsttim3}) if
\begin{align}
\min\{\rho(\varphi_i),\rho(\varphi_{i-1})\}> 0\text{ a.e. on }\Omega.\label{intcond} %,\ \forall1\leq i\leq M-1
\end{align}
% equation (\ref{NSwithop}) coincides with equation (\ref{firsttim3}).
For the double-obstacle potential this always holds true, since $\varphi_i$ is contained in the interval $[\psi_1,\psi_2]$. % (since $\varphi_i\geq -1-\overline{\varphi_a}\geq \frac{\rho_2}{\rho_1-\rho_2}-\overline{\varphi_a}$
% for every $i=-1,..,M-1$)
Hence in this case $\rho(\varphi_i)\geq\rho(\psi_1)=\rho_1>0$.

In Theorem \ref{Linftycon} we show that for the double-well type potentials under consideration
$\varphi$ remains in a close neighborhood of $[\psi_1,\psi_2]$ and therefore condition (\ref{intcond}) is satisfied as well.
\begin{proposition}[Existence of solution to a modified state system]\label{feaspoint}
%Let a time step size $h>0$ and a number of time steps $M$ % and an inital point
%$\Omega\subset\mathbb{R}^N,N=2,3$ be a smooth bounded domain and
%$(v_a,\varphi_{a})\in H^{2}_{0,\sigma}(\Omega;\mathbb{R}^N)\times \left(\overline{H}^2_{\partial_n}(\Omega)\cap\mathbb{K}\right)$  %$p\in H^1(\Omega)$,
%be given.
Let $\overline{\nu}:H^{1}_{0,\sigma}(\Omega;\mathbb{R}^N)\times \overline{H}^1(\Omega)^3\rightarrow H^1(\Omega;\mathbb{R}^N)$ be defined by (\ref{vudef}). %
Then for every $u\in L^2(\Omega;\mathbb{R}^N)^{M-1}$
there exists a point $(\varphi,\mu,v)\in \overline{H}^2_{\partial_n}(\Omega)^{M} \times \overline{H}^2_{\partial_n}(\Omega)^{M}\times H^{1}_{0,\sigma}(\Omega;\mathbb{R}^N)^{M-1}$
which solves the semi-discrete system (\ref{firsttim1}),(\ref{firsttim2}),(\ref{NSwithop}).
%where the equation (\ref{firsttim3}) is replaced by equation (\ref{NSwithop}).\\
Moreover, every solution $(\varphi,\mu,v)$ satisfies
$(\varphi,\mu,v)\in \overline{H}^2_{\partial_n}(\Omega)^{M} \times \overline{H}^2_{\partial_n}(\Omega)^{M}\times H^{1}_{0,\sigma}(\Omega;\mathbb{R}^N)^{M-1}.$
\end{proposition}
\begin{proof} % We set $v_0=v_a$ and $\varphi_{-1}=\varphi_a$.
Standard arguments % and Lemma \ref{regsol},
guarantee the existence of $(\varphi_0,\mu_0)\in \overline{H}^1(\Omega)\times \overline{H}^1(\Omega)$
such that (\ref{firsttim1})-(\ref{firsttim2}) is satisfied for $i=-1$.
Lemma \ref{regsol} yields $(\varphi_0,\mu_0)\in \overline{H}^2_{\partial_n}(\Omega)\times \overline{H}^2_{\partial_n}(\Omega)$.

If condition (\ref{intcond}) holds true, then Assumption \ref{assum1}.3 and (\ref{firsttim1}) imply
% the definition of $\overline{\nu}$ implies
\begin{align*}
\textnormal{div}\overline{\nu}(v_i,\varphi_i,\varphi_{i-1},\mu_i)
&=\frac{\rho_2-\rho_1}{2}(\nabla\varphi_{i-1}v_i
-\textnormal{div}(m(\varphi_{i-1})\nabla\mu_i))\\
%&=\frac{\rho_2-\rho_1}{2}(-\frac{\varphi_0-\varphi_{-1}}{\tau}).\\
%&=-\frac{\rho(\varphi_0)-\rho(\varphi_{-1})}{\tau}.
&=-\frac{\rho_2-\rho_1}{2}(\frac{\varphi_i-\varphi_{i-1}}{\tau})
=-\frac{\rho(\varphi_i)-\rho(\varphi_{i-1})}{\tau}.
\end{align*}
% due to equation (\ref{firsttim1}).
%Else this is trivially the case.
Consequently, if $(\varphi_i,\mu_i,v_i)$ satisfies (\ref{firsttim1}), then assumption (\ref{prereq1}) is always satisfied by the definition of $\overline{\nu}$ in the sense that
% By the definition of $\overline{\nu}$ it holds that
\begin{align}
\frac{ \rho(\varphi_i)-\rho(\varphi_{i-1})}{\tau}+\textnormal{div}\overline{\nu}(v_i,\varphi_i,\varphi_{i-1},\mu_i)=0\text{ a.e. on }\Omega.
\end{align}
% since by equation (\ref{firsttim1})
% \begin{align*} %
% f_0
% &=\rho(\varphi_0)\\
% f_{-1}
% &=\rho(\varphi_{-1})\\
% \nu
% &=\overline{\nu}(v_0,\varphi_0,\varphi_{-1},\mu_0)\\
% \Theta_v
% &=u\\
% \Theta_\varphi
% =\Theta_\mu
% &=0. %
% \end{align*} % % % % %
Therefore, we can apply Theorem \ref{Exist!} with the setting (\ref{V}) for $i=0$
to guarantee the existence of
$(\varphi_1,\mu_1,v_1)\in \overline{H}^1(\Omega) \times \overline{H}^1(\Omega)\times H^{1}_{0,\sigma}(\Omega;\mathbb{R}^N)$
such that the system (\ref{firsttim1}),(\ref{firsttim2}),(\ref{NSwithop}) is satisfied.

Now Lemma \ref{regsol} yields $(\varphi_1,\mu_1,v_1)\in \overline{H}^2_{\partial_n}(\Omega) \times \overline{H}^2_{\partial_n}(\Omega)\times H^{1}_{0,\sigma}(\Omega;\mathbb{R}^N)$.
In the case of the double-obstacle potential it additionally follows that $\varphi_1\in\mathbb{K}$.

Repeated applications of Theorem \ref{Exist!} and Lemma \ref{regsol} for each time step $i=1,..,M-2$ prove the assertion.\end{proof} 

As discussed above, this theorem directly guarantees the existence of a solution to the semi-discrete CHNS system for the double-obstacle potential.
% Next we address the case of the double-well potential.
% For this purpose
Next we address the boundedness of the solutions which is needed later on to ensure the existence of optimal points for ($P_\Psi$).
For this purpose, we apply the energy estimate of Lemma \ref{energyest} at each time step.
\begin{lemma}[Boundedness of the state]\label{energy2}
%Let the assumptions from Theorem \ref{feaspoint} be fulfilled.
There exists a positive constant $C = C(N,\Omega,b_1,b_2,\tau,\kappa,v_a,\varphi_a,u)>0$
such that for every solution $(\varphi,\mu,v)\in \overline{H}^2_{\partial_n}(\Omega)^{M} \times \overline{H}^2_{\partial_n}(\Omega)^M\times H^{1}_{0,\sigma}(\Omega;\mathbb{R}^N)^{M-1}$
of Theorem \ref{feaspoint}
%to the system (\ref{firsttim1}),(\ref{firsttim2}),(\ref{NSwithop})
it holds that
\begin{align}
\left\|v\right\|^2_{(H^2)^{M}}
+\left\|\mu\right\|^2_{(H^2)^{M}}
+\left\|\varphi\right\|^2_{(H^2)^{M+1}}
\leq  %\left\|v_a\right\|^2_{(H^1)} %
C.
\end{align}
Furthermore, the operator $L^2(\Omega,\mathbb{R}^N)^{M-1}\ni u\longmapsto C(N,\Omega,b_1,b_2,\tau,\kappa,v_a,\varphi_a,u)\in\mathbb{R}$ % mapping from $L^2(\Omega,\mathbb{R}^N)^{M-1}$ into $\mathbb{R}$
is bounded.
\end{lemma}
\begin{proof}
We define the functional $E:\overline{H}^1_{0,\sigma}(\Omega)\times\overline{H}^1(\Omega)\times\overline{H}^1(\Omega)\rightarrow\mathbb{R}$ as follows:
\begin{align}
 E(v,\varphi,\phi):=\int_\Omega\frac{\rho(\phi) \left| v\right|^2}{2}dx
+\int_\Omega\frac{\left| \nabla\varphi\right|^2}{2}dx
+\Psi(\varphi).
\end{align}
Let $j\in\left\{1,..,M-2\right\}$ be arbitrarily fixed.
Then by repeatedly applying Lemma \ref{energyest} with the setting (\ref{V}) for $i=j,j-1,..,0$, %(cf. proof of Theorem \ref{feaspoint})
% \begin{align*}
% \begin{aligned}
% &v_0
% =v_i,\ 
% \varphi_0
% =\varphi_i,\ 
% f_0
% =\rho(\varphi_i),\ 
% f_{-1}
% =\rho(\varphi_{i-1})\\
% &\nu
% =\overline{\nu}(v_i,\varphi_i,\varphi_{i-1},\mu_i),\
% \Theta_v
% =u_{i+1},\ 
% \Theta_\varphi
% =\Theta_\mu
% =0,
% \end{aligned}
% \end{align*}
we conclude that
\begin{align*}
&E(v_{j+1},\varphi_{j+1},\varphi_{j})
+\tau\int_\Omega2\eta(\varphi_{j})\left|\epsilon(v_{j+1})\right|^2dx
+\tau\int_\Omega m(\varphi_{j})\left|\nabla\mu_{j+1}\right|^2dx\\
&\leq E(v_{j},\varphi_{j},\varphi_{j-1})
+\left(u_{j+1},v_{j+1}\right)\\
&\leq E(v_{j-1},\varphi_{j-1},\varphi_{j-2})
+\left(u_{j},v_{j}\right)
+\left(u_{j+1},v_{j+1}\right)\\
&\ :\\
&\leq E(v_{0},\varphi_{0},\varphi_{-1})
+\sum_{i=1}^{j+1}\left(u_{i},v_{i}\right).
\end{align*}
By Assumptions \ref{assum1} and \ref{assPsi} this yields
\begin{align}
&\int_\Omega\frac{\left| \nabla\varphi_{j+1}\right|^2}{2}dx
+\Psi(\varphi_{j+1}) %
+2 \tau b_1\int_\Omega\left|\epsilon(v_{j+1})\right|^2dx
+\tau b_1\int_\Omega \left|\nabla\mu_{j+1}\right|^2dx\nonumber\\
% &\leq E(v_{0},\varphi_{0},\varphi_{-1})
% +\sum_{i=1}^{j+1}\left(u_{i},v_{i}\right)\nonumber\\
&\leq E(v_{0},\varphi_{0},\varphi_{-1})
+\sum_{i=1}^{M-1}\left\|u_{i}\right\|\left\| v_{i}\right\|
\leq C_1 %(v_{a},\varphi_{a})
+C_2\left\|u\right\|_{(L^2)^{M-1}}\left\| v\right\|_{(H^1)^M},\label{potest}
\end{align}
where $C_1$ depends only on the initial data $(N,\Omega,B_l,B_u,v_{a},\varphi_{a})$.
Due to Korn's inequality and Poincar\'e's inequality, this ensures
%the existence of constants $C_1,C_2$ %$C_2=C_2(h,\Omega)$
%such that %
\begin{align*}
\left\|v_{j+1}\right\|^2_{H^1}
+\left\|\mu_{j+1}\right\|^2_{H^1}
+\left\|\varphi_{j+1}\right\|^2_{H^1}
\leq C_1 %(v_{a},\varphi_{a},h)
+C_2\left\|u\right\|_{(L^2)^{M-1}}\left\| v\right\|_{(H^1)^M}.
\end{align*}
Since $j\in\left\{1,..,M-1\right\}$ is arbitrarily chosen, we infer
\begin{align}
&\left\|v\right\|^2_{(H^1)^{M}}
+\left\|\mu\right\|^2_{(H^1)^{M}}
+\left\|\varphi\right\|^2_{(H^1)^{M+1}}
\leq  %\left\|v_a\right\|^2_{(H^1)} %
C_1 %(v_a,\varphi_a,h)
+C_2\left(\left\|u\right\|_{(L^2)^{M-1}}\left\| v\right\|_{(H^1)^M}\right)\label{enerineq}.
\end{align} % % % % %
Hence $(\varphi,\mu,v)$ is bounded in $\overline{H}^1(\Omega)^{M} \times \overline{H}^1(\Omega)^M\times H^{1}_{0,\sigma}(\Omega;\mathbb{R}^N)^{M-1}$.
Then the boundedness in the respective $H^2$-spaces follows directly by applying Lemma \ref{regsol} for each time step.\end{proof}

Next we address the case of the double-well potential
and show that for appropriate double-well type potentials
%the density corresponding to a solution of the semi-discrete CHNS system remains strictly positive.
the order parameter of a solution to the system (\ref{firsttim1}),(\ref{firsttim2}),(\ref{NSwithop})
%can not take values much smaller than $\psi_1$.
is always greater than $\psi_1-\varepsilon$ for some small $\varepsilon>0$.
%takes values in a bounded interval close to $[\psi_1,\psi_2]$.
%the next theorem verifies that for an appropriate sequence of double-well type potentials the interval
%converges and a corresponding sequence of sol
\begin{theorem}\label{Linftycon} %
Let $u\in L^2(\Omega;\mathbb{R}^N)^{M-1}$ be given and
$\left\{\Psi_0^{(k)}\right\}_{k\in\mathbb{N}}$ a sequence of functions which satisfies the following two conditions:
\begin{enumerate}
\item For every $k\in\mathbb{N}$ $\Psi_0^{(k)}$ fulfills Assumption \ref{assPsi}.
%which satisfy Assumption \ref{assPsi} such that the subsequent condition is fulfilled.
\item If $\left\{\hat{\varphi}^{(k)}\right\}_{k\in\mathbb{N}}$ is a sequence in $\overline{H}^1(\Omega)$
such that there exists $C>0$ with $\Psi_0^{(k)}\left(\hat{\varphi}^{(k)}\right)\leq C$ for $k\in\mathbb{N}$, then
\begin{align*} %
\left\|\max(-\hat{\varphi}^{(k)}+\psi_1,0)\right\|_{L^1}\rightarrow 0,\text{ as }k\rightarrow\infty.
\end{align*}
\end{enumerate}

Furthermore, let $\left\{(\varphi^{(k)},\mu^{(k)},v^{(k)})\right\}_{k\in\mathbb{N}}$ be a sequence of solutions to the systems
(\ref{firsttim1}),(\ref{firsttim2}),(\ref{NSwithop}) with $\Psi_0=\Psi_0^{(k)}$. % for every $k\in\mathbb{N}$.
Then
$$\left\|\max(-\varphi^{(k)}+\psi_1,0)\right\|_{L^\infty}\rightarrow 0,\text{ as }k\rightarrow\infty.$$ %
\end{theorem}
\begin{proof}
Employing Lemma \ref{energy2}, in particular inequality (\ref{potest}) from its proof, then we see that for every $i\in\left\{-1,..,M-1\right\}$ and $k\in\mathbb{N}$ it holds that
%\begin{align*}
$\Psi^{(k)}(\varphi_i^{(k)})\leq C_1.$
%\end{align*}
Hence, we conclude
\begin{align*}
\Psi_0^{(k)}\left(\varphi_i^{(k)}\right)\leq C_1+\frac{\kappa}{2}\left\|\varphi_i^{(k)}\right\|^2_{L^2}\leq C_2.
\end{align*}
By assumption, this yields
\begin{align}
\left\|\max(-\varphi_i^{(k)}+\psi_1,0)\right\|_{L^1}\rightarrow 0\textnormal{ for }k\rightarrow\infty.\label{hinf}
\end{align}
Next, we use the technique of \cite[Proposition 2.4]{Hintermueller2014a} and \cite[Remark 2.5]{Hintermueller2014a} to derive that
$\left\|\max(-\varphi_i^{(k)}+\psi_1,0)\right\|_{L^\infty}\rightarrow 0$ for $k\rightarrow\infty$.
%For the sake of clarity we shortly sketch it here.\\
We stay brief here and refer to \cite{Hintermueller2014a} for details on the technique.
By Lemma \ref{energy2} the sequence $\left\{\varphi^{(k)}\right\}_{k\in\mathbb{N}}$ is bounded in $\overline{H}^2(\Omega)$,
and due to Sobolev's embedding theorem it is also bounded in $W^{1,6}(\Omega)$ and $C^{0,\beta}(\overline{\Omega})$, $\beta\leq\frac{1}{2}$, respectively.
Thus, there exists a constant $C_\beta$ such that for every $k \in \mathbb{N}$ we have $\left\|\varphi^{(k)}\right\| _{C^{0,\beta}}\leq C_\beta$.

For fix $k\in\mathbb{N}$ assume that $\left\|\max(-\varphi_i^{(k)}+\psi_1,0)\right\|_{L^\infty}>0$
and define the set $\mathbb{G}:=\left\{\omega\in\overline{\Omega}:\varphi_i^{(k)}(\omega)\leq\psi_1<0\right\}$.
%and assume that $\left\|\max(-\varphi_i^{(k)}+\psi_1,0)\right\|_{L^\infty}>0$
Then let $\omega_{max}\in\mathbb{G}$ satisfy
\begin{align*}
-\varphi_i^{(k)}(\omega_{max})+\psi_1=\left\|-\varphi_i^{(k)}+\psi_1\right\|_{L^\infty(\mathbb{G})}=\left\|\max(-\varphi_i^{(k)}+\psi_1,0)\right\|_{L^\infty(\Omega)}.
\end{align*}
Due to the H\"older continuity of $\varphi_i^{(k)}$,
for every $x\in\Omega$ which satisfies
$\left|x-\omega_{max}\right|_{\mathbb{R}^N}\leq \left(\frac{-\varphi_i^{(k)}(\omega_{max})+\psi_1}{2C_\beta}\right)^\frac{1}{\beta}$ it holds that
\begin{align*}
-\varphi_i^{(k)}(x)+\psi_1
&\geq-\varphi_i^{(k)}(\omega_{max})+\psi_1-\left\|\varphi_i^{(k)}\right\|_{C^{(0,\beta)}(\Omega)}\left|\omega_{max}-x\right|^\beta_{\mathbb{R}^N}\\
&\geq\frac{-\varphi_i^{(k)}(\omega_{max})+\psi_1}{2}>0.
\end{align*} % % %
As %$\omega_{max}\in\Omega$ and
$\Omega$ satisfies the cone condition, there exists a finite cone $K_r(\omega_{max}):=K(\omega_{max})\cap B(\omega_{max},r)$ of radius $r$ and with vertex $\omega_{max}$ such that $K_r(\omega_{max})\subset\Omega$.
Hence the cone $K_R(\omega_{max})$ with
$R:=\min\left(r,\left(\frac{-\varphi_i^{(k)}(\omega_{max})+\psi_1}{2C_\beta}\right)^\frac{1}{\beta}\right)$
is contained in $\mathbb{G}$.
Consequently, we find
\begin{align*}
\left\|\max(-\varphi_i^{(k)}+\psi_1,0)\right\|_{L^1(\Omega)} %
&\geq\int_{K_R(\omega_{max})}-\varphi_i^{(k)}+\psi_1dx\\
&\geq\int_{K_R(\omega_{max})} %\left\|\max(-\varphi_i^{(k)}+\psi_1,0)\right\|_{L^\infty(\Omega)}
\frac{\left(-\varphi_i^{(k)}(\omega_{max})+\psi_1\right)}{2}dx\\
&\geq \frac{|K_R(0)|}{2}\left\|\max(-\varphi_i^{(k)}+\psi_1,0)\right\|_{L^\infty(\Omega)}
\end{align*} % % % %
In combination with (\ref{hinf}) this proves the assertion.
\end{proof}

%The last theorem yields the following.
% Let $u\in L^2(\Omega;\mathbb{R}^N)^{M-1}$ be given and
We define $\varphi^-\in\mathbb{R}$ as
\begin{align}
\varphi^-:=\inf\left\{\varphi\in\mathbb{R}:\rho(\varphi)>0\right\}<\psi_1.
\end{align}
Let $u\in L^2(\Omega;\mathbb{R}^N)^{M-1}$ be given and let $\left\{\Psi_0^{(k)}\right\}_{k\in\mathbb{N}}$ be a sequence of double-well type potentials
which approximates the double-obstacle potential in a certain sense, i.e., it satisfies condition 2 of Theorem \ref{Linftycon}.
Then Theorem \ref{Linftycon} ensures that there exists $k^*\in\mathbb{N}$ such that for every $k\geq k^*$ the solutions
$(\varphi^{(k)},\mu^{(k)},v^{(k)})$ %\in \overline{H}^2(\Omega)^{M} \times \overline{H}^2(\Omega)^{M}\times H^{1}_{0,\sigma}(\Omega;\mathbb{R}^N)^{M-1}$
to the corresponding systems (\ref{firsttim1}),(\ref{firsttim2}),(\ref{NSwithop}) with $\Psi_0=\Psi_0^{(k)}$ satisfy
\begin{align}
\varphi_i^{(k)}>\varphi^-,\ \forall i=-1,..,M-1.
\end{align}
Hence $\rho(\varphi_i^{(k)})>0$ for every $i=-1,..,M-1$ and $k\geq k^*$.
Thus, (\ref{NSwithop}) coincides with (\ref{firsttim3}),
which leads to the subsequent theorem.
\begin{theorem}[Existence of feasible points]
Let $u\in L^2(\Omega;\mathbb{R}^N)^{M-1}$.
Let $\overline{\Psi}_0$ be the double-obstacle potential defined in Assumption \ref{assPsi}.1
and let $\left\{\Psi_0^{(k)}\right\}_{k\in \mathbb{N}}$ be a sequence which satisfies the conditions of Theorem \ref{Linftycon}.

Then there exists $k^*\in\mathbb{N}$ such that the system (\ref{firsttim1})-(\ref{firsttim3}) admits a solution $(\varphi,\mu,v)$
for every $\Psi_0\in\left\{\overline{\Psi}_0\right\}\cup\left\{\Psi_0^{(k)}\right\}_{k\geq k^*}$. For this solution $(\varphi,\mu,v)$ the result of Lemma \ref{energy2} remains true.
\end{theorem}

% and hence, $(\varphi^{(k)},\mu^{(k)},v^{(k)})$ also solves the semi-discrete CHNS system (\ref{firsttim1})-(\ref{firsttim3}),
% i.e. $(\varphi^{(k)},\mu^{(k)},v^{(k)})\in S_{\Psi^{(k)}}.$\\
In other words, the semi-discrete CHNS system (\ref{firsttim1})-(\ref{firsttim3}) has a solution if the double-well type potential under consideration
is close enough to the double-obstacle potential.
In the following sections we always assume that this is the case.
In Definition \ref{defpotapp} below, we propose a specific regularization which satisfies the conditions of Theorem \ref{Linftycon}.
\section{Existence of globally optimal points}
% The next step is to show the existence of a global solution to the problem.
The previous section guarantees the existence of feasible points for the optimal control problem ($P_\Psi$).
Next we investigate the existence of a solution to ($P_\Psi$).
For this purpose, we need to impose additional assumptions on the objective functional and the constraint set $U_{ad}$.
\begin{theorem}[Existence of global solutions]\label{exsol}
%Let the assumptions of Theorem \ref{feaspoint} be satisfied.\\
%Let $U_{ad}$ be a closed, convex subset of $L^2(\Omega;\mathbb{R}^N)^{M-1}$.
Suppose that $\mathcal{J}:\overline{H}^2_{\partial_n}(\Omega)^{M} \times \overline{H}^2_{\partial_n}(\Omega)^{M}\times H^{1}_{0,\sigma}(\Omega;\mathbb{R}^N)^{M-1}\times L^2(\Omega;\mathbb{R}^N)^{M-1}\rightarrow\mathbb{R}$
is convex and weakly lower-semi-continuous and $U_{ad}$ is non-empty, closed and convex.
Assume that either $U_{ad}$ is bounded
or $\mathcal{J}$ is partially coercive, i.e. for every sequence $\left\{(\varphi^{(k)},\mu^{(k)},v^{(k)},u^{(k)})\right\}_{k\in\mathbb{N}}$
with $\lim_{k\rightarrow\infty}\left\| u^{(k)}\right\|=\infty$ it holds that
$\lim_{k\rightarrow\infty} \mathcal{J}(\varphi^{(k)},\mu^{(k)},v^{(k)},u^{(k)})=\infty.$
Then the optimization problem ($P_\Psi$) admits a global solution. 
\end{theorem}
\begin{proof}
By Theorem \ref{feaspoint} the feasible set of the problem ($P_\Psi$) is non-empty and contained in
$\overline{H}^2_{\partial_n}(\Omega)^{M} \times \overline{H}^2_{\partial_n}(\Omega)^{M}\times H^{1}_{0,\sigma}(\Omega;\mathbb{R}^N)^{M-1}\times U_{ad}$.

%Since $\mathcal{J}$ is bounded, there exists
Let $\left\{(\varphi^{(k)},\mu^{(k)},v^{(k)},u^{(k)})\right\}_{k\in\mathbb{N}}$ be an infimizing sequence of ${\mathcal{J}}$
in $\overline{H}^2_{\partial_n}(\Omega)^{M} \times \overline{H}^2_{\partial_n}(\Omega)^{M}\times H^{1}_{0,\sigma}(\Omega;\mathbb{R}^N)^{M-1}\times U_{ad}$ with $(\varphi^{(k)},\mu^{(k)},v^{(k)})\in S_\Psi(u^{(k)})$
such that
\begin{align}
\lim_{k\rightarrow\infty}{\mathcal{J}}(\varphi^{(k)},\mu^{(k)},v^{(k)},u^{(k)})=\inf_{u\in U_{ad},(\varphi,\mu,v)\in S_\Psi(u)} {\mathcal{J}}(\varphi,\mu,v,u).\label{infJ}
\end{align}
Note that the infimum on the right-hand side may be $-\infty$.
The sequence $\left\{u^{(k)}\right\}_{k\in\mathbb{N}}$ is bounded in the reflexive Banach space $L^2(\Omega;\mathbb{R}^N)^{M-1}$.
This follows either directly from the boundedness of the set $U_{ad}$ or from the partial coercivity of $\mathcal{J}$.
Then by Lemma \ref{energy2} the sequence $(\varphi^{(k)},\mu^{(k)},v^{(k)})$ is bounded
in $\overline{H}^2_{\partial_n}(\Omega)^{M} \times \overline{H}^2_{\partial_n}(\Omega)^{M}\times H^{1}_{0,\sigma}(\Omega;\mathbb{R}^N)^{M-1}$.
Setting $\left\{w^{(k)}\right\}_{k\in\mathbb{N}}:=\left\{(\varphi^{(k)},\mu^{(k)},v^{(k)},u^{(k)})\right\}_{k\in\mathbb{N}}$, there exists a weakly convergent subsequence
% $\left\{w^{(l)}\right\}_{l\in\mathbb{N}}=\left\{w^{(k_l)}\right\}_{l\in\mathbb{N}}=\left\{(\varphi^{(k_l)},\mu^{(k_l)},v^{(k_l)},u^{(k_l)})\right\}_{l\in\mathbb{N}}$
$\left\{w^{(k_l)}\right\}_{l\in\mathbb{N}}$
with limit point $w^*:=(\varphi^*,\mu^*,v^*,u^*)\in \overline{H}^2_{\partial_n}(\Omega)^{M} \times \overline{H}^2_{\partial_n}(\Omega)^{M}\times H^{1}_{0,\sigma}(\Omega;\mathbb{R}^N)^{M-1}$.
%in the respective product space.
Using the weak lower-semicontinuity of $\mathcal{J}$, this implies
\begin{align*}
-\infty<{\mathcal{J}}(w^*)\leq\liminf_{l\rightarrow\infty}\mathcal{J}(w^{(k_l)})=\inf_{u\in U_{ad},(\varphi,\mu,v)\in S_\Psi(u)} {\mathcal{J}}(\varphi,\mu,v,u) %
\end{align*}
where the last equality holds due to (\ref{infJ}).
Since $U_{ad}$ is %
weakly closed, %
$u^*$ belongs to $U_{ad}$.

It remains to show that $(\varphi^*,\mu^*,v^*)\in S_\Psi(u^*)$.
For this purpose, we write $l$ instead of $k_l$, and we start by considering the limit of
$\left\langle-\textnormal{div}(v^{(l)}_{i+1}\otimes \frac{\rho_2-\rho_1}{2}m(\varphi^{(l)}_{i-1})\nabla\mu^{(l)}_i),\psi\right\rangle$
for arbitrary $i\in\left\{0,..,M-2\right\}$ and $\psi\in H^1(\Omega;\mathbb{R}^N)$.
Using the triangle and H\"older's inequality we derive
\begin{align*}
&\left\| m(\varphi^{(l)}_{i-1})\nabla\mu^{(l)}_i\cdot\nabla\psi-m(\varphi^*_{i-1})\nabla\mu^*_i\cdot\nabla\psi\right\|_{L^{4/3}}\\
\leq& \left\| m(\varphi^{(l)}_{i-1})(\nabla\mu^{(l)}_i-\nabla\mu^*_i)\cdot\nabla\psi\right\|_{L^{4/3}}
+\left\| (m(\varphi^{(l)}_{i-1})-m(\varphi^*_{i-1}))\nabla\mu^*_i\cdot\nabla\psi\right\|_{L^{4/3}}\\
\leq& \left\| m(\varphi^{(l)}_{i-1})\right\|_{L^\infty}\hspace{-0.1cm} \left\|\nabla\mu^{(l)}_i\hspace{-0.1cm}-\nabla\mu^*_i \right\|_{L^4}\hspace{-0.1cm}\left\|\nabla\psi\right\|_{L^2}\hspace{-0.1cm}
+\hspace{-0.1cm}\left\| m(\varphi^{(l)}_{i-1})-m(\varphi^*_{i-1})\right\|_{L^\infty}\hspace{-0.1cm}\left\|\nabla\mu^*_i\right\|_{L^4}\hspace{-0.1cm}\left\|\nabla\psi\right\|_{L^2}.
\end{align*}
Since $\nabla\mu^{(l)}_i$ converges weakly to $\nabla\mu^*_i$ in $H^1(\Omega)$ and $H^1(\Omega)$ is compactly embedded into $L^4(\Omega)$,
$\left\|\nabla\mu^{(l)}_i-\nabla\mu^*_i \right\|_{L^4}$ tends to zero for $l\rightarrow\infty$.
Due to the compact embedding of $H^{2}(\Omega)$ into $W^{1,4}(\Omega)$,
we have $\varphi^{(l)}_{i-1}\rightarrow \varphi^*_{i-1}$ strongly in ${W^{1,4}}(\Omega)$.
Due to Assumption \ref{assum1}.1, $m$ is Lipschitz continuous.
Since ${W^{1,4}}(\Omega) $ can be embedded into $L^\infty(\Omega)$,
we infer $\left\| m(\varphi^{(l)}_{i-1})-m(\varphi^*_{i-1})\right\|_{L^\infty}\rightarrow 0$.

Consequently, the sequence $\frac{\rho_2-\rho_1}{2}m(\varphi^{(l)}_{i-1})\nabla\mu^{(l)}_i\cdot\nabla\psi$ converges strongly in $L^{4/3}(\Omega)$
to $\frac{\rho_2-\rho_1}{2}m(\varphi^*_{i-1})\nabla\mu^*_i\cdot\nabla\psi$.
By Sobolev's embedding theorem and the weak continuity of the embedding operator, $v^{(l)}_{i+1}$ converges weakly in
$L^{4}(\Omega)$ to $v^*_{i+1}$.
Hence $\left\langle\textnormal{div}(v^{(l)}_{i+1}\otimes\hspace{-0.05cm} \frac{\rho_2-\rho_1}{2}m(\varphi^{(l)}_{i-1})\nabla\mu^{(l)}_i),\psi\right\rangle$
converges to %for $l\rightarrow\infty$ towards
$\left\langle\textnormal{div}(v^*_{i+1}\otimes \frac{\rho_2-\rho_1}{2}m(\varphi^*_{i-1})\nabla\mu^*_i),\psi\right\rangle$ as $l\rightarrow\infty$.

One proceeds analogously for the remaining terms in the system (\ref{firsttim1})-(\ref{firsttim3})
which do not depend on the subdifferential of $\Psi_0$.

In this way we also show that $\Delta\varphi^{(l)}_{i+1}+\mu^{(l)}_{i+1}+\kappa\varphi^{(l)}_{i}$ converges strongly in $\overline{H}^{-1}(\Omega)$ to
$\Delta\varphi^*_{i+1}+\mu^*_{i+1}+\kappa\varphi^*_{i}$ for every $i=-1,..,M-2$.
Furthermore, $\varphi^{(l)}_{i+1}\rightarrow\varphi^*_{i+1}$ in ${\overline{H}^1}(\Omega)$, and for every $l\in\mathbb{N}$ % and $i\in\left\{-1,..,M-2\right\}$
it holds that
$\Delta\varphi^{(l)}_{i+1}+\mu^{(l)}_{i+1}+\kappa\varphi^{(l)}_{i}\in\partial\Psi_0(\varphi^{(l)}_{i+1})$.
Due to the maximal monotonicity of $\partial\Psi_0$, this implies
\begin{align}
\Delta\varphi^*_{i+1}+\mu^*_{i+1}+\kappa\varphi^*_{i}\in\partial\Psi_0(\varphi^*_{i+1})\label{feaslim}
\end{align}
for every $i=-1,..,M-2$. In summary, we have shown $(\varphi^*,\mu^*,v^*)\in S_\Psi(u^*)$.
Hence the $w^*$ is contained in the feasible set of the problem ($P_\Psi$) and therefore solves the problem.\end{proof}  % % %
%\\
%Thus, the optimal control problems under consideration always possess a globally optimal solution.

\section{Convergence of minimizers}
Now we turn our focus to the consistency of the regularization,
i.e. the convergence of a sequence of solutions to ($P_{\Psi^{(k)}}$) with $\Psi^{(k)}$ a double-well potential approaching the double-obstacle potential in the limit as $k\rightarrow\infty$,
to a solution of ($P_{{\Psi}}$) with $\Psi$ the double-obstacle potential.
For this purpose, we consider a sequence of functionals $\left\{\Psi^{(k)}\right\}_{k\in\mathbb{N}}$
satisfying Assumption \ref{assPsi}.2 and a corresponding limit functional $\overline{\Psi}$.

The following theorem provides conditions under which a sequence of globally optimal solutions to ($P_{\Psi^{(k)}}$) converge
to a global solution of ($P_{\overline{\Psi}}$), as $k\rightarrow\infty$.
% In the previous section we always considered the problem ($P_\Psi$) for a fixed potential $\Psi$.
% The next theorem establishes a convergence result for a series of functionals $\Psi^{(k)}$. % % % %
\begin{theorem}[Consistency of the regularization]\label{conmin}
Let the assumptions of Theorem \ref{exsol} be fulfilled.
The objective $\mathcal{J}:\overline{H}^1(\Omega)^{M} \times \overline{H}^1(\Omega)^{M}\times H^{1}_{0,\sigma}(\Omega;\mathbb{R}^N)^{M-1}\times L^2(\Omega;\mathbb{R}^N)^{M-1}\rightarrow\mathbb{R}$
is supposed to be upper-semicontinuous,
and let $\left\{\Psi^{(k)}\right\}_{k\in\mathbb{N}}$ be a sequence of potentials satisfying Assumption \ref{assPsi}.2.
Assume further that $\overline{\Psi}$ is given such that
%the following assumptions are satisfied.\\
%\textbf{(A)} Assume that %
for every sequence
$\left\{(x^{(k)},y^{(k)})\right\}_{k\in\mathbb{N}}\subset \overline{H}^1(\Omega)\times \overline{H}^{-1}(\Omega)$
with $y^{(k)}={\Psi^{(k)}}'(x^{(k)})$
and $(x^{(k)},y^{(k)})\rightarrow (x^{(\infty)},y^{(\infty)})$ strongly in $\overline{H}^1(\Omega)\times \overline{H}^{-1}(\Omega)$
%which converges to $(x^{(\infty)},y^{(\infty)})$
it holds that $y^{(\infty)}\in \partial{\overline{\Psi}}(x^{(\infty)})$.

% \textbf{(B)} For every $k\in\mathbb{K}$ the functional $\Psi_0^{(k)}:\overline{H}^1(\Omega)\rightarrow \mathbb{R}$ is continuously Fr\'echet differentiable.
% %such that $\partial \Psi_0^{(k)}(\varphi)=\left\{\left(\Psi_0^{(k)}\right)'(\varphi)\right\}$.\\
% \\
% Furthermore, let the assumptions of Theorem \ref{exsol} be fulfilled. %
% If $\left\{(\varphi^{(k)},\mu^{(k)},v^{(k)},u^{(k)})\right\}_{k\in\mathbb{N}}$ is a sequence of globally optimal points of ($P_{\Psi^{(k)}}$) in
% $\overline{H}^2(\Omega)^{M} \times \overline{H}^2(\Omega)^{M}\times H^{1}_{0,\sigma}(\Omega;\mathbb{R}^N)^{M-1}\times U_{ad}$
% (if $U_{ad}$ is unbounded, additionally assume that the sequence $\left\{\mathcal{J}(\varphi^{(k)},\mu^{(k)},v^{(k)},u^{(k)})\right\}_{k\in\mathbb{N}}$ is bounded),
% then it converges to a globally optimal point of ($P_{\overline{\Psi}}$).
Then a sequence $\left\{(\varphi^{(k)},\mu^{(k)},v^{(k)},u^{(k)})\right\}_{k\in\mathbb{N}}$ of \hspace{0.1cm}global\hspace{0.1cm} solutions \hspace{0.1cm}to ($P_{\Psi^{(k)}}$) in 
$\overline{H}^2(\Omega)^{M} \times \overline{H}^2(\Omega)^{M}\times H^{1}_{0,\sigma}(\Omega;\mathbb{R}^N)^{M-1}\times U_{ad}$
converges to a global solution of ($P_{\overline{\Psi}}$), provided that $\left\{\mathcal{J}(\varphi^{(k)},\mu^{(k)},v^{(k)},u^{(k)})\right\}_{k\in\mathbb{N}}$ is assumed bounded,
whenever $U_{ad}$ is unbounded.
\end{theorem}
\begin{proof}
First note that the sequence $\left\{u^{(k)}\right\}_{k\in\mathbb{N}}$ is bounded in the reflexive Banach space $L^2(\Omega;\mathbb{R}^N)^{M-1}$.
This follows either from the boundedness of the set $U_{ad}$ or from the partial coercivity of $\mathcal{J}$ and the boundedness of
$\left\{\mathcal{J}(\varphi^{(k)},\mu^{(k)},v^{(k)},u^{(k)})\right\}_{k\in\mathbb{N}}$.
By Lemma \ref{energy2}, the sequence $\left\{(\varphi^{(k)},\mu^{(k)},v^{(k)})\right\}_{k\in\mathbb{N}}$ is bounded
in $\overline{H}^2_{\partial_n}(\Omega)^{M} \times \overline{H}^2_{\partial_n}(\Omega)^{M}\times H^{1}_{0,\sigma}(\Omega;\mathbb{R}^N)^{M-1}$.
Hence there exists a weakly convergent sequence
$\left\{w^{(k_l)}\right\}_{l\in\mathbb{N}}:=\left\{(\varphi^{(k_l)},\mu^{(k_l)},v^{(k_l)},u^{(k_l)})\right\}_{l\in\mathbb{N}}$
with limit point
$\overline{w}\hspace{-0.1cm}:=\hspace{-0.1cm}(\overline{\varphi},\overline{\mu},\overline{v},\overline{u})\hspace{-0.1cm}\in\hspace{-0.1cm}\overline{H}^2_{\partial_n}(\Omega)^{M} \times \overline{H}^2_{\partial_n}(\Omega)^{M}\times H^{1}_{0,\sigma}(\Omega;\mathbb{R}^N)^{M-1}$.
Moreover, since $U_{ad}$ is weakly closed, $\overline{u}$ belongs to $U_{ad}$.

As in the proof of Theorem \ref{exsol}, it can be shown that the limit point satisfies $(\overline{\varphi},\overline{\mu},\overline{v})\in S_{\overline{\Psi}}(\overline{u})$.
The only difference is that inclusion (\ref{feaslim}) follows from the above assumption instead of the maximal monotonicity.

Next, we prove that $\overline{w}$ is an optimal point of ($P_{\overline{\Psi}}$).
For this purpose, let $(\widehat{\varphi},\widehat{\mu},\widehat{v},\widehat{u})$ be an optimal solution of ($P_{\overline{\Psi}}$). % and fix $\widehat{v}$.\\
We consider a sequence $(\widehat{\varphi}^{(k)},\widehat{\mu}^{(k)})\in \overline{H}^2_{\partial_n}(\Omega)^{M} \times \overline{H}^2_{\partial_n}(\Omega)^{M}$
such that
\begin{align*}
\left\langle\frac{\widehat{\varphi}^{(k)}_{i+1} -\widehat{\varphi}^{(k)}_{i} }{\tau},\phi\right\rangle
+\left\langle \widehat{v}_{i+1}\nabla\widehat{\varphi}^{(k)}_{i},\phi\right\rangle
-\left\langle\textnormal{div}(m(\widehat{\varphi}^{(k)}_{i})\nabla\widehat{\mu}^{(k)}_{i+1}),\phi\right\rangle=0,\\ %,\ \forall \phi\in \overline{H}^1(\Omega),\\ % % % %
\left\langle -\Delta\widehat{\varphi}^{(k)}_{i+1},\phi\right\rangle
+\left\langle \left(\Psi_0^{(k)}\right)'(\widehat{\varphi}^{(k)}_{i+1}),\phi\right\rangle
-\left\langle \widehat{\mu}^{(k)}_{i+1},\phi\right\rangle
-\left\langle \kappa\widehat{\varphi}^{(k)}_{i},\phi\right\rangle=0, %\ \forall \phi\in \overline{H}^1(\Omega),
\end{align*}
for every $\phi\in \overline{H}^1(\Omega)$ and $i\in\left\{-1,..,M-2\right\}$,
where $\widehat{v}$ corresponds to the previously specified solution of ($P_{\overline{\Psi}}$).
Note that the operator $L^{(k)}_a:\overline{H}^1(\Omega)\times \overline{H}^1(\Omega)\rightarrow \overline{H}^{-1}(\Omega)\times \overline{H}^{-1}(\Omega)$ defined by
\begin{align}
L^{(k)}_a(\varphi,\mu):=\left(-\Delta\varphi+\left(\Psi^{(k)}_0\right)'(\varphi)-\mu,\varphi-\textnormal{div}(a\nabla\mu)\right)
\end{align}
is monotone, coercive and continuous, if $a\in H^2(\Omega)$ satisfies $0<\tau b_1\leq a(x)\leq \tau b_2$ almost everywhere on $\Omega$.
% To see this, note that Poincar\'e's inequality yields %
% \begin{align}
% &\left\langle L^{(k)}_a(\varphi_1,\mu_1)-L^{(k)}_a(\varphi_2,\mu_2),(\varphi_1,\mu_1)-(\varphi_2,\mu_2)\right\rangle\nonumber\\
% &=\left\|\nabla(\varphi_1-\varphi_2)\right\|^2
% +\left\langle \left(\Psi_0^{(k)}\right)'(\varphi_1)-\left(\Psi_0^{(k)}\right)'(\varphi_2),(\varphi_1-\varphi_2)\right\rangle\nonumber\\
% &+\left( a\nabla(\mu_1-\mu_2),\nabla(\mu_1-\mu_2)\right)\nonumber\\
% &\geq C(\left\|\varphi_1-\varphi_2\right\|_{H^1}^2+\left\|\mu_1-\mu_2\right\|_{H^1}^2).\label{hemme}
% \end{align}
Hence for fixed $k\in\mathbb{N}$, the pair $(\widehat{\varphi}^{(k)}_{i+1},\widehat{\mu}^{(k)}_{i+1})$ of each subsequent time step is uniquely determined as the solution to
\begin{align}
L^{(k)}_{m(\widehat{\varphi}^{(k)}_{i})\tau}(\widehat{\varphi}^{(k)}_{i+1},\widehat{\mu}^{(k)}_{i+1})
=(\kappa \widehat{\varphi}^{(k)}_{i},\widehat{\varphi}^{(k)}_{i}-\tau\widehat{v}_{i+1}\nabla\widehat{\varphi}^{(k)}_{i})
\end{align}
where $0<\tau b_1\leq a:=m(\widehat{\varphi}^{(k)}_{i})h\leq \tau b_2$ almost everywhere on $\Omega$ (cf. \cite[Chapter II, Theorem 2.2]{Showalter1997}.
Then, by Lemma \ref{regsol} the sequence $(\widehat{\varphi}^{(k)},\widehat{\mu}^{(k)},\widehat{v})_{k\in\mathbb{N}}$ is bounded in
$\overline{H}^2(\Omega)^{M} \times \overline{H}^2(\Omega)^{M}\times H^{1}_{0,\sigma}(\Omega;\mathbb{R}^N)^{M-1}$.
Consequently, there exists a subsequence (denoted the same) which converges weakly in the associated product space to a limit point $(\widehat{\varphi}^*,\widehat{\mu}^*,\widehat{v})$.
In accordance with the above observations, % (in particular (\ref{hemme})), % and due to assumption \textbf{(A)},
$(\widehat{\varphi}_{i+1}^*,\widehat{\mu}_{i+1}^*)$ is the unique solution to
\begin{align*}
\left\langle\frac{\widehat{\varphi}^*_{i+1} -\widehat{\varphi}^*_{i} }{\tau},\phi\right\rangle
+\left\langle \widehat{v}_{i+1}\nabla\widehat{\varphi}^*_{i},\phi\right\rangle
-\left\langle\textnormal{div}(m(\widehat{\varphi}^*_{i})\nabla\widehat{\mu}^*_{i+1}),\phi\right\rangle=0,\ \forall \phi\in \overline{H}^1(\Omega),\\
\left\langle -\Delta\widehat{\varphi}^*_{i+1},\phi\right\rangle
+\left\langle \partial\Psi_0^*(\widehat{\varphi}^*_{i+1}),\phi\right\rangle
-\left\langle \widehat{\mu}^*_{i+1},\phi\right\rangle
-\left\langle \kappa\widehat{\varphi}^*_{i},\phi\right\rangle=0,\ \forall \phi\in \overline{H}^1(\Omega)
\end{align*}
for every $i\in\left\{-1,..,M-2\right\}$.
Note that here we also use the prerequisite that $y^*_{i+1}\in\partial\Psi^*_0(\widehat{\varphi}^*_{i+1}) $
when $(y^{(k)}_{i+1}, \widehat{\varphi}^{(k)}_{i+1})\rightarrow(y^*_{i+1},\widehat{\varphi}^*_{i+1})$ with $y^{(k})_{i+1}={\Psi^{(k)}_0}'(\widehat{\varphi}^{(k)}_{i+1}) $.
Since the feasibility of $(\widehat{\varphi},\widehat{\mu},\widehat{v},\widehat{u})$ implies
$(\widehat{\varphi},\widehat{\mu},\widehat{v})\in S_{\overline{\Psi}}(\widehat{u})$, this yields
$\widehat{\varphi}^*=\widehat{\varphi}$ and $\widehat{\mu}^*=\widehat{\mu}$.

Now we show that $\widehat{\mu}^{(k)}$ converges strongly in $(\overline{H}^2_{\partial_n}(\Omega))^M$ to $\widehat{\mu}^*$.
For this purpose, fix $i\in \left\{-1,..,M-2\right\}$ and define
\begin{align}
g_i^{(k)}:=\frac{\widehat{\varphi}^{(k)}_{i+1} -\widehat{\varphi}^{(k)}_{i} }{\tau}+\widehat{v}_{i+1}\nabla\widehat{\varphi}^{(k)}_{i},\quad
g_i^*:=\frac{\widehat{\varphi}^*_{i+1} -\widehat{\varphi}^*_{i} }{\tau}
+\widehat{v}_{i+1}\nabla\widehat{\varphi}^*_{i}.
\end{align}
By the Rellich-Kondrachov theorem $g_i^{(k)}$ converges strongly in $L^2(\Omega)$ to $g^*_i$.
%For every $i\in\left\{-1,..,M-2\right\}$
It further holds that
% \begin{align*}
$g_i^{(k)}-g_i^*
=\textnormal{div}(m(\widehat{\varphi}^{(k)}_{i})\nabla\widehat{\mu}^{(k)}_{i+1})
-\textnormal{div}(m(\widehat{\varphi}^*_{i})\nabla\widehat{\mu}^*_{i+1}).
% &=\textnormal{div}((m(\widehat{\varphi}^{(k)}_{i})-m(\widehat{\varphi}^*_{i}))\nabla\widehat{\mu}^{(k)}_{i+1})
% +\textnormal{div}(m(\widehat{\varphi}^*_{i})(\nabla\widehat{\mu}^{(k)}_{i+1}-\nabla\widehat{\mu}^*_{i+1}))
$
%\end{align*}
Hence, we have
\begin{align*}
\textnormal{div}(m(\widehat{\varphi}^*_{i})\nabla(\widehat{\mu}^{(k)}_{i+1}-\widehat{\mu}^*_{i+1}))
=&g_i^{(k)}-g_i^*
% &=\textnormal{div}(m(\widehat{\varphi}^{(k)}_{i})\nabla\widehat{\mu}^{(k)}_{i+1})
% -\textnormal{div}(m(\widehat{\varphi}^*_{i})\nabla\widehat{\mu}^*_{i+1})\\
-\textnormal{div}((m(\widehat{\varphi}^{(k)}_{i})-m(\widehat{\varphi}^*_{i}))\nabla\widehat{\mu}^{(k)}_{i+1})=: \delta_i^{(k)}.
\end{align*}
%We denote the right-hand side of the equation by $\delta_i^{(k)}$.
Again by the Rellich-Kondrachov theorem $m(\widehat{\varphi}^{(k)}_{i})$ converges strongly to $m(\widehat{\varphi}^*_{i})$ in $W^{1,5}(\Omega)$.
Furthermore, $\nabla\widehat{\mu}^{(k)}_{i+1}$ is bounded in $H^1(\Omega)$.
As a consequence, $\delta_i^{(k)}\rightarrow 0$ strongly in $L^2(\Omega)$.
Applying \cite[Theorem 2.3.1]{Maugeri2000}, we conclude
\begin{align*}
\left\|\widehat{\mu}^{(k)}_{i+1}-\widehat{\mu}^*_{i+1}\right\|_{H^2}\leq C\left\|\delta_i^{(k)}\right\|\rightarrow 0.
\end{align*}
Next, we define $\widehat{u}_{i+1}^{(k)}\in L^2(\Omega;\mathbb{R}^N)$ for all $i\in\left\{0,..,M-2\right\}$ by
\begin{align*}
\widehat{u}^{(k)}_{i+1}&:=
\frac{\rho(\widehat{\varphi}^{(k)}_{i}) \widehat{v}_{i+1}-\rho(\widehat{\varphi}^{(k)}_{i-1}) \widehat{v}_i}{\tau}
+\textnormal{div}(\widehat{v}_{i+1}\otimes \rho(\widehat{\varphi}^{(k)}_{i-1})\widehat{v}_i)\\
&\hspace{1cm}-\textnormal{div}(\widehat{v}_{i+1}\otimes \frac{\rho_2-\rho_1}{2}m(\widehat{\varphi}^{(k)}_{i-1})\nabla\widehat{\mu}^{(k)}_i)\\
&\hspace{2cm}-\textnormal{div}(2\eta(\widehat{\varphi}^{(k)}_{i})\epsilon(\widehat{v}_{i+1}))
-\widehat{\mu}^{(k)}_{i+1}\nabla\widehat{\varphi}^{(k)}_{i}.
\end{align*} % % % %
Similarly to the proof of Theorem \ref{exsol}, it can be shown that $\widehat{u}^{(k)}$ converges strongly in $L^2(\Omega;\mathbb{R}^N)^{M-1}$ to $\widehat{u}$.

%In doing so one takes advantage of the strong convergence of $\widehat{\mu}^{(k)}$ and the fact that $\widehat{v}$ is fixed.\\
Summarizing, the sequence $\left\{(\widehat{\varphi}^{(k)},\widehat{\mu}^{(k)},\widehat{v},\widehat{u}^{(k)})\right\}_{k\in\mathbb{N}}$ converges towards
$(\widehat{\varphi},\widehat{\mu},\widehat{v},\widehat{u})$ strongly in
$\overline{H}^1(\Omega)^{M} \times \overline{H}^1(\Omega)^{M}\times H^{1}_{0,\sigma}(\Omega;\mathbb{R}^N)^{M-1}\times L^2(\Omega;\mathbb{R}^N)^{M-1}$.
Employing the continuity properties of the objective functional $\mathcal{J}$, this yields
\begin{align}
\mathcal{J}(\overline{\varphi},\overline{\mu},\overline{v},\overline{u})
&\leq \lim_{k\rightarrow\infty} \mathcal{J}(\varphi^{(k)},\mu^{(k)},v^{(k)},u^{(k)})
\leq \lim_{k\rightarrow\infty} \mathcal{J}(\widehat{\varphi}^{(k)},\widehat{\mu}^{(k)},\widehat{v},\widehat{u}^{(k)})\nonumber\\
&\leq  \mathcal{J}(\widehat{\varphi},\widehat{\mu},\widehat{v},\widehat{u}).
\end{align}
Since $(\widehat{\varphi},\widehat{\mu},\widehat{v},\widehat{u})$ is optimal, the assertion holds true.\end{proof}

In summary, the optimal control problems under consideration are well-posed and admit globally optimal solutions.
Furthermore, the chosen regularization approach is consistent in the sense of Theorem \ref{conmin}.
%\definecolor{ccred}{rgb}{1,0,0}

%\vskip2ex
%

\section{Stationarity conditions}
Now we turn our attention to the derivation of stationarity conditions for the optimal control problem.
For smooth potentials $\Psi_0 $ stationarity or first-order optimality conditions for the problem~($P_\Psi$) can be derived by applying classical results concerning the existence of Lagrange multipliers.
The latter approach is employed in the following theorem.
%In what follows, $\widetilde m $ denotes the function $-\frac {\rho_2-\rho_1}2 m $.
%Furthermore, we use the shortcuts $f_i:=f(\varphi _i)$ and $f_i':=f'(\varphi _i)$ for $f\in\{\eta , m ,\widetilde m ,\rho \}$ and $i\in\{-1,...{M-1}\}$.
\begin{theorem}[First-order optimality conditions for smooth potentials]\label{T:Mult}
	Let $\mathcal{J}:\overline{H}^1(\Omega )^ M \times \overline{H}^1(\Omega )^ M \times H^1_{0,\sigma}(\Omega ;\mathbb{R}^N )^{M-1} \times L^2(\Omega;\mathbb{R}^N)^{M-1}\rightarrow\mathbb{R}$
	be Fr\'echet differentiable and
	let $\Psi_0$  satisfy Assumption \ref{assPsi}.2 %: \overline{H}^2_{\partial_n }(\Omega )\rightarrow \mathbb{R}$ be convex, lower-semicontinuous and proper
	such that $\Psi_0'$ maps $ \overline{H}^2_{\partial_n }(\Omega )$ continuously Fr\`echet-differentiably
	into $L^2(\Omega )$.
	Further, let
		$\overline {z}:=(\bar\varphi ,\bar\mu ,\bar v ,\bar u )$ be a minimizer of ($P_\Psi$).
	Then there exist $( p,r,q )\in{ \overline{L}^2(\Omega ) }^ M \times { \overline{L}^2(\Omega ) }^ M \times { H^1_{0,\sigma}(\Omega ;\mathbb{R}^N ) }^{M-1} $,
 		$ p =( p _{-1},... p _{{M-2}})$,
		$ r =( r _{-1},... r _{{M-2}})$,
		$ q =( q _{0},... q _{{M-2}})$,
		such that
	\begin{align}	
 % %
							- \frac1 \tau ( p _{i} - p _{i-1} ) + a ( m'(\varphi_{i}), \mu _{i+1} , p _{i} ) - \mathop{\rm div}( p _{i} v _{i+1} ) - \Delta^t r _{i-1} 				\nonumber \\»\hskip1cm
								+ \Psi_0'' (\varphi _{i} )^* r _{i-1} 
								- \kappa r _{i+1} - \frac 1 \tau \rho' (\varphi_{i}) v _{i+1} \cdot( q _{i+1} - q _{i} )	 							\nonumber \\»\hskip1cm
 % %
								- (\rho' (\varphi_{i}) v _{i+1}  -\frac{\rho_2-\rho_1}{2}m'(\varphi_{i})\nabla \mu _{i+1} ) (D q _{i+1} )^\top v _{i+2} 									\nonumber \\»\hskip1cm
								+ 2\eta' (\varphi_{i}) \epsilon( v _{i+1} ) : D q _{i} + \mathop{\rm div}( \mu _{i+1} q _{i} )									 %\\»
					& \>=\>		\frac {\partial \mathcal{J}}{\partial \varphi _{i} }(\overline {z}),														\label{T:Mult.1}\\»
 % %
							- r _{i-1} + b ( m(\varphi_{i-1}) , p _{i-1} )													
								- \mathop{\rm div}( \frac{\rho_2-\rho_1}{2}m(\varphi_{i-1}) (D q _{i} )^\top v _{i+1} )												\nonumber \\»\hskip1cm
								- q _{i-1} \cdot\nabla \varphi _{i-1} 															 %\\»
					& \>=\>	\frac {\partial \mathcal{J}}{\partial \mu _{i} }(\overline {z}),															\label{T:Mult.2}\\»
 % %
 							- \frac 1 \tau \rho (\varphi_{j-1}) ( q _{j} - q _{j-1} ) - \rho (\varphi_{j-1}) (D q _{j} )^\top v _{j+1} 									\nonumber \\»\hskip1cm
								- (D q _{j-1} )( \rho (\varphi_{j-2}) v _{j-1}  -\frac{\rho_2-\rho_1}{2}m(\varphi_{j-2}) \nabla \mu _{j-1} )									\nonumber \\»\hskip1cm
								- \mathop{\rm div}( 2 \eta (\varphi_{j-1}) \epsilon( q _{j-1} ) ) + p _{j-1} \nabla \varphi _{j-1} 								 %\\»
					& \>=\>	\frac {\partial \mathcal{J}}{\partial v _{j} }(\overline {z}),															\label{T:Mult.3}\\»
					\Big(
							\frac {\partial \mathcal{J}}{\partial u _{k} }(\overline {z}) -  q _{k-1} 
					\Big)_{k=1}^{M-1}
					& \>\in\>	\big[\mathbb{R}_+( U _{ad}-\bar u )\big]^+,																		\label{T:Mult.4}
	\end{align}
	for all $i=0,...,{M-1}$ and $j=1,...,{M-1}$.
 Here, $\big[\mathbb{R}_+( U _{ad}-\bar u )\big]^+$ denotes the polar cone of the set $\left\{r(w-u)|w\in U_{ad} \wedge r\in\mathbb{R}^+\right\}$.
	Furthermore, we use the convention that $ p _{i} , r _{i} , q _{i} $ are equal to $0$ for $i\ge{M-1}$ along with $ q _{-1}$ and
		$\varphi _{i} ,\mu _{i} , v _{i} $ for $i\ge M $.
	Moreover, $ a (\hat f ,\hat w ,\hat p ), b (\hat m ,\hat p ), \Delta^t(\hat r ) \in \overline{H}^2_{\partial_n }(\Omega )^*$ are defined by
			$\langle \Delta^t\hat r , \hat z \rangle		:=		\int_\Omega \hat r \Delta\hat zdx$, 
	$\langle a (\hat f ,\hat w ,\hat p ) , \hat z \rangle		:=		\int_\Omega -\hat p \mathop{\rm div}(\hat f \hat z \nabla \hat w )dx$,				
			$\langle b (\hat m ,\hat p ) , \hat z \rangle		:=		\int_\Omega -\hat p \mathop{\rm div}(\hat m \nabla \hat z )dx$,						
 	% 	\begin{align*}
% 			&\langle a (\hat f ,\hat w ,\hat p ) , \hat z \rangle		:=		\int_\Omega -\hat p \mathop{\rm div}(\hat f \hat z \nabla \hat w )dx,				
% 			\ \langle b (\hat m ,\hat p ) , \hat z \rangle		:=		\int_\Omega -\hat p \mathop{\rm div}(\hat m \nabla \hat z )dx,\\						
% 			&\langle \Delta^t\hat r , \hat z \rangle		:=		\int_\Omega \hat r \Delta\hat zdx, 
% 	\end{align*}
	for functions $\hat f ,\hat m \in C^1(\overline \Omega ),\hat w \in H^1(\Omega ),\hat r ,\hat p \in L^2(\Omega )$ and $\hat z \in \overline{H}^2_{\partial_n }(\Omega )$.
 % %
\end{theorem}

\begin{proof}
	Utilizing the spaces $ X $ and $ Y $ and the set $»C$ given by
	\begin{align*}
			X 	&:= \overline{H}^2_{\partial_n }(\Omega )^ M \times \overline{H}^2_{\partial_n }(\Omega )^ M \times H^1_{0,\sigma}(\Omega ;\mathbb{R}^N )^{M-1} \times L^2(\Omega;\mathbb{R}^N)^{M-1}, 		\\
			»C	&:= \overline{H}^2_{\partial_n }(\Omega )^ M \times \overline{H}^2_{\partial_n }(\Omega )^ M \times H^1_{0,\sigma}(\Omega ;\mathbb{R}^N )^{M-1} \times U _{ad},	\\
			Y 	&:={( \overline{L}^2(\Omega ) )^*}^ M \times {( \overline{L}^2(\Omega ) )^*}^ M \times { H^1_{0,\sigma}(\Omega ;\mathbb{R}^N ) ^*}^{M-1}, 
	\end{align*}
	for
		$\varphi =(\varphi _0,...,\varphi _{M-1})$,
		$\mu =(\mu _0,...,\mu _{M-1})$, 
		$ v =( v _1,..., v _{M-1})$, 
		$ u =( u _1,..., u _{M-1})$
	we define a mapping $g: X \rightarrow Y $
		by
	\begin{align*}
		&	g(\varphi ,\mu , v , u )\\	
		&	:=\left(\begin{array}{l}
						\hspace{-0.1cm}\left(\begin{array}{l} 	 \frac1 \tau ( \varphi _{i+1} - \varphi _{i} ) - \mathop{\rm div}( m(\varphi_{i}) \nabla \mu _{i+1} ) + v _{i+1} \cdot\nabla \varphi _{i} 							\end{array}\right)_{i=-1}^{M-2}	\\[3pt]
						\hspace{-0.1cm}\left(\begin{array}{l} 	- \mu _{i+1} - \Delta \varphi _{i+1} + \Psi_0'( \varphi _{i+1}) - \kappa \varphi _{i} 									\end{array}\right)_{i=-1}^{M-2}	\\[3pt]
						\hspace{-0.1cm}\left(\begin{array}{l} 	\frac1 \tau ( \rho (\varphi_{i}) v _{i+1} - \rho (\varphi_{i-1}) v _{i} )- \mathop{\rm div}(2 \eta (\varphi_{i}) \epsilon( v _{i+1} ) )  							\\[3pt]%\hskip6mm 
								+ \mathop{\rm div}( v _{i+1} \otimes ( \rho (\varphi_{i-1}) v _{i}  -\frac{\rho_2-\rho_1}{2}m(\varphi_{i-1}) \nabla \mu _{i} ))- \mu _{i+1} \nabla \varphi _{i} - u _{i+1} 									\end{array}\right)_{i=0}^{M-2}	
					\hspace{-0.1cm}\end{array}\right).
	\end{align*}
% 	To simplify our discussion, here, we do not include
% 		the initial values $\varphi _a $ and $ v _a $ as components of $(\varphi ,\mu , v )$ in the space $ X $.
% 	Therefore, we use the notational convention that all occurrences of $\varphi _{-1} $ and $ v _{0}$ in $g$ and its derivative
% 		refer to $\varphi _a $ and $ v _a $, respectively.  % and the unique $\mu _0$ solving~((2)), respectively.
	Then, ($P_\Psi$) can be stated as
		%$\overline {z}=(\bar\varphi ,\bar\mu ,\bar v ,\bar u )$ is a minimizer of
	%\[
			$\min \{ 	\mathcal{J}(\varphi ,\mu , v , u ) :	(\varphi ,\mu , v , u )\in»C,\enspace 	g(\varphi ,\mu , v , u )=0		\}$,
	%\]
	with $\overline {z}=(\bar\varphi ,\bar\mu ,\bar v ,\bar u )$ an associated minimizer.
	The mapping $g$ is continuously Fr\`echet differentiable from $ X $ into $ Y $.
	To see this, let us exemplarily consider the term $\mathop{\rm div}( m(\varphi_{i}) \nabla \mu _{i+1} )$.  %$\mathop{\rm div}( \mu _{i+1} \otimes «mm\nabla \widetilde\mu )$.
	The other terms can be treated analogously.
	First note that $\mathop{\rm div}( m(\varphi_{i}) \nabla \mu _{i+1} )$ equals $\nabla m(\varphi_{i}) \cdot\nabla \mu _{i+1} + m(\varphi_{i}) \Delta \mu _{i+1} $
		where $ m(\varphi_{i}) $ is given by $ m ( \varphi _{i} )$.
	Hence $\nabla m(\varphi_{i}) = m '( \varphi _{i} )\nabla \varphi _{i} $.
	Assumption~\ref{assum1} implies that both superposition operators
	$
			\widetilde\varphi \mapsto m (\widetilde\varphi ),$ $\widetilde\varphi \mapsto m '(\widetilde\varphi )
	$
	are continuously Fr\`echet differentiable from $H^2(\Omega )\hookrightarrow L^\infty (\Omega )$ into $L^\infty (\Omega )$ (cf.~\cite{Troeltzsch2010}).
	Therefore, the mappings
	\def\xa#1#2#3{#1 & \enspace \mapsto \enspace \hbox to 2.4cm{$#2$\hss}:\enspace #3}
				$(\widetilde\varphi ,\widetilde\mu )\rightarrow		m '(\widetilde\varphi )\nabla \widetilde\varphi \cdot\nabla \widetilde\mu :		H^2(\Omega )\times H^2(\Omega )\rightarrow L^3(\Omega )	 $ and
				$(\widetilde\varphi ,\widetilde\mu )\rightarrow		m (\widetilde\varphi )\Delta\widetilde\mu 				:H^2(\Omega )\times H^2(\Omega )\rightarrow L^2(\Omega )$,				 
% 	\begin{align*}
% 			\xa{	(\widetilde\varphi ,\widetilde\mu )	}{	m '(\widetilde\varphi )\nabla \widetilde\varphi \cdot\nabla \widetilde\mu 	}{	H^2(\Omega )\times H^2(\Omega )\rightarrow L^3(\Omega ),	 }\\
% 			\xa{	(\widetilde\varphi ,\widetilde\mu )	}{	m (\widetilde\varphi )\Delta\widetilde\mu 			}{	H^2(\Omega )\times H^2(\Omega )\rightarrow L^2(\Omega ),				 }
% 	\end{align*}
	are continuously Fr\`echet differentiable.
	This shows the continuous Fr\`echet differentiability of $\mathop{\rm div}( m(\varphi_{i}) \nabla \mu _{i+1} )$.
 % %
	The Fr\`echet derivative of $g$ in $(\varphi ,\mu , v , u )$ applied to $(\varphi ^\delta ,\mu ^\delta , v ^\delta , u ^\delta )\in X $ is given by
	\begin{align*}
		&	g'(\varphi ,\mu , v , u )(\varphi ^\delta ,\mu ^\delta , v ^\delta , u ^\delta )													\\%\label{T:Mult.b1}\\
		&	=\left(\begin{array}{l}
						\hspace{-0.1cm}\left(\begin{array}{l} 	 \frac1 \tau ( \varphi ^\delta _{i+1} - \varphi ^\delta _{i} ) - \mathop{\rm div}( m'(\varphi_{i}) \varphi ^\delta _{i} \nabla \mu _{i+1} ) - \mathop{\rm div}( m(\varphi_{i}) \nabla \mu ^\delta _{i+1} )					\\[5pt]\hskip6mm 
								+ v _{i+1} \cdot\nabla \varphi ^\delta _{i} + v ^\delta _{i+1} \cdot\nabla \varphi _{i} 									\end{array}\right)_{i=-1}^{M-2}	\\[5pt]
						\hspace{-0.1cm}\left(\begin{array}{l} 	- \mu ^\delta _{i+1} - \Delta \varphi ^\delta _{i+1} + \Psi_0'' ( \varphi _{i+1} ; \varphi ^\delta _{i+1}) - \kappa \varphi ^\delta _{i}							\end{array}\right)_{i=-1}^{M-2}	\\[5pt]
						\hspace{-0.1cm}\left(\begin{array}{l} 	\frac1 \tau ( \rho' (\varphi_{i}) \varphi ^\delta _{i} v _{i+1} - \rho' (\varphi_{i-1}) \varphi ^\delta _{i-1} v _{i} ) + \frac1 \tau ( \rho (\varphi_{i}) v ^\delta _{i+1} - \rho (\varphi_{i-1}) v ^\delta _{i} )											\\[5pt]\hskip6mm 
 % %
									+	\mathop{\rm div}( v _{i+1} \otimes (	 \rho' (\varphi_{i-1}) \varphi ^\delta _{i-1} v _{i} + \rho (\varphi_{i-1}) v ^\delta _{i} 			))					\\[5pt]\hskip6mm 
									-	\mathop{\rm div}( v _{i+1} \otimes (	 \frac{\rho_2-\rho_1}{2}m'(\varphi_{i-1}) \varphi ^\delta _{i-1} \nabla \mu _{i} -\frac{\rho_2-\rho_1}{2}m(\varphi_{i-1}) \nabla \mu ^\delta _{i} ))					\\[5pt]\hskip6mm 
								+ \mathop{\rm div}( v ^\delta _{i+1} \otimes ( \rho (\varphi_{i-1}) v _{i}  -\frac{\rho_2-\rho_1}{2}m(\varphi_{i-1}) \nabla \mu _{i} ))									\\[5pt]\hskip6mm 
								- \mathop{\rm div}(2 \eta' (\varphi_{i}) \varphi ^\delta _{i} \epsilon( v _{i+1} ) )
								- \mathop{\rm div}(2 \eta (\varphi_{i}) \epsilon( v ^\delta _{i+1} ) )																\\[5pt]\hskip6mm 
								- \mu _{i+1} \nabla \varphi ^\delta _{i} - \mu ^\delta _{i+1} \nabla \varphi _{i} 
								- u ^\delta _{i+1} 																		\end{array}\right)_{i=0}^{M-2}	\nonumber 
					\hspace{-0.1cm}\end{array}\right).
	\end{align*}
	Due to our convention for $\varphi _{-1}$ and $ v _{0}$,
		we require that $\varphi ^\delta _{-1}=0$ and $ v ^\delta _{0}=0$.
	%In order to be able to apply the existence result of Lagrange multipliers by
	For the application of a result due to
		Zowe and Kurcyusz~\cite{Zowe1979} concerning the existence of Lagrange multipliers, we show that $g'(\overline {z})$ maps $\mathbb{R}_+(C-\overline {z})\subset X $ onto $ Y $.
	For this purpose, let $(\Theta^c_{i},\Theta^w_{i},\Theta^v_{i})\in Y $ be arbitrarily fixed.
	We have to show that there exists a tuple $(\varphi ^\delta ,\mu ^\delta , v ^\delta , u ^\delta )\in\mathbb{R}_+(C-\overline {z})$ such that
	\begin{align}
				 \frac1 \tau ( \varphi ^\delta _{i+1} - \varphi ^\delta _{i} ) - \mathop{\rm div}( m'(\varphi_{i}) \varphi ^\delta _{i} \nabla \mu _{i+1} ) - \mathop{\rm div}( m(\varphi_{i}) \nabla \mu ^\delta _{i+1} )							\nonumber \\[3pt]
								+ v _{i+1} \cdot\nabla \varphi ^\delta _{i} + v ^\delta _{i+1} \cdot\nabla \varphi _{i} 					&=	 \Theta^w_{i}, 	\label{T:Mult.a1}\\[3pt]
				- \mu ^\delta _{i+1} - \Delta \varphi ^\delta _{i+1} - \kappa \varphi ^\delta _{i} + \Psi_0'' ( \varphi _{i+1} ; \varphi ^\delta _{i+1} )						&=	 \Theta^c_{i}, 	\label{T:Mult.a2}\\[3pt]
% 	\end{align}
% 	\begin{align}
				\frac1 \tau ( \rho' (\varphi_{i}) \varphi ^\delta _{i} v _{i+1} - \rho' (\varphi_{i-1}) \varphi ^\delta _{i-1} v _{i} ) + \frac1 \tau ( \rho (\varphi_{i}) v ^\delta _{i+1} - \rho (\varphi_{i-1}) v ^\delta _{i} )													\nonumber \\[3pt]
								+	\mathop{\rm div}( v _{i+1} \otimes (	 \rho' (\varphi_{i-1}) \varphi ^\delta _{i-1} v _{i} + \rho (\varphi_{i-1}) v ^\delta _{i} 			))					\nonumber \\[3pt]
								-	\mathop{\rm div}( v _{i+1} \otimes (	 \frac{\rho_2-\rho_1}{2}m'(\varphi_{i-1}) \varphi ^\delta _{i-1} \nabla \mu _{i}  -\frac{\rho_2-\rho_1}{2}m(\varphi_{i-1}) \nabla \mu ^\delta _{i} ))					\nonumber \\[3pt]
								+ \mathop{\rm div}( v ^\delta _{i+1} \otimes ( \rho (\varphi_{i-1}) v _{i} -\frac{\rho_2-\rho_1}{2}m(\varphi_{i-1}) \nabla \mu _{i} ))								\nonumber \\[3pt]
								- \mathop{\rm div}(2 \eta' (\varphi_{i}) \varphi ^\delta _{i} \epsilon( v _{i+1} ) )
								- \mathop{\rm div}(2 \eta (\varphi_{i}) \epsilon( v ^\delta _{i+1} ) )															\nonumber \\[3pt]
								- \mu _{i+1} \nabla \varphi ^\delta _{i} - \mu ^\delta _{i+1} \nabla \varphi _{i} 
								- u ^\delta _{i+1} 														&=	 \Theta^v_{i} ,	\label{T:Mult.a3}
	\end{align}
	where (\ref{T:Mult.a1}) and (\ref{T:Mult.a2}) hold for $i=-1,...,{M-2}$ and (\ref{T:Mult.a3}) for all $i=0,...,{M-1}$.
	As in~Theorem~\ref{Exist!}, standard arguments show the existence of 
		 $(\varphi ^\delta _{0} ,\mu ^\delta _{0} )\in \overline{H}^2_{\partial_n }(\Omega )\times \overline{H}^2_{\partial_n }(\Omega )$ such that~(\ref{T:Mult.a1}) and~(\ref{T:Mult.a2}) are fulfilled for $i=-1$.
	Now we apply induction over $i$.
	Therefore, let us assume that (\ref{T:Mult.a1})--(\ref{T:Mult.a3}) hold for $i<M-1$.
	In order to show the existence of a solution to this system for $i+1$, we note that it can be written as
	\begin{align*} 
				 \frac1 \tau ( \varphi ^\delta _{i+2} - \varphi ^\delta _{i+1} ) - \mathop{\rm div}( m(\varphi_{i+1}) \nabla \mu ^\delta _{i+2} ) + v ^\delta _{i+2} \cdot\nabla \varphi _{i+1} 						&=	\Theta_\mu, 		\\[3pt]
				- \mu ^\delta _{i+2} - \Delta \varphi ^\delta _{i+2} - \kappa \varphi ^\delta _{i+1} + \Psi_0'' ( \varphi _{i+2} ; \varphi ^\delta _{i+2} )						&=	\Theta_\varphi, 		\\[3pt]
				\frac1 \tau ( \rho (\varphi_{i+1}) v ^\delta _{i+2} - \rho (\varphi_{i}) v ^\delta _{i+1} )
								+ \mathop{\rm div}( v ^\delta _{i+2} \otimes ( \rho (\varphi_{i}) v _{i+1}  -\frac{\rho_2-\rho_1}{2}m(\varphi_{i}) \nabla \mu _{i+1} ))								\nonumber \\[3pt]
								- \mathop{\rm div}(2 \eta (\varphi_{i+1}) \epsilon( v ^\delta _{i+2} ) )
								- \mu ^\delta _{i+2} \nabla \varphi _{i+1} 
								- u ^\delta _{i+2} 														&=	\Theta_v ,
	\end{align*}
		for a triple $( \Theta_\varphi , \Theta_\mu , \Theta_v )\in( \overline{L}^2(\Omega ) )^*\times ( \overline{L}^2(\Omega ) )^*\times H^1_{0,\sigma}(\Omega ;\mathbb{R}^N ) ^* $ that only depends on $(\varphi ,\mu , v )$, on $\varphi ^\delta _{i} ,\mu ^\delta _{i} $ and $ v ^\delta _{i} $ for $i<M-1$
		and on $ (\Theta^c_{i+1},\Theta^w_{i+1},\Theta^v_{i+1}) $.
	But now the existence of a solution follows readily from Theorem~\ref{Exist!} and from Lemma~\ref{regsol}
		when choosing
 		 $\nu= \rho (\varphi_{i}) v _{i+1}  -\frac{\rho_2-\rho_1}{2}m(\varphi_{i}) \nabla \mu _{i+1} $ as well as $f_0=\rho (\varphi_{i+1}),\> f_{-1}= \rho (\varphi_{i}) $ and $ u ^\delta _{i+2} =0$.
	Notice, here the functions $\rho (\varphi_{i+1}), m(\varphi_{i+1}),\eta (\varphi_{i+1})$ do not depend on the unknown $ \varphi ^\delta _{i+2} $.
	Further observe that we can always find a convex, affine functional $\psi: \overline{H}^2_{\partial_n }(\Omega )\mapsto \mathbb{R}$ with
		$(D\psi)z = \Psi_0'' ( \varphi _{i+2};z )$ for all $z\in \overline{H}^2_{\partial_n }(\Omega )$.
	Hence we deduce the existence of a Lagrange multiplier $( p,r,q )\in Y ^*$ such that
	\begin{align}
		\mathcal{J}'(\bar\varphi ,\bar\mu ,\bar v ,\bar u )(\varphi ^\delta ,\mu ^\delta , v ^\delta , u ^\delta )
			& =		\langle g'(\bar\varphi ,\bar\mu ,\bar v ,\bar u )(\varphi ^\delta ,\mu ^\delta , v ^\delta , u ^\delta ) , ( p,r,q ) \rangle		\nonumber \\
			& =		\langle g'(\bar\varphi ,\bar\mu ,\bar v ,\bar u )^*( p,r,q ) , (\varphi ^\delta ,\mu ^\delta , v ^\delta , u ^\delta ) \rangle		\label{T:Mult.c1}
	\end{align}
	for all $(\varphi ^\delta ,\mu ^\delta , v ^\delta , u ^\delta )\in \overline{H}^2_{\partial_n }(\Omega )^ M \times \overline{H}^2_{\partial_n }(\Omega )^ M \times H^1_{0,\sigma}(\Omega ;\mathbb{R}^N )^{M-1} \times \mathbb{R}_+( U _{ad}-\bar u )$.
	In order to derive the desired system for $( p,r,q )$ from this variational equation, the adjoint of $g'(\bar\varphi ,\bar\mu ,\bar v ,\bar u )$
		has to be calculated.
	Exemplarily, we show this calculation for two terms. %that belong to the last block in equation~(\ref{T:Mult.b1}).
	First, consider the term $\mathop{\rm div}( v _{i+1} \otimes (	 \rho' (\varphi_{i-1}) \varphi ^\delta _{i-1} v _{i} ))$ which gets tested by $ q _{i} $.
	Notice that for vector fields $ z^{(1)} , z^{(2)} , z^{(3)} $ in $H^1(\Omega ;\mathbb{R}^N )$ and with $ z^{(2)} |_{\partial \Omega }=0$ we have
	\begin{align}
			\int_\Omega z^{(3)} \cdot\mathop{\rm div}( z^{(2)} \otimes z^{(1)} )
% 			&	=	\sum_{i,j=1}^N	\int_\Omega 		z^{(3)} _i \partial _j( z^{(2)} _i z^{(1)} _j)							\nonumber\\
% 			&	=	\sum_{i,j=1}^N	\int_\Omega 		\partial _j( z^{(3)} _i z^{(2)} _i z^{(1)} _j) - z^{(2)} _i z^{(1)} _j\partial _j z^{(3)} _i		\nonumber\\
% 			&	=	-\sum_{i,j=1}^N	\int_\Omega 		z^{(2)} _i z^{(1)} _j\partial _j z^{(3)} _i								\nonumber\\
			&	=	- \int_\Omega 					z^{(2)} \cdot(D z^{(3)} ) z^{(1)} ,		\label{T:Mult.x1}
	\end{align}
	by Gau\ss' theorem.
 % %
	Hence we get
	\begin{align*}
			\langle \mathop{\rm div}( v _{i+1} \otimes \rho' (\varphi_{i-1}) \varphi ^\delta _{i-1} v _{i} ) , q _{i} \rangle
			&=	- \int_\Omega 	 v _{i+1} \cdot (D q _{i} )( \rho' (\varphi_{i-1}) \varphi ^\delta _{i-1} v _{i} )dx	\\
			&=	- \int_\Omega 	 \rho' (\varphi_{i-1}) \varphi ^\delta _{i-1} v _{i} \cdot (D q _{i} )^\top v _{i+1}dx .
	\end{align*}
	Secondly, the term $\mathop{\rm div}( v _{i+1} \otimes -\frac{\rho_2-\rho_1}{2}m(\varphi_{i-1}) \nabla \mu ^\delta _{i} )$ gets tested by $ q _{i} $.
	This yields
	\begin{align*}
			\langle \mathop{\rm div}( v _{i+1} \otimes -\frac{\rho_2-\rho_1}{2}m(\varphi_{i-1}) \nabla \mu ^\delta _{i} ) , q _{i} \rangle
			&=	 \int_\Omega 	 v _{i+1} \cdot (D q _{i} )( \frac{\rho_2-\rho_1}{2}m(\varphi_{i-1}) \nabla \mu ^\delta _{i} )dx		\\
			&=	 \int_\Omega 	 \frac{\rho_2-\rho_1}{2}m(\varphi_{i-1}) \nabla \mu ^\delta _{i} \cdot (D q _{i} )^\top v _{i+1}dx 		\\
			&=	\int_\Omega 		\mu ^\delta _{i} \mathop{\rm div}( -\frac{\rho_2-\rho_1}{2}m(\varphi_{i-1}) (D q _{i} )^\top v _{i+1} )dx
	\end{align*}
	since $ v _{i+1} |_{\partial \Omega }=0$.
	The other terms can be treated similarly.
	After collecting all terms which contain $\varphi ^\delta _{i} $, $\mu ^\delta _{i} $ and $ v ^\delta _{i} $, respectively, it follows that
	\begin{align*}	&	g'(\bar\varphi ,\bar\mu ,\bar v ,\bar u )^*( p,r,q ) 						\\
			&	=\left(\begin{array}{l}
						\left(\begin{array}{l} 
							- \frac1 \tau ( p _{i} - p _{i-1} ) + a ( m'(\varphi_{i}), \mu _{i+1} , p _{i} ) - \mathop{\rm div}( p _{i} v _{i+1} ) - \Delta^t r _{i-1} 				\\[3pt]\hskip6mm 
								+ \Psi_0'' (\varphi _{i} )^* r _{i-1} 
								- \kappa r _{i+1} - \rho' (\varphi_{i}) v _{i+1} \cdot \frac1 \tau ( q _{i+1} - q _{i} ) 		 									\\[3pt]\hskip6mm 
								- (\rho' (\varphi_{i}) v _{i+1}  -\frac{\rho_2-\rho_1}{2}m'(\varphi_{i})\nabla \mu _{i+1} ) (D q _{i+1} )^\top v _{i+2} 									\\[3pt]\hskip6mm 
								+ 2\eta' (\varphi_{i}) \epsilon( v _{i+1} ) : D q _{i} + \mathop{\rm div}( \mu _{i+1} q _{i} )									\end{array}\right)_{i={0}}^{{M-1}}\\[3pt]
						\vrule width0pt height4.5ex
						\left(\begin{array}{l} 
							- r _{i-1} + b ( m(\varphi_{i-1}) , p _{i-1} )													
								- \mathop{\rm div}( \frac{\rho_2-\rho_1}{2}m(\varphi_{i-1}) (D q _{i} )^\top v _{i+1} )												\\[3pt]\hskip6mm 
								- q _{i-1} \cdot\nabla \varphi _{i-1} 															\end{array}\right)_{i={1}}^{{M-1}}\\[3pt]
						\left(\begin{array}{l} 	- \rho (\varphi_{i-1}) \frac1 \tau ( q _{i} - q _{i-1} ) - \rho (\varphi_{i-1}) (D q _{i} )^\top v _{i+1} 											\\[3pt]\hskip6mm 
								- (D q _{i-1} )( \rho (\varphi_{i-2}) v _{i-1}  -\frac{\rho_2-\rho_1}{2}m(\varphi_{i-2}) \nabla \mu _{i-1} )									\\[3pt]\hskip6mm 
								- \mathop{\rm div}(2 \eta (\varphi_{i-1}) \epsilon( q _{i-1} ) ) + p _{i-1} \nabla \varphi _{i-1} 											\end{array}\right)_{i={1}}^{{M-1}}\\[3pt]
						\left(\begin{array}{l} 	-  q _{i-1} 																			\end{array}\right)_{i={1}}^{{M-1}}
				\end{array}\right).
	\end{align*}
	Plugging this into~(\ref{T:Mult.c1}) and using the fact that $(\varphi ^\delta ,\mu ^\delta , v ^\delta , u ^\delta )$ can be chosen arbitrarily in
		$ \overline{H}^2_{\partial_n }(\Omega )^ M \times \overline{H}^2_{\partial_n }(\Omega )^ M \times H^1_{0,\sigma}(\Omega ;\mathbb{R}^N )^{M-1} \times \mathbb{R}_+( U _{ad}-\bar u )$, we obtain the desired system for $( p,r,q )$.
\end{proof}
 % %

The preceding theorem states first-order optimality conditions for problem~($P_\Psi$) in the case of smooth double-well type potentials.
In the following, we derive stationarity conditions for a nonsmooth potential via a limit process; compare section 7.
For this purpose,
% 	by a limit process involving the results from Theorem \ref{T:Mult}.
% In the following, we derive stationarity conditions for problem~($P_\Psi$) in the case of the nonsmooth double-obstacle potential $\Psi_0 $ 
% 	by a limit process involving the results from Theorem \ref{T:Mult}.
% 	%where we replace $\Psi_0 $ by a family of smooth approximations.
% In order to pass to the limit within the first-order optimality conditions, % for the respective approximations,
the boundedness of the adjoint states is crucial.
In order to guarantee this, further regularity conditions on $\mathcal{J}$ are required.

\begin{lemma}\label{L:Regul}
	Suppose that the assumptions of Theorem~\ref{T:Mult} are fulfilled.
	%Using the same notation for $( p,r,q )$ and for $\overline {z}$,
% 	We further assume that
% 	\begin{align}
% 			\frac {\partial \mathcal{J}}{\partial \varphi _i}(\overline {z}), \> \frac {\partial \mathcal{J}}{\partial \mu _j}(\overline {z}) \>\in\> \overline{H}^1(\Omega )^*		\label{L:Regul.1}
% 	\end{align}
% 		for all $i=0,...,{M-1}$ and $j=1,...,{M-1}$.
	Then $( p , r )\in \overline{H}^1(\Omega )^ M \times \overline{H}^1(\Omega )^{M-1}$ and it holds that
	\begin{align*}
			a ( m'(\varphi_{i}), \mu _{i+1} , p _{i} )	&\>=\>		 m'(\varphi_{i})\nabla \mu _{i+1} \cdot\nabla p _{i} 	&\>\in\>	 \overline{H}^1(\Omega )^*,	\\
			b ( m(\varphi_{i-1}) , p _{i-1} )	&\>=\>		-\mathop{\rm div}( m(\varphi_{i-1}) \nabla p _{i-1} )		&\>\in\>	 \overline{H}^1(\Omega )^*,	\\
			-\Delta^t r _{i-1} 			&\>=\>		-\Delta r _{i-1} 				&\>\in\>	 \overline{H}^1(\Omega )^*.
	\end{align*}
\end{lemma}

\begin{proof}
	We prove the claim by backward induction over $i$.
	For $i={M-1}$ we have $ p _{M-1}= r _{M-1}=0$ by convention.
	Now, we take the induction step from $i$ to $i-1$ assuming that $ p _{i} , r _{i} \in \overline{H}^1(\Omega )$.
	This higher regularity implies for $\hat z \in \overline{H}^2_{\partial_n }(\Omega )$ that
	\begin{align*}
			\langle a ( m'(\varphi_{i}), \mu _{i+1} , p _{i} ) , \hat z \rangle
			&\>=\>		-\int_\Omega 	 p _{i} \mathop{\rm div}( m'(\varphi_{i})\hat z \nabla \mu _{i+1} )dx		\\
			&\>=\>		\int_\Omega 		 m'(\varphi_{i})\hat z \nabla \mu _{i+1} \cdot\nabla p _{i}dx 		\\
			&\>\le\>				C || m'(\varphi_{i})||_{L^\infty } ||\nabla \mu _{i+1} ||_{L^4} ||\nabla p _{i} ||_{L^2} ||\hat z ||_{L^4}	\\
			&\>\le\>				C || m'(\varphi_{i})||_{L^\infty } || \mu _{i+1} ||_{H^2} || p _{i} ||_{H^1} ||\hat z ||_{H^1}
	\end{align*}
		because of $\nabla \mu _{i+1} \cdot\vec n =0$ on $\partial \Omega $.
	Consequently, $ a ( m'(\varphi_{i}), \mu _{i+1} , p _{i} ) \in \overline{H}^1(\Omega )^*$.
	Equations~(\ref{T:Mult.1}) and~(\ref{T:Mult.2}) and the assumption yield that
	%\[
			$\Delta^t r _{i-1} ,\>	b ( m(\varphi_{i-1}) , p _{i-1} )		\>\in\>		 \overline{H}^1(\Omega )^*$.
	%\]
	By standard regularity arguments one shows that $ r _{i-1} $ and $ p _{i-1} $ are indeed elements of $ \overline{H}^1(\Omega )$
		and the desired relations for $ b ( m(\varphi_{i-1}) , p _{i-1} )$ and $\Delta^t r _{i-1} $ follow at once.
\end{proof}

The next lemma is used in the subsequent theorem in order to prove the boundedness of the adjoint state.

\def\xa#1#2#3{#1\hbox to 1cm{\hss $#2$\hss}& #3}

\begin{lemma}\label{L:Beschr}
 % %
	Let $\alpha >0$ be given and $ M_1 $ and $ M_2 $ be bounded subsets of $ \overline{H}^1(\Omega )^*$ and $ H^1_{0,\sigma}(\Omega ;\mathbb{R}^N )^*$, respectively.
	Let ${\cal M}$ be the set of all tuples 
		$(\hat p ,\hat r ,\hat q ;$\hskip0pt $\hat A;$\hskip0pt $ h_ p , h_ r , h_ q ;$\hskip0pt $\hat c ,\hat u ;$\hskip0pt $\hat m ,\hat\eta ,\hat\rho )$ with
	\begin{align*}
			\xa{	(\hat p ,\hat r ,\hat q )			}{	\in	}{	 \overline{H}^1(\Omega )\times \overline{H}^1(\Omega )\times H^1_{0,\sigma}(\Omega ;\mathbb{R}^N ),								}\\
			\xa{	\hat A				}{	\in	}{	{\cal L}( \overline{H}^1(\Omega ); \overline{H}^1(\Omega )^*) \mbox{ be monotone},			}\\
			\xa{	( h_ r , h_ p , h_ q )	}{	\in	}{	M_1 \times M_1 \times M_2 ,							}\\
			\xa{	(\hat c ,\hat u )			}{	\in	}{	 \overline{H}^1(\Omega )\times H^1(\Omega ;\mathbb{R}^N ),						}\\
			\xa{	\hat m ,\hat\eta ,\hat\rho 		}{	\in	}{	L^\infty (\Omega ) \mbox{ with $1/\alpha \ge\hat m ,\hat\eta \ge\alpha $ and $\hat\rho \ge0$ a.e. on $\Omega $},	}
	\end{align*}
		for which the following system is satisfied:
	\begin{align}
			\frac 1 \tau \hat p 	- \Delta \hat r + \hat A\hat r 											&	= h_ r ,		\label{L:Beschr.1}\\
			- \hat r 	 - \mathop{\rm div}(\hat m \nabla \hat p ) 		 - \hat q \cdot\nabla \hat c 	&	= h_ p ,		\label{L:Beschr.2}\\
			\frac 1 \tau \hat\rho \hat q 			- \mathop{\rm div}(2\hat\eta \epsilon(\hat q ) ) - (D \hat q ) \hat u + \hat p \nabla \hat c 	&	= h_ q ,		\label{L:Beschr.3}\\
			\frac 1 \tau \int_\Omega \hat\rho |\hat q |^2	- \langle (D \hat q ) \hat u ,\hat q \rangle		 							&	\ge 0.		\label{L:Beschr.4}
	\end{align}
	Then the set $\{(\hat p ,\hat r ,\hat q ) \>:\> (\hat p ,\hat r ,\hat q ;$\hskip0pt $\hat A;$\hskip0pt $ h_ p , h_ r , h_ q ;$\hskip0pt $\hat c ,\hat u ;$\hskip0pt $\hat m ,\hat\eta ,\hat\rho )\in{\cal M}\}$ is bounded in $ \overline{H}^1(\Omega )\times \overline{H}^1(\Omega )\times H^1_{0,\sigma}(\Omega ;\mathbb{R}^N )$.
\end{lemma}

In order to keep the flow of the presentation, we defer the proof to the appendix.
% \begin{proof}
%  %
% 	Testing~(\ref{L:Beschr.1})--(\ref{L:Beschr.3}) by $\tau\hat r $, $\hat p $ and $\hat q $, respectively, and summing up we get
% 	\begin{align*}
% 		&		\tau \langle h_ r ,\hat r \rangle + \langle h_ p ,\hat p \rangle + \langle h_ q ,\hat q \rangle										\\
%  %
% 		&=		\tau \langle \nabla \hat r ,\nabla \hat r \rangle +\tau\langle \hat A\hat r ,\hat r \rangle + \langle \hat m \nabla \hat p ,\nabla \hat p \rangle						\\&\hskip6mm 
% 					+	\frac 1 \tau \langle \hat\rho \hat q ,\hat q \rangle - \langle (D \hat q ) \hat u ,\hat q \rangle
% 					+	\langle 2\hat\eta \epsilon(\hat q ) ,\epsilon(\hat q )\rangle 													\\
% 		&\ge	\tau ||\hat r ||_{ \overline{H}^1(\Omega )}^2 + C\Big( ||\hat p ||_{ \overline{H}^1(\Omega )}^2 + ||\hat q ||_{ H^1_{0,\sigma}(\Omega ;\mathbb{R}^N )}^2 \Big)
% 	\end{align*}
% 	for a positive constant $C$ depending only on $\alpha $ and on the constants in Korn's and Poincar\'e's inequalities.
% 	This inequality implies the assertion.
% \end{proof}\\ 

\def\xa#1#2#3#4{\hbox to 1cm{\hss $#1$}\hbox to 1cm{\hss $#2$\hss}\hbox to 1cm{$#3$\hss}\hbox to 10cm{#4\hss}}
\def\xb#1#2#3#4{\hbox to 4.5cm{\hss $#1#2#3$\hss}\hbox to 8.5cm{#4\hss}}

Employing the preceding results, we finally perform the limit process with respect to the first-order optimality conditions of Theorem \ref{T:Mult}.

\begin{theorem}[Stationarity conditions]\label{T:MultOrg}
	Suppose that the following assumptions are satisfied.
	\begin{enumerate}
	 \item $\mathcal{J}'$ is a bounded mapping from $ \overline{H}^1(\Omega )^ M \times \overline{H}^1(\Omega )^ M \times H^1_{0,\sigma}(\Omega ;\mathbb{R}^N )^{M-1} \times U _{ad}$
	into the space $({ \overline{H}^1(\Omega )}^ M \times { \overline{H}^1(\Omega )}^ M \times { H^1_{0,\sigma}(\Omega ;\mathbb{R}^N )}^{M-1} \times L^2(\Omega;\mathbb{R}^N)^{M-1} )^*$
		and $\frac{\partial \mathcal{J}}{\partial u}$ satisfies the following weak lower-semicontinuity property
		\vspace{-0.1cm}
	\[
		\Big\langle \frac{\partial \mathcal{J}}{\partial u}(\hat z), \hat u \Big\rangle
		\>\le\>	\liminf_{n\rightarrow\infty}
				\Big\langle \frac{\partial \mathcal{J}}{\partial u}(\hat z^{(n)}), \hat u^{(n)} \Big\rangle
	\]
	\vspace{-0.1cm}
	for $\hat z^{(n)}=(\hat\varphi ^{(n)} , \hat\mu ^{(n)} , \hat v ^{(n)} , \hat u ^{(n)})$ converging weakly in 
		$\overline{H}^2_{\partial_n }(\Omega )^ M \times \overline{H}^2_{\partial_n }(\Omega )^ M \times H^1_{0,\sigma}(\Omega ;\mathbb{R}^N )^{M-1} \times U_{ad} $
		to $\hat z=(\hat\varphi , \hat\mu , \hat v , \hat u )$.
%		and that $ v _{0}\in H^{2}_{0,\sigma}(\Omega;\mathbb{R}^N)$.
 %
	\item For every $n\in\mathbb{N}$
		let $ \Psi_0 ^{(n)} : \overline{H}^2_{\partial_n }(\Omega )\rightarrow \overline {\mathbb{R}}$ be a convex, lower-semicontinuous and proper functional satisfying the assumptions of Theorem~\ref{T:Mult}.
		%and let $\ \Psi_0' ^{(n)} :=\partial \Psi_0 ^{(n)} $.
 %
	\item Let $( \varphi ^{(n)} , \mu ^{(n)} , v ^{(n)} , u ^{(n)} )\in \overline{H}^2_{\partial_n }(\Omega )^ M \times \overline{H}^2_{\partial_n }(\Omega )^ M \times H^1_{0,\sigma}(\Omega ;\mathbb{R}^N )^{M-1} \times U_{ad} $ be a minimizer for (P$_{ \Psi ^{(n)} }$)
		and let $( p ^{(n)} , r ^{(n)} , q ^{(n)} )\in{ \overline{H}^1(\Omega )}^ M \times { \overline{H}^1(\Omega )}^ M \times { H^1_{0,\sigma}(\Omega ;\mathbb{R}^N )}^{M-1} $ be given as in Theorem~\ref{T:Mult} and Lemma~\ref{L:Regul}.
	\end{enumerate}

	Then there exists an element $( \varphi , \mu , v , u , p , r , q )$ and a subsequence denoted by $\left\{( \varphi ^{(m)} , \mu ^{(m)} , v ^{(m)} , u ^{(m)} , p ^{(m)} , r ^{(m)} , q ^{(m)} )\right\}_{m\in\mathbb{N}}$
	 with
	 \vspace{-0.1cm}
	\begin{align*}
			 \varphi ^{(m)} 		{	\rightarrow 	}{	 \varphi 		}\text{	weakly in } \overline{H}^2_{\partial_n }(\Omega )^ M , % strongly in $ \overline{H}^1(\Omega )^ M $ and $L^\infty (\Omega )^ M $,				}\\
			 &\ \mu ^{(m)} 		{	\rightarrow 	}{	 \mu 		}\text{	weakly in } \overline{H}^2_{\partial_n }(\Omega )^{M-1},\\% strongly in $ \overline{H}^1(\Omega )^{M-1}$ and $L^\infty (\Omega )^{M-1}$,					}\\
			 v ^{(m)} 		{	\rightarrow 	}{	 v 		}\text{	weakly in }H^2(\Omega ;\mathbb{R}^N )^{M-1},% strongly in $ H^1_{0,\sigma}(\Omega ;\mathbb{R}^N )^{M-1}$ and $L^\infty (\Omega ;\mathbb{R}^N )^{M-1}$,	}\\
			 &\ u ^{(m)} 		{	\rightarrow 	}{	 u 		}\text{	weakly in } L^2(\Omega;\mathbb{R}^N)^{M-1},									\\
			 p ^{(m)} 		{	\rightarrow 	}{	 p 		}\text{	weakly in } \overline{H}^1(\Omega )^ M ,								
			 &\ r ^{(m)} 		{	\rightarrow 	}{	 r 		}\text{	weakly in } \overline{H}^1(\Omega )^{M-1},									\\
			 q ^{(m)} 		{	\rightarrow 	}{	 q 		}\text{	weakly in } H^1_{0,\sigma}(\Omega ;\mathbb{R}^N )^{M-1},% strongly in $L^2_\sigma(\Omega;\mathbb{R}^N)$,									}\\
		 &\ {\Psi_0  ^{(m)}}'' ( \varphi ^{(m)}_{i+1} )^* r ^{(n)}_{i} 	{	\rightarrow 	}{	 \lambda _{i} 		}\text{	weakly in } \overline{H}^1(\Omega )^*,									
	\end{align*}
	\vspace{-0.1cm}
	 for all $i=-1,...,{M-2}$
	such that for $ z =( \varphi , \mu , v , u )$ and $\tilde q_k:=q_{k-1}$ it holds that
	\begin{align}	
							- \frac1 \tau ( p _{i} - p _{i-1} ) + { m(\varphi_{i}) }'\nabla \mu _{i+1} \cdot p _{i} - \mathop{\rm div}( p _{i} v _{i+1} ) - \Delta r _{i-1} 			\nonumber \\»\hskip6mm 
								+ \lambda _{i-1} 
								- \kappa r _{i+1} - \frac 1 \tau { \rho (\varphi_{i}) }' v _{i+1} \cdot( q _{i+1} - q _{i} )							\nonumber \\»\hskip6mm 
								- ({ \rho (\varphi_{i}) }' v _{i+1} -\frac{\rho_2-\rho_1}{2}m'(\varphi_{i})\nabla \mu _{i+1} ) (D q _{i+1} )^\top v _{i+2} 								\nonumber \\»\hskip6mm 
								+ 2{ \eta (\varphi_{i}) }' \epsilon( v _{i+1} ) : D q _{i} + \mathop{\rm div}( \mu _{i+1} q _{i} )								 %\\»
					& \>=\>		\frac {\partial \mathcal{J}}{\partial \varphi _{i} }( z ),														\label{T:MultOrg.1}\\»
							- r _{i-1} - \mathop{\rm div}( m(\varphi_{i-1}) \nabla p _{i-1} )													
								- \mathop{\rm div}( \frac{\rho_2-\rho_1}{2}m(\varphi_{i-1}) (D q _{i} )^\top v _{i+1} )											\nonumber \\»\hskip6mm 
								- q _{i-1} \cdot\nabla \varphi _{i-1} 															 %\\»
					& \>=\>	\frac {\partial \mathcal{J}}{\partial \mu _{i} }( z ),															\label{T:MultOrg.2}\\»
 							- \frac 1 \tau \rho (\varphi_{j-1}) ( q _{j} - q _{j-1} ) - \rho (\varphi_{j-1}) (D q _{j} )^\top v _{j+1} 						\nonumber \\»\hskip6mm 
								- (D q _{j-1} )( \rho (\varphi_{j-2}) v _{j-1} -\frac{\rho_2-\rho_1}{2}m(\varphi_{j-2}) \nabla \mu _{j-1} )						\nonumber \\»\hskip6mm 
								- \mathop{\rm div}( 2 \eta (\varphi_{j-1}) \epsilon( q _{j-1} ) ) + p _{j-1} \nabla \varphi _{j-1} 							 %\\»
					& \>=\>	\frac {\partial \mathcal{J}}{\partial v _{j} }( z ),															\label{T:MultOrg.3}\\»
 %
% 					\Big(
% 							\frac {\partial \mathcal{J}}{\partial u _{k} }( z ) -  q _{k-1} 
% 					\Big)_{k=1}^{M-1}
							\frac {\partial \mathcal{J}}{\partial u}( z ) -  \tilde q
					& \>\in\>	\big[\mathbb{R}_+( U _{ad}- u )\big]^+\hspace{-0.05cm}.																		\label{T:MultOrg.4}
	\end{align}
\end{theorem}

\begin{proof}
	1. In the first step, we show the boundedness of
	$\left\{( p ^{(n)} , r ^{(n)} , q ^{(n)} )\right\}_{n\in\mathbb{N}}$ in ${ \overline{H}^1(\Omega )}^ M \times { \overline{H}^1(\Omega )}^ M \times { H^1_{0,\sigma}(\Omega ;\mathbb{R}^N )}^{M-1} $.
	Moreover, the boundedness of the sequence $\left\{( \varphi ^{(n)} , \mu ^{(n)} , v ^{(n)} , u ^{(n)} )\right\}_{n\in\mathbb{N}}$
	in $ \overline{H}^2_{\partial_n }(\Omega )^ M \times \overline{H}^2_{\partial_n }(\Omega )^ M \times H^2_{0,\sigma}(\Omega ;\mathbb{R}^N )^{M-1} \times L^2(\Omega;\mathbb{R}^N)^{M-1} $ follows from Lemma~\ref{energy2}.
	For $i=0,...,{M-1}$, $j=1,...,{M-1}$ and $n\in\mathbb{N}$ the adjoint system for (P$_{ \Psi ^{(n)} }$) corresponding to~(\ref{T:Mult.1})--(\ref{T:Mult.4})
		can be rewritten as
	\begin{align}	
					\frac 1 \tau p ^{(n)}_{i-1} - \Delta r ^{(n)}_{i-1} + {\Psi_0 ^{(n)}}'' ( \varphi ^{(n)}_{i} )^* r ^{(n)}_{i-1} 			&\>=\>		 \Theta ^{(n)}_{r,i-1}, 				\label{T:MultiOrg.5}\\
					- r ^{(n)}_{j-1} - \mathop{\rm div}( m(\varphi^{(n)}_{j-1}) \nabla p ^{(n)}_{j-1} ) - q ^{(n)}_{j-1} \cdot\nabla \varphi ^{(n)}_{j-1} 	&\>=\>		 \Theta ^{(n)}_{p,j-1}, 				\nonumber \\
					\frac 1 \tau \rho(\varphi^{(n)}_{j-1}) q ^{(n)}_{j-1} - \mathop{\rm div}( 2 \eta(\varphi^{(n)}_{j-1}) \epsilon( q ^{(n)}_{j-1} ) )
						+ p ^{(n)}_{j-1} \nabla \varphi ^{(n)}_{j-1} 									&							\nonumber \\
						- (D q _{j-1} )( \rho(\varphi^{(n)}_{j-2}) v ^{(n)}_{j-1}  -\frac{\rho_2-\rho_1}{2}m(\varphi^{(n)}_{j-2}) \nabla \mu ^{(n)}_{j-1} )	&\>=\>		 \Theta ^{(n)}_{q,j-1}, 				\nonumber 
	\end{align} %
		where the functionals $ \Theta ^{(n)}_{r} $, $ \Theta ^{(n)}_{p} $ and $ \Theta ^{(n)}_{q} $ are given by
	\begin{align*}
			 \Theta ^{(n)}_{r,i-1} 	&\>=\>	\frac {\partial \mathcal{J}}{\partial \varphi _{i} }( z ^{(n)} ) + \frac 1 \tau p ^{(n)}_{i} 
						-\Big[
								m'(\varphi^{(n)}_{i})\nabla \mu ^{(n)}_{i+1} \cdot p ^{(n)}_{i} 
								- \mathop{\rm div}( p ^{(n)}_{i} v ^{(n)}_{i+1} ) 
								- \kappa r ^{(n)}_{i+1} 													\\&\hskip5mm
								- \frac 1 \tau { \rho(\varphi^{(n)}_{i}) }' v ^{(n)}_{i+1} \cdot( q ^{(n)}_{i+1} - q ^{(n)}_{i} )
								+ 2{ \eta(\varphi^{(n)}_{i}) }' \epsilon( v ^{(n)}_{i+1} ) : D q ^{(n)}_{i} + \mathop{\rm div}( \mu ^{(n)}_{i+1} q ^{(n)}_{i} )				\\&\hskip5mm
								- ({ \rho(\varphi^{(n)}_{i}) }' v ^{(n)}_{i+1} -\frac{\rho_2-\rho_1}{2}m'(\varphi^{(n)}_{i})\nabla \mu ^{(n)}_{i+1} ) (D q ^{(n)}_{i+1} )^\top v ^{(n)}_{i+2} 		
						\Big],																	\\
 % %
			 \Theta ^{(n)}_{p,i-1} 	&\>=\>	\frac {\partial \mathcal{J}}{\partial \mu _{i} }( z ^{(n)} )
								+ \mathop{\rm div}( \frac{\rho_2-\rho_1}{2}m(\varphi^{(n)}_{i-1}) (D q ^{(n)}_{i} )^\top v ^{(n)}_{i+1} ),							\\
 % %
			 \Theta ^{(n)}_{q,i-1} 	&\>=\>	\frac {\partial \mathcal{J}}{\partial v _{i} }( z ^{(n)} )
								+ \frac 1 \tau \rho(\varphi^{(n)}_{i-1}) q ^{(n)}_{i} - \rho(\varphi^{(n)}_{i-1}) (D q ^{(n)}_{i} )^\top v ^{(n)}_{i+1} .				
	\end{align*}
	Here, $ z ^{(n)} $ denotes the tuple $( \varphi ^{(n)} , \mu ^{(n)} , v ^{(n)} , u ^{(n)} )$.
	We prove the boundedness of $\left\{( p ^{(n)} , r ^{(n)} , q ^{(n)} )\right\}_{n\in\mathbb{N}}$ in ${ \overline{H}^1(\Omega )}^ M \times { \overline{H}^1(\Omega )}^ M \times { H^1_{0,\sigma}(\Omega ;\mathbb{R}^N )}^{M-1} $ by backward induction over $i$.
	If $i\ge{M-1}$, then $( p ^{(n)}_{i} , r ^{(n)}_{i} , q ^{(n)}_{i} )=0$ by convention.
	In the induction step assume that for $i\in\{0,...,{M\hspace{-0.05cm}-\hspace{-0.05cm}1}\}$ and for $j\hspace{-0.05cm}\geq\hspace{-0.05cm} i$ the sequence
		$\left\{( p ^{(n)}_{j} , r ^{(n)}_{j} , q ^{(n)}_{j} )\right\}_{n\in\mathbb{N}}$ is bounded in $ \overline{H}^1(\Omega )\times \overline{H}^1(\Omega )\times H^1_{0,\sigma}(\Omega ;\mathbb{R}^N ) $.  %for $n\in\mathbb{N}$.
	This and the assumption on $\mathcal{J}$ 
		imply that $\left\{ ( \Theta ^{(n)}_{p,i-1} , \Theta ^{(n)}_{r,i-1} , \Theta ^{(n)}_{q,i-1} ) \right\}_{n\in\mathbb{N}}$ is bounded in $( \overline{H}^1(\Omega )\times \overline{H}^1(\Omega )\times H^1_{0,\sigma}(\Omega ;\mathbb{R}^N ) )^*$.
	To see this, we exemplarily consider
		first $2{ \eta(\varphi^{(n)}_{i}) }' \epsilon( v ^{(n)}_{i+1} )\hspace{-0.1cm} :\hspace{-0.1cm} D q ^{(n)}_{i} $, which is bounded by
		\begin{align*}
			|| 2{ \eta(\varphi^{(n)}_{i}) }' \epsilon( v ^{(n)}_{i+1} ) : D q ^{(n)}_{i} ||_{L^{6/5}}
			&\>\le\>	C ||{ \eta(\varphi^{(n)}_{i}) }'||_{L^\infty } ||\epsilon( v ^{(n)}_{i+1} )||_{L^{3}} ||D q ^{(n)}_{i} ||_{L^2}		\\
			&\>\le\>	C ||{ \eta(\varphi^{(n)}_{i}) }'||_{L^\infty } || v ^{(n)}_{i+1} ||_{H^2} || q ^{(n)}_{i} ||_{H^1}
		\end{align*}
		and secondly $-\frac{\rho_2-\rho_1}{2}m'(\varphi^{(n)}_{i})\nabla \mu ^{(n)}_{i+1} (D q ^{(n)}_{i+1} )^\top v ^{(n)}_{i+2} $, which we bounded using
		\begin{align*}
			&	|| -\frac{\rho_2-\rho_1}{2}m'(\varphi^{(n)}_{i})\nabla \mu ^{(n)}_{i+1} (D q ^{(n)}_{i+1} )^\top v ^{(n)}_{i+2} ||_{L^{6/5}}							\\
			&\hskip6mm 	\>\le\>	C ||-\frac{\rho_2-\rho_1}{2}m'(\varphi^{(n)}_{i})||_{L^\infty } ||\nabla \mu ^{(n)}_{i+1} ||_{L^6} || D q ^{(n)}_{i+1} ||_{L^2} || v ^{(n)}_{i+2} ||_{L^6}	\\
			&\hskip6mm 	\>\le\>	C ||-\frac{\rho_2-\rho_1}{2}m'(\varphi^{(n)}_{i})||_{L^\infty } || \mu ^{(n)}_{i+1} ||_{H^2} || q ^{(n)}_{i+1} ||_{H^1} || v ^{(n)}_{i+2} ||_{H^2}.
		\end{align*}
	Consequently, these terms define continuous linear functionals on $ \overline{H}^1(\Omega )$, that are bounded independently of $n$.
	The other summands can be estimated similarly.

	In case of $i>0$ we apply Lemma~\ref{L:Beschr} to
	\begin{align*}
		&		(\hat p ,\hat r ,\hat q ; \hat A ;
					h_ p , h_ r , h_ q ;
					\hat c ,\hat u ;
					\hat m ,\hat\eta ,\hat\rho )											\\
		& 	:=	( p ^{(n)}_{i-1} , r ^{(n)}_{i-1} , q ^{(n)}_{i-1} ; {\Psi_0 ^{(n)}}'' ( \varphi ^{(n)}_{i} )^*;
					 \Theta ^{(n)}_{p,i-1} , \Theta ^{(n)}_{r,i-1} , \Theta ^{(n)}_{q,i-1} ;												\\
		&  \quad 	 \varphi ^{(n)}_{i-1} , \rho(\varphi^{(n)}_{i-2}) v ^{(n)}_{i-1}  -\frac{\rho_2-\rho_1}{2}m(\varphi^{(n)}_{i-2}) \nabla \mu ^{(n)}_{i-1} ;
					 -\frac{\rho_2-\rho_1}{2}m(\varphi^{(n)}_{i-1}) , \eta(\varphi^{(n)}_{i-1}) , \rho(\varphi^{(n)}_{i-1}) ).
	\end{align*}
	Note that due to $\textnormal{div}v^{(n)}_{i-1}=0$ we have
	\begin{align*}
			\mathop{\rm div} \hat u 
			&=	{ \rho(\varphi^{(n)}_{i-2}) }' v ^{(n)}_{i-1} \cdot \nabla \varphi ^{(n)}_{i-1} - \mathop{\rm div}( \frac{\rho_2-\rho_1}{2}m(\varphi^{(n)}_{i-2}) \nabla \mu ^{(n)}_{i-1} )		\\
			&=	\frac {\rho_2-\rho_1}2 \Big[ v ^{(n)}_{i-1} \cdot \nabla \varphi ^{(n)}_{i-1} - \mathop{\rm div}( m(\varphi^{(n)}_{i-2}) \nabla \mu ^{(n)}_{i-1} ) \Big]	\\
			&=	\frac 1 \tau \frac {\rho_2-\rho_1}2 ( \varphi ^{(n)}_{i-1} - \varphi ^{(n)}_{i-2} )
			\enspace =\enspace 		-\frac 1 \tau ( \rho(\varphi^{(n)}_{i-1}) - \rho(\varphi^{(n)}_{i-2}) ).
	\end{align*}
	With the help of $\int_\Omega \langle (D \hat q ) \hat u ,\hat q \rangle = -\int_\Omega  \hat q \cdot \mathrm{div}(\hat q \otimes \hat u)$
		(cf.~(\ref{T:Mult.x1})),	(\ref{helpeq}) yields
	\begin{align*}
			\frac 1 \tau \int_\Omega \hat\rho |\hat q |^2 dx - \langle (D \hat q ) \hat u ,\hat q \rangle
			&\>=\>		\frac 1 \tau \int_\Omega \rho(\varphi^{(n)}_{i-1}) | q ^{(n)}_{i-1} |^2 - \frac 1 2 ( \rho(\varphi^{(n)}_{i-1}) - \rho(\varphi^{(n)}_{i-2}) ) | q ^{(n)}_{i-1} |^2 dx		\\
			&\>=\>		\frac 1 {2\tau} \int_\Omega (\rho(\varphi^{(n)}_{i-2}) + \rho(\varphi^{(n)}_{i-1}) ) | q ^{(n)}_{i-1} |^2 dx
			\enspace \ge\enspace 	0,
	\end{align*}
	because of $ \rho(\varphi^{(n)}_{i-2}) \ge0$ and $ \rho(\varphi^{(n)}_{i-1}) \ge0$ almost everywhere.
	Hence Lemma~\ref{L:Beschr} implies the boundedness of $( p ^{(n)}_{i-1} , r ^{(n)}_{i-1} , q ^{(n)}_{i-1} )$ in
	$ \overline{H}^1(\Omega )\times \overline{H}^1(\Omega )\times H^1_{0,\sigma}(\Omega ;\mathbb{R}^N ) $.

	The case $i=0$ needs some modifications 
		in order to be treated by Lemma~\ref{L:Beschr} since~(\ref{T:Mult.3}) is not defined for $i=0$.
	In this case we set
	%\[
			$(\hat q , h_q, \hat c ,\hat u ,\hat\eta)	\hspace{-0.1cm}=\hspace{-0.1cm} (0,0,0,0, \eta(\varphi^{(n)}_{i}) )$
	%\]
		together with the definition of the remaining quantities as in the case $i>0$.
	Now, by Lemma~\ref{L:Beschr} we conclude the boundedness of $( p ^{(n)}_{i-1} , r ^{(n)}_{i-1} )$ in $ \overline{H}^1(\Omega )\times \overline{H}^1(\Omega )$.
	Moreover, from~(\ref{T:MultiOrg.5}) it follows that also $({\Psi_0 ^{(n)}}'' ( \varphi ^{(n)}_{i} )^* r ^{(n)}_{i-1} )$ remains bounded in $ \overline{H}^1(\Omega )^*$.

	2. With the bounds derived in step~1 and with the usual compact embeddings
		of Sobolev spaces, we can pass to a subsequence with the desired convergence properties.

	3. Now we pass to the limit in the the adjoint systems corresponding to~(\ref{T:Mult.1})--(\ref{T:Mult.4}) for (P$_{ \Psi ^{(n)} }$).
	The limits for the equations~(\ref{T:Mult.1}) and~(\ref{T:Mult.2}) are considered in $ \overline{H}^1(\Omega )^*$ and
		the limit for~(\ref{T:Mult.3}) in $ H^1_{0,\sigma}(\Omega ;\mathbb{R}^N )^*$.
	In the linear terms we can pass to the limit at once.
	For $m'(\varphi^{(n)}_{i})\nabla \mu ^{(n)}_{i+1} \cdot p ^{(n)}_{i} $ we have that
		$m'(\varphi^{(n)}_{i})$ converges	strongly in	$L^\infty(\Omega) $		to $ m'(\varphi_{i})$,
		$\nabla \mu ^{(n)}_{i+1} $			strongly in	$L^{6-\varepsilon }(\Omega)$	to $\nabla \mu _{i+1} $		and
		$ p ^{(n)}_{i} $				weakly in	$L^6(\Omega)$		to $ p _{i} $.
		Hence,
		$m'(\varphi^{(n)}_{i})\nabla \mu ^{(n)}_{i+1} \cdot p ^{(n)}_{i} $
							converges weakly in	$ \overline{H}^1(\Omega )^*$		to $ m'(\varphi_{i})\nabla \mu _{i+1} \cdot p _{i} $.
	For $\frac{\rho_2-\rho_1}{2}m'(\varphi^{(n)}_{i})\nabla \mu ^{(n)}_{i+1} (D q ^{(n)}_{i+1} )^\top v ^{(n)}_{i+2} $ we note that
		$\frac{\rho_2-\rho_1}{2}m'(\varphi^{(n)}_{i})$ and $ v ^{(n)}_{i+2} $ converge
							strongly in	$L^\infty(\Omega) $		to $\frac{\rho_2-\rho_1}{2}m'(\varphi_{i})$ respectively $ v _{i+2} $,
		$\nabla \mu ^{(n)}_{i+1} $			strongly in	$L^{6-\varepsilon }(\Omega)$	to $\nabla \mu _{i+1} $		and
		$D q ^{(n)}_{i} $				weakly in	$L^2(\Omega)$		to $D q _{i} $.
		\hspace{-0.05cm}Therefore
		$\frac{\rho_2-\rho_1}{2}m'(\varphi^{(n)}_{i})\nabla \mu ^{(n)}_{i+1} (D q ^{(n)}_{i+1} )\hspace{-0.05cm}^\top\hspace{-0.05cm} v ^{(n)}_{i+2} $ converges
							weakly in	$ \overline{H}^1(\Omega )^*$		to $\frac{\rho_2-\rho_1}{2}m'(\varphi_{i})\nabla \mu _{i+1} (D q _{i+1} )^\top v _{i+2} $.
 
	For $\mathop{\rm div}( \frac{\rho_2-\rho_1}{2}m(\varphi^{(n)}_{i-1}) (D q ^{(n)}_{i} )^\top v ^{(n)}_{i+1} )$ we use that
		$ \frac{\rho_2-\rho_1}{2}m(\varphi^{(n)}_{i-1}) $ and $ v ^{(n)}_{i+1} $ converge
							strongly in	$L^\infty(\Omega) $		to $ \frac{\rho_2-\rho_1}{2}m(\varphi_{i-1}) $ respectively $ v _{i+1} $, and
		$ q ^{(n)}_{i} $				weakly in	$L^6(\Omega)$		to $ q _{i} $.
		As a consequence,
		$\mathop{\rm div}( \frac{\rho_2-\rho_1}{2}m(\varphi^{(n)}_{i-1}) (D q ^{(n)}_{i} )^\top v ^{(n)}_{i+1} )$
							converges weakly in	$ \overline{H}^1(\Omega )^*$		to $\mathop{\rm div}( \frac{\rho_2-\rho_1}{2}m(\varphi_{i-1}) (D q _{i} )^\top v _{i+1} )$.
	For the convergence of $\mathop{\rm div}( 2 \eta(\varphi^{(n)}_{i-1}) \epsilon( q ^{(n)}_{i-1} ) )$ note that
		$ \eta(\varphi^{(n)}_{i-1}) $ converges	strongly in	$L^\infty (\Omega)$		to $ \eta (\varphi_{i-1}) $			 and
		$\epsilon( q ^{(n)}_{i-1} )$		weakly in	$L^2(\Omega)$		to $\epsilon( q _{i-1} )$.
		Hence,
		$\mathop{\rm div}( 2 \eta(\varphi^{(n)}_{i-1}) \epsilon( q ^{(n)}_{i-1} ) )$ converges
							weakly in $ H^1_{0,\sigma}(\Omega ;\mathbb{R}^N )^*$			to the limit $\mathop{\rm div}( 2\eta (\varphi_{i-1}) \epsilon( q _{i-1} ) )$.
	Apart from ${\Psi_0 ^{(n)}}'' ( \varphi ^{(n)}_{i} )^* r ^{(n)}_{i-1} $, all remaining terms appearing on the left hand sides can be treated similarly.
	Moreover, our assumptions on $\mathcal{J}$ imply that $\mathcal{J}'( \varphi ^{(n)} , \mu ^{(n)} , v ^{(n)} , u ^{(n)} )$ converges weakly to $\mathcal{J}'( \varphi , \mu , v , u )$
		in $( \overline{H}^2_{\partial_n }(\Omega )^ M \times \overline{H}^2_{\partial_n }(\Omega )^ M \times H^1_{0,\sigma}(\Omega ;\mathbb{R}^N )^{M-1} \times L^2(\Omega;\mathbb{R}^N)^{M-1} )^*$.
	
	Consequently, by~(\ref{T:Mult.1}) also $ \Psi_0'' ( \varphi ^{(n)}_{i} )^* r ^{(n)}_{i-1} $ converges weakly in $( \overline{H}^1(\Omega )^*)^ M $ to some $ \lambda _{i-1} $.
	Therefore, we arrive at the system~(\ref{T:MultOrg.1})--(\ref{T:MultOrg.3}).
	Finally, notice that for all $y\in U_{ad}$ and with
		$z^{(n)}:=(\varphi ^{(n)} , \mu ^{(n)} , v ^{(n)} , u ^{(n)})$ and $\tilde q_k^{(n)}:=q_{k-1}^{(n)}$
		by the weak lower-semicontinuity of $\frac{\partial \mathcal{J}}{\partial u}$
		and the weak and strong convergence of the sequences involved
		we deduce that
	\begin{align*}
		&\textstyle	
						\langle \frac{\partial \mathcal{J}}{\partial u}(z)-\tilde q , y-u \rangle
		\>=\> \textstyle
						\langle \frac{\partial \mathcal{J}}{\partial u}(z) , y \rangle
					-	\langle \frac{\partial \mathcal{J}}{\partial u}(z) , u \rangle
					-	\langle \tilde q , y-u \rangle							\\
		&\>\ge\>\textstyle
				\liminf\limits_{n\rightarrow\infty} \Big(
						\langle \frac{\partial \mathcal{J}}{\partial u}(z^{(n)}) , y \rangle
					-	\langle \frac{\partial \mathcal{J}}{\partial u}(z^{(n)}) , u^{(n)} \rangle
					-	\langle \tilde q^{(n)} , y-u^{(n)} \rangle	
				\Big)	\\
		&\>=\>\textstyle
				\liminf\limits_{n\rightarrow\infty} \Big(
						\frac{\partial \mathcal{J}}{\partial u}(z^{(n)})-\tilde q^{(n)} , y-u^{(n)} \rangle
				\Big)	\\
		&\>\ge\>0,
	\end{align*}
	due to the optimality of $z^{(n)}$ for ($P_{\Psi^{(n)}}$).
	This shows (\ref{T:MultOrg.4}) and finishes the proof.
\end{proof}
 % %
\begin{remark}
 We point out that a tracking-type functional, like, e.g.,
 \begin{align*}
  \mathcal{J}(\varphi,\mu,v,u):=\frac{1}{2}\left\|\varphi_{M-1}-\varphi_d \right\|^2+\frac{\xi}{2}\left\| u \right\|^2_{(L^2)^{(M-1)}},\ \xi>0,
 \end{align*}
 %where the control is penalized in $L^2(\Omega;\mathbb{R}^N)^{M-1} )$
with $\varphi_d\in L^2(\Omega)$ a desired final state,
satisfies the assumptions of Theorem \ref{T:MultOrg}.
\end{remark}
\begin{remark}
 If the set $U_{ad}$ is bounded, Theorem \ref{T:MultOrg} holds also true for a sequence
 $\left\{( \varphi ^{(n)} , \mu ^{(n)} , v ^{(n)} , u ^{(n)} )\right\}_{n\in\mathbb{N}}$
 of stationary points for ($P_{\Psi^{(n)}}$).
 If it is unbounded, then the result can still be transferred to sequences of stationary points by assuming that the sequence $\left\{u ^{(n)} \right\}_{n\in\mathbb{N}}$ is bounded in $L^2(\Omega;\mathbb{R}^N)^{M-1}$.
\end{remark}

\section{Stationarity conditions in case of the double-obstacle potential}\label{sec:4}
In this section, we apply the developed theory to the initially stated optimal control problem associated to the double-obstacle potential.
For this purpose, let $\psi_0$ be defined as in Assumption \ref{assPsi}.1 and set $\gamma:=\partial\psi_0\subset \mathbb{R}\times \mathbb{R}$.
Then we define the sequence of approximating double-well type potentials as follows.
\begin{definition}\label{defpotapp}%[Double-obstacle potential]
% 	For $ \psi_1 , \psi_2 \in\mathbb{R}$ with $ \psi_1 <0< \psi_2 $ we set
% 	\begin{align*}
% 		&
% 			K 	:=	[ \psi_1 , \psi_2 ],			\hskip1cm
% 			\psi 	:=	\imath _{ K }:\mathbb{R}\rightarrow \overline {\mathbb{R}},		\hskip1cm
% 			\gamma 	:=	\partial \psi \subset \mathbb{R}\times \mathbb{R},
% 		&\\&
% 			K _0 	:=	\{\varphi \in \overline{L}^2(\Omega ): \varphi (x)\in K \text{ for a.e. } x\in\Omega \},		\enspace \enspace 
% 			K _1 	:=	K _0 \cap \overline{H}^2_{\partial_n }(\Omega ).							 %\enspace \enspace 
% 		&
% 	\end{align*}
%  %
	Let a mollifier $\zeta \in C^1(\mathbb{R})$ with
			$\mathop{\rm supp}\zeta \subset [-1,1]$,\enspace 
			$\int_\mathbb{R}\zeta =1$ and
			$0\le\zeta \le1$ a.e.~on $\mathbb{R}$,
		and a function $\theta :\mathbb{R}^+\rightarrow \mathbb{R}^+$, with $\theta (\alpha )>0$ and $\frac {\theta (\alpha )}{\alpha }\rightarrow 0$ as $\alpha \rightarrow 0$, be given.
	For the Yosida approximation $\gamma _\alpha $ with parameter $\alpha >0$ of $\gamma $ define
		\begin{eqnarray*}
		&	\zeta _{\alpha } (s)	:=	\frac 1\alpha \zeta \Big(\frac {s}{\alpha }\Big),	\enspace \enspace 
			\widetilde{\gamma} _\alpha 		:=	\gamma _\alpha * \zeta _{\theta (\alpha )} ,					\enspace \enspace 
			{\psi_0}_\alpha (s)	:=	\int_0^s\widetilde{\gamma} _\alpha(t)\,dt ,	
		&\\&
			{\Psi_0 }_\alpha (c)	:=	\int_\Omega ({\psi_0}_\alpha \circ c)(t)\,dt.
		&
		\end{eqnarray*}
	Moreover, we set
		%\[
 %
			$\alpha _n :=  n^{-1}$,			%\hskip1cm
			$\Psi_0 ^{(n)}:={\Psi_0 }_{\alpha _n }$.		%\hskip1cm
			%$A ^{(n)} :=\partial \Psi_0 ^{(n)}$.
		%\]
\end{definition}
\begin{remark}
        We note that ${\Psi_0^{(n)}}'$ can be identified with the superposition
operator corresponding to
                $\widetilde{\gamma}_{\alpha_n}$, cf.~\cite{Hintermueller2012}.
        Since $\widetilde{\gamma}_{\alpha_n}'$ is bounded and
                since $\overline{H}^2_{\partial_n}(\Omega)$ embeds continuously
                into $L^{2-\delta}(\Omega)$ for $\delta>0$,
                it follows that ${\Psi_0^{(n)}}'$ maps
$\overline{H}^2_{\partial_n}(\Omega)$
                        continuously Fr\`echet-differentiably into $L^2(\Omega)$, see,
e.g., \cite{Goldberg1992}.
\end{remark}

In order to obtain a stationarity condition for the optimal 
control problem of CHNS with the double-obstacle potential we pass to the limit (with the Yosida parameter) 
in a sequence of optimal control
    problems with approximating double-well-type potentials.
\begin{theorem}[Limiting $\varepsilon$-almost C-stationarity]\label{T:Doub}
	Let $ \Psi_0 ^{{(n)}} ,\ n\in\mathbb{N}$ be the functionals of Definition \ref{defpotapp}, and
		let the tuples $( \varphi ^{(m)} , \mu ^{(m)} ,$\hskip0pt $ v ^{(m)} , u ^{(m)} ,$\hskip0pt \,$ p ^{(m)} , r ^{(m)} , q ^{(m)} )$,\hspace{-0.05cm} $( \varphi , \mu , v , u , p , r , q )$
		and $\mathcal{J}$ be as in Theorem~\ref{T:MultOrg}.
	Moreover, let $\Lambda :\mathbb{R}\rightarrow \mathbb{R}$ be a Lipschitz function with $\Lambda ( \psi_1 )= \Lambda ( \psi_2 )=0$.
	For
	\[
% 		 a ^{(m)}_{i} 	:=	 \mu ^{(m)} - ( - \Delta \varphi ^{(m)}_{i+1} - \kappa \varphi ^{(m)}_{i-1} 	)	,\enspace \enspace 
 		 a ^{(m)}_{i} 	:=	 {\Psi_0 ^{(m)}}'(\varphi^{(m)}_{i})	,\enspace \enspace 
		 \lambda ^{(m)}_{i} 	:=	{\Psi_0 ^{(m)}}'' ( \varphi ^{(m)}_{i} )^* r ^{(m)}_{i-1} 
	\]
	for $i=0,..., M $, and for $ a _{i} $ denoting the limit of $ a ^{(m)}_{i} $, it holds that
	\begin{align*}
			(\, a _{i} , \Lambda ( \varphi _{i} )\,)_{L^2}	&	=	0	,&
			\langle \lambda _{i} , \Lambda ( \varphi _{i} )\rangle	&	=	0	,\\
			(\, a _{i} , r _{i-1} \,)_{L^2}		&	=	0	,&
		\liminf (\, \lambda ^{(m)}_{i} , r ^{(m)}_{i-1} \,)_{L^2}	&	\ge	0	.
	\end{align*}
	Moreover,
		for every $\varepsilon >0$
		there exist a measurable subset $ M ^\varepsilon _{i} $ of $ M _{i} :=\{x\in\Omega \>:\> \psi_1 < \varphi _{i} (x)< \psi_2 \}$ with
		$| M _{i} \setminus M ^\varepsilon _{i} |<\varepsilon $ and
	\[
			\langle \lambda _{i} ,v\rangle=0	\enspace \enspace \enspace \enspace \forall v\in \overline{H}^1(\Omega ),\enspace v|_{\Omega \setminus M ^\varepsilon _{i} }=0.
	\]
\end{theorem}

\begin{proof}
 	1. The subdifferential $\gamma$ % =\partial \imath _{ K }$ of the indicator function of $ K $
		satisfies $y \Lambda(x)=0$ if $(x,y)\in\gamma $.
	Since $( \varphi _{i} , a _{i} )\in\gamma $ a.e.~on $\Omega $ and since $ a _{i} \in L^2(\Omega )$, integration yields the complementarity condition
		$(\, a _{i} ,\Lambda ( \varphi _{i} )\,)_{L^2}=0$.
		
	2. Now we show that $(\, \lambda _{i} , \Lambda ( \varphi _{i} )\,)_{L^2}=0$.
	It is well-known that the superposition $ P_{ K }$ of the
		metric projection $ p_{ K } $ of $\mathbb{R}$ onto $ K :=[ \psi_1 , \psi_2 ]$ 
		maps $ \overline{H}^1(\Omega )$ continuously into itself.
 % % %
	Denoting by $ L_\Lambda $ the Lipschitz constant of $\Lambda$, it holds that $|\Lambda(s)|\leq L_\Lambda \min(|s- \psi_1 |,|s- \psi_2 |)$ for $s\in\mathbb{R}$.
	Using $|\widetilde{\gamma} _\alpha '(s)|\le\frac {1}{\alpha }$ for all $s$ and $\widetilde{\gamma} _\alpha '(s)=0$ for $ \psi_1 +\theta (\alpha )\leq s\leq \psi_2 -\theta (\alpha )$ (cf.~\cite{Hintermueller2012}) yields
	\begin{align*}
		| (\, \lambda ^{(m)}_{i} ,\Lambda( P_{ K }( \varphi ^{(m)}_{i} )) \,)_{L^2} |^2
			&\>=\>		| (\, r ^{(m)}_{i} , {\Psi_0 ^{(m)}}'' ( \varphi ^{(m)}_{i} )\Lambda( P_K( \varphi ^{(m)}_{i}) ) \,)_{L^2} |^2							\\
			&\>\le\>	||\, r ^{(m)}_{i} \,||_{L^2}^2 \int_{\Omega } | \widetilde\gamma _{\alpha _m }'( \varphi ^{(m)}_{i} )\Lambda( P_K( \varphi ^{(m)}_{i} ) ) |^2		\\
			&\>\le\>	\bigg( |\Omega | \> ||\, r ^{(m)}_{i} \,||_{L^2}\, L_\Lambda \frac {\theta (\alpha _m )}{\alpha _m } \bigg)^2 \enspace \rightarrow \enspace 	0
	\end{align*}
	as $m\rightarrow \infty $ and consequently
	\begin{align*}
		&		\lim (\, \lambda ^{(m)}_{i} ,\Lambda( \varphi ^{(m)}_{i} )\,)_{L^2}													\\
		&\>=\>		\lim (\, \lambda ^{(m)}_{i} ,\Lambda(  P_{ K }( \varphi ^{(m)}_{i} ))\,)_{L^2} + \lim \langle \lambda ^{(m)}_{i} ,\Lambda( \varphi ^{(m)}_{i} )-\Lambda( P_{ K }( \varphi ^{(m)}_{i} )) \rangle_{ \overline{H}^1(\Omega )}		\\
		&\>=\>		0,
	\end{align*} %
	which implies $\langle \lambda _{i} ,\Lambda( \varphi _{i} )\rangle=0$ since $ \varphi ^{(m)}_{i} $ converges strongly to $ \varphi _{i} = P_{ K }( \varphi _{i} )$ in $ \overline{H}^1(\Omega )$.
	
	3. Denoting $ g_m (s):=\widetilde\gamma _{\alpha _m }(s)-\widetilde\gamma _{\alpha _m }'(s)\pi (s)$ with  %$\pi :=\gamma _1$. With
		$s- p_{ K } (s)=:\pi (s)$ yields
	\begin{align*}
		(\, a ^{(m)}_{i} , r ^{(m)}_{i-1} \,)_{L^2} 
			&\>=\>		\big(\, r ^{(m)}_{i-1} ,  \widetilde\gamma _{\alpha _m }( \varphi ^{(m)}_{i} )\,\big)_{L^2}								\\
			&\>=\>		\big(\, r ^{(m)}_{i-1} ,  g_m ( \varphi ^{(m)}_{i} )\,\big)_{L^2} + \big(\, \lambda ^{(m)}_{i} ,  \varphi ^{(m)}_{i} - P_K( \varphi ^{(m)}_{i}) \,\big)_{L^2}.
	\end{align*}
	Since $| g_m (s)|=|\widetilde\gamma _{\alpha _m }(s)-\widetilde\gamma _{\alpha _m }'(s)\pi (s)| \leq C\,\frac {\theta (\alpha _m )}{\alpha _m }$ for $m$ sufficiently large (cf. Lemma 4.2 in~\cite{Hintermueller2012}),
		the first term on the right-hand side converges to $0$ and
		the second one as well because of the strong convergence of $( \varphi ^{(m)}_{i} )$ and $( P_{ K }( \varphi ^{(m)}_{i} ))$ to $ \varphi _{i} $ in $ \overline{H}^1(\Omega )$, respectively.
		
	4. The property $\liminf (\, \lambda ^{(m)}_{i} , r ^{(m)}_{i-1} \,)_{L^2}\ge0$ follows readily from the monotonicity of ${\Psi_0 ^{(m)}}'' ( \varphi ^{(m)}_{i} )$.
	
	5. The convergence properties of $ \varphi ^{(m)}_{i} $ imply that the subset
		$G:=\{x\in\Omega \>:\> \varphi ^{(m)}_{i} (x)\rightarrow \varphi _{i} (x) \text{ as } m\rightarrow \infty \}$ of $\Omega $ has full measure (i.e.~$|G|=|\Omega |$).
	Therefore, for every $x\in G\cap M _{i} $ we can find $m_0(x)\in\mathbb{N}$ with $ \psi_1 +\theta (\alpha _m ) < \varphi ^{(m)}_{i} (x) < \psi_2 -\theta (\alpha _m )$ for all $m\ge m_0(x)$.
	Thus, $ \lambda ^{(m)}_{i} (x)=\widetilde\gamma _{\alpha _m }'( \varphi ^{(m)}_{i} (x)) r ^{(m)}_{i} (x)$ converges to $0$ on $G\cap M _{i} $.
	Using Egorov's theorem shows that for every $\varepsilon >0$ there exists a subset $ M ^\varepsilon _{i} $ of $G\cap M _{i} $ with $| M _{i} \setminus M ^\varepsilon _{i} |<\varepsilon $ such that
		$ \lambda ^{(m)}_{i} $ converges uniformly to zero on $ M ^\varepsilon _{i} $.
	Hence, we obtain
		$\langle \lambda _{i} ,v\rangle=\lim \langle \lambda ^{(m)}_{i} ,v\rangle=0$ for every $v\in \overline{H}^1(\Omega )$ with $v|_{\Omega \setminus M ^\varepsilon _{i} }=0$.
\end{proof}

 % % % %
%
In combination with the results from Theorem \ref{T:MultOrg}, Theorem \ref{T:Doub} states stationarity conditions corresponding to a function space version of C-stationarity for MPECs, cf. \cite{Hintermuller2009,Hintermuller2014}.
\section{Conclusion}

Our specific semi-discretization in time for the coupled CHNS system 
with non-matched fluid densities
represents a first step towards a numerical investigation/realization of 
the problem. Most importantly, it preserves the
strong coupling of the Cahn-Hilliard and Navier-Stokes system which, in 
the case of
non-matched densities, is additionally enforced through the presence of 
the relative
flux $J$. As a result, well-posedness of the time discrete scheme is 
guaranteed and
energy estimates mirroring the physical fact of decreasing energies can 
be argued.
Such an energy property is not clear for the time continuous problem at 
this point in time and might
be the subject of further research.

Concerning the potential chosen in the Ginzburg-Landau energy, we note 
that while the existence of global solutions to the
optimal control problem can be shown for both cases (i.e., for 
double-well and double obstacle potentials) simultaneously,
the derivation of stationarity conditions is more delicate.
In fact, the double-obstacle potential gives rise to a degenerate 
constraint system with the overall problem falling into the realm of 
mathematical programs with equilibrium constraints (MPECs).
In our approach, the constraint degeneracy  is handled by a Moreau-Yosida
regularization approach (resulting in an approximating sequence of 
double-well-type potentials) and a subsequent limiting process leading 
to a function
space version of so-called C-stationarity. For the underlying problem 
class, our limiting version of C-stationarity is currently
the most (and, to the best of our knowledge, only) selective 
stationarity system available. As an alternative analytical approach, 
one may want to pursue set-valued analysis in order to derive 
stationarity conditions directly, i.e., from applying variational 
geometry (contingent, critical and normal cones)
and generalized differentiation. This, however, is usually not possible 
by simple application of available tools, but rather by
expanding current technology. It, thus, may serve as a subject of our 
future work on this problem class.

Finally, we point out that the constructive nature of our derivation of 
stationarity conditions facilitates
a numerical implementation of the approach which can be exploited in future
investigations of these problem types, both, from a numerical, as well 
as, from a
practical point of view. In \cite{Hintermuller2014a}, this has already been effectively 
done for the case of
matched densities.

\appendix
%\section{Appendix}
\section{Proof of Lemma \ref{kind}}
\begin{proof}
Let $L_\varepsilon:\overline{H}^1(\Omega)\rightarrow \overline{H}^{-1}(\Omega)$ be defined by
\begin{align}
\left\langle L_\varepsilon(\varphi),\phi\right\rangle
&:=\left\langle-\Delta\varphi,\phi\right\rangle\nonumber\\
&-\left\langle g_1 +\max(-g_1,0){\theta}_{\varepsilon}(\varphi-\psi_1) %\nonumber\\
+\min(-g_1,0){\theta}_{\varepsilon}(\psi_2-\varphi),\phi\right\rangle\label{opL}
\end{align}
where $\phi\in \overline{H}^1(\Omega)$ and $\theta_\varepsilon$ is defined by
$$
\theta_{\varepsilon}(x):=\left\{\begin{array}[c]{ll}
1 & \text{if } x\leq 0,\\
1-\frac{x}{\varepsilon}& \text{if } 0\leq x\leq\varepsilon, \\
0 & \text{if } x\geq\varepsilon.
\end{array} \right.
% \ \theta_{\varepsilon}^{(2)}(x):=\left\{\begin{array}[c]{ll}
% 1 & \text{if } x\geq 0\\
% 1+\frac{x}{\varepsilon}& \text{if } 0\geq x\geq-\varepsilon \\
% 0 & \text{if } x\leq-\varepsilon\\
% \end{array} \right..
$$
Since $g_1\in L^2(\Omega)$ and ${\theta}_{\varepsilon}(\varphi-\psi_{1}), {\theta}_{\varepsilon}(\psi_{2}-\varphi)\in L^\infty(\Omega)$,
it holds that
\begin{align}
\left\| g_1 +\max(-g_1,0){\theta}_{\varepsilon}(\varphi-\psi_1) %\nonumber\\
+\min(-g_1,0){\theta}_{\varepsilon}(\varphi-\psi_2)\right\|\leq \left\| g_1 \right\|.\label{h1h}
\end{align}
We show that for every $0<\varepsilon\leq\min(-\psi_1,\psi_2)$ there exists a unique $\varphi_\varepsilon\in H^2_{m\mbox{}}\cap\mathbb{K}$ such that
\begin{align}
L_\varepsilon(\varphi_\varepsilon) % %
&=0,\label{hVIreg}
\end{align} % %
In fact, for every $w,v\in \overline{H}^1(\Omega)$, it can be seen that
\begin{align*}
\left\langle L_\varepsilon (w)-L_\varepsilon (v),w-v\right\rangle
% &=\left\langle -\Delta (w-v),w-v\right\rangle\\
% &-\int_\Omega\max(-g_1,0)\left({\theta}_{\varepsilon}(w-\psi_1)-{\theta}_{\varepsilon}(v-\psi_1)\right)(w-v)dx\\
% &-\int_\Omega\min(-g_1,0)\left({\theta}_{\varepsilon}(\psi_2-w)-{\theta}_{\varepsilon}(\psi_2-v)\right)(w-v)dx\\ %
&\geq\int_\Omega\left|\nabla w-\nabla v\right|^2dx %
\end{align*}
where we use the monotonicity of ${\theta}_{\varepsilon}$. % and ${\theta}_{\varepsilon}^{(2)}$, respectively.	
By Poincar\'e's inequality there exists a constant $C>0$ such that
\begin{align*}
\left\langle L_\varepsilon (w)-L_\varepsilon (v),w-v\right\rangle
\geq \left\| \nabla w-\nabla v\right\|^2
\geq C\left\| w-v\right\|_{H^1}^2.
\end{align*}
Consequently, $L_\varepsilon$ is strongly monotone and coercive. %
Since $L_\varepsilon$ is also continuous on finite dimensional subspaces of $\overline{H}^1(\Omega)$, \cite[III: Corollary 1.8]{Kinderlehrer2000} is applicable
which yields the existence of $\varphi_\varepsilon\in \overline{H}^1(\Omega)$ with $L_\varepsilon(\varphi_\varepsilon)=0$.

Due to the definition of $L_\varepsilon$ and inequality (\ref{h1h}), we have $\Delta\varphi_\varepsilon\in L^2(\Omega)$.
By
\cite[Theorem 2.3.6]{Maugeri2000} and \cite[Theorem 2.3.1]{Maugeri2000}
%\cite[Theorem 3.1.3.3]{gris}
there exists a constant $C_1>0$  % independent from $\varepsilon$
such that
\begin{align}
\left\|\varphi_\varepsilon\right\|_{H^2}\leq C_1 \left\|\Delta\varphi_\varepsilon\right\|+\left\|\varphi_\varepsilon \right\|. %\label{hVIes} %_{L^2}
\end{align}
In combination with (\ref{h1h}) and Poincar\'e's inequality, this leads to
\begin{align}
\left\|\varphi_\varepsilon\right\|_{H^2}\leq C_2 \left\|g_1 \right\|.\label{hVIes} %_{L^2}
\end{align} %
Now, we set $\beta_\varepsilon:=\varphi_\varepsilon-\min(\varphi_\varepsilon,\psi_2)\geq 0$ and observe that
\begin{align}
\left\|\nabla\beta_\varepsilon\right\|^2 %_{L^2}^2
=\int_{\Omega_1}\nabla(\varphi_\varepsilon-\psi_2)\nabla\beta_\varepsilon dx
=\left\langle-\Delta\varphi_\varepsilon,\beta_\varepsilon\right\rangle
\end{align}
where $\Omega_1:=\left\{x\in\Omega:\beta_\varepsilon(x)>0 \right\}=\left\{x\in\Omega:\varphi_\varepsilon(x)>\psi_2\geq\psi_1+\varepsilon \right\}$.
By equation (\ref{opL}) and (\ref{hVIreg}), this leads to
\begin{align*}
\left\|\nabla\beta_\varepsilon\right\|^2 %_{L^2}^2 %
% &\leq\int_\Omega (g_1+\max(-g_1,0){\theta}_{\varepsilon}(\varphi_\varepsilon-\psi_1) %\nonumber\\
% +\min(-g_1,0){\theta}_{\varepsilon}(\psi_2-\varphi_\varepsilon))\beta_\varepsilon dx\nonumber\\
&=\int_{\Omega_1} (g_1+\max(-g_1,0){\theta}_{\varepsilon}(\varphi_\varepsilon-\psi_1) %\nonumber\\
+\min(-g_1,0){\theta}_{\varepsilon}(\psi_2-\varphi_\varepsilon))\beta_\varepsilon dx\nonumber\\
&=\int_{\Omega_1} (g_1+\min(-g_1,0))\beta_\varepsilon dx
\leq 0.
\end{align*}
Thus, $\beta_\varepsilon=0$ and therefore $\varphi_\varepsilon\leq\psi_2$ almost everywhere in $\Omega$.

In a similar way, we prove that $\varphi_\varepsilon-\max(\varphi_\varepsilon,\psi_1)=0$ and therefore $\varphi_\varepsilon\geq\psi_1$ almost everywhere on $\Omega$.
Hence $\varphi_\varepsilon$ is contained in $\overline{H}^2(\Omega)\cap \mathbb{K}$.
By inequality (\ref{hVIes}), the sequence $\left\{\varphi_\varepsilon\right\}_{\varepsilon\rightarrow 0}$ is bounded in $\overline{H}^2(\Omega)$
and there exists a weakly convergent subsequence (denoted the same) such that $\varphi_\varepsilon\rightharpoonup_{\overline{H}^2}\varphi^*$
with $\left\|\varphi^*\right\|_{H^2}\leq C_2 \left\|g_1\right\| $.
Since $\mathbb{K}$ is weakly closed, it contains $\varphi^*$.

For arbitrarily small $0<\delta\leq \min(-\psi_1,\psi_2)$, let $v\in \mathbb{K}$ be such that $\psi_1+\delta\leq v\leq\psi_2-\delta$ almost everywhere in $\Omega$. %
Using equation (\ref{hVIreg}) and the monotonicity of $L_\varepsilon$, we infer
% Using \cite[III: Lemma 1.5]{Kinderlehrer2000}, we infer
\begin{align*}
0\leq \left\langle L_\varepsilon (v),v-\varphi_\varepsilon\right\rangle
&=\left\langle-\Delta v,v-\varphi_\varepsilon\right\rangle
-\int_\Omega(g_1+\max(-g_1,0){\theta}_{\varepsilon}(v-\psi_1) \\
&\hspace{1cm}+\min(-g_1,0){\theta}_{\varepsilon}(\psi_2-v))(v-\varphi_\varepsilon)dx\\
&=\left\langle-\Delta v,v-\varphi_\varepsilon\right\rangle
-\int_\Omega g_1 (v-\varphi_\varepsilon)dx
\end{align*}
for every $0<\varepsilon<\delta$. For $\varepsilon\rightarrow 0$ this leads to
\begin{align*}
0\leq\left\langle-\Delta v,v-\varphi^*\right\rangle
-\int_\Omega g_1 (v-\varphi^*)dx.
\end{align*}
Since $\delta>0$ can be chosen arbitrarily small, the last relation holds for every $v\in \mathbb{K}$ via a limiting process.
Applying \cite[III: Lemma 1.5]{Kinderlehrer2000} once more, this implies
\begin{align*}
0\leq\left\langle-\Delta \varphi^*,v-\varphi^*\right\rangle
-\int_\Omega g_1 (v-\varphi^*)dx,\ \forall v\in\mathbb{K}.
\end{align*}
Due to the uniqueness of the solution for our variational inequality problem, this yields the assertion.\end{proof}

\section{Proof of Lemma \ref{L:Beschr}}
\begin{proof}
	Testing~(\ref{L:Beschr.1})--(\ref{L:Beschr.3}) by $\tau\hat r $, $\hat p $ and $\hat q $, respectively, and summing up we get
	\begin{align*}
		&		\tau \langle h_ r ,\hat r \rangle + \langle h_ p ,\hat p \rangle + \langle h_ q ,\hat q \rangle										\\
		&=		\tau \langle \nabla \hat r ,\nabla \hat r \rangle +\tau\langle \hat A\hat r ,\hat r \rangle + \langle \hat m \nabla \hat p ,\nabla \hat p \rangle						\\&\hskip6mm 
					+	\frac 1 \tau \langle \hat\rho \hat q ,\hat q \rangle - \langle (D \hat q ) \hat u ,\hat q \rangle
					+	\langle 2\hat\eta \epsilon(\hat q ) ,\epsilon(\hat q )\rangle 													\\
		&\ge	\tau ||\hat r ||_{ \overline{H}^1(\Omega )}^2 + C\Big( ||\hat p ||_{ \overline{H}^1(\Omega )}^2 + ||\hat q ||_{ H^1_{0,\sigma}(\Omega ;\mathbb{R}^N )}^2 \Big)
	\end{align*}
	for a positive constant $C$ depending only on $\alpha $ and on the constants in Korn's and Poincar\'e's inequalities.
	This estimate yields the assertion.
\end{proof}

\bibliographystyle{siam}
\bibliography{CHNSdraft9}

\end{document}